\documentclass[12pt]{amsart}
\usepackage{amsfonts}
\usepackage{amscd}
\usepackage{amssymb}
\usepackage{tikz-cd}
\usetikzlibrary{arrows}
\usepackage{graphics,enumitem}
\allowdisplaybreaks
\usepackage{amsmath}

\usepackage{hyperref}
\hypersetup{colorlinks,citecolor=blue}

\newtheorem{theorem}{Theorem}[section]

\newtheorem{hypothesis}[theorem]{Theorem}

\newtheorem{claim}[theorem]{Claim}

\newtheorem{corollary}[theorem]{Corollary}

\newtheorem{example}[theorem]{Example}

\newtheorem{lemma}[theorem]{Lemma}

\newtheorem{proposition}[theorem]{Proposition}

\newcommand{\BZ}{{\mathbb{Z}}}
\newcommand{\BN}{{\mathbb{N}}}
\newcommand{\BR}{{\mathbb{R}}}

\newcommand{\BT}{{\mathbb{T}}}
\newcommand{\BH}{{\mathbb{H}}}

\newcommand{\Min}{\mathrm{Min}}

\newcommand{\gd}{\delta}

\newcommand{\gC}{\Gamma}
\newcommand{\gc}{\gamma}

\newcommand{\gep}{\epsilon}

\newcommand{\inj}{\mathrm{inj}}

\newcommand{\rank}{\text{rank}}

\newcommand{\vol}{\text{vol}}
\newcommand{\diam}{\text{diam}}

\newtheorem{prop}{Proposition}[section]

\theoremstyle{definition}

\newtheorem{rem}[prop]{Remark}

\newcommand\RR{{\mathcal{R}}}

\catcode`\@=11
\long\def\@savemarbox#1#2{\global\setbox#1\vtop{\hsize\marginparwidth 
  \@parboxrestore\tiny\raggedright #2}}
\catcode`\@=12

\title{Convergence of normalized Betti numbers in nonpositive curvature}

\author{Miklos Abert}
\address{MTA Renyi Institute, Realtanoda 13-15, 1053 Budapest, Hungary}
\email{abert@renyi.hu}

\author{Nicolas Bergeron}
\address{ENS, D\'epartement de Math\'ematiques et Applications, F-75005, Paris, France}
\email{nicolas.bergeron@ens.fr}

\author{Ian Biringer}
\address{Boston College, Department of Mathematics, 140 Commonwealth Ave, Chestnut Hill, MA 02467}
\email{ian.biringer@gmail.com}

\author{Tsachik Gelander}
\address{The Weizmann Institute of Science, Faculty of Mathematics and Computer Science, Rehovot 7610001, Israel}
\email{tsachik.gelander@gmail.com}

\begin{document}

\begin{abstract}
We study the convergence of volume-normalized Betti numbers in Benjamini-Schramm convergent sequences of non-positively curved manifolds with finite volume.  In particular, we show that if $X$ is an irreducible symmetric space of noncompact type, $X \neq \BH^3$, and $(M_n)$ is any Benjamini-Schramm convergent sequence of finite volume $X$-manifolds, then the normalized Betti numbers $b_k(M_n)/vol(M_n)$ converge for all $k$.

As a corollary, if $X$ has higher rank and $(M_n)$ is any sequence of distinct, finite volume $X$-manifolds, the normalized Betti numbers of $M_n$ converge to the $L^2$ Betti numbers of $X$. This extends our earlier work with Nikolov, Raimbault and Samet in \cite{Abertgrowth}, where we proved the same  convergence result for uniformly thick sequences of compact $X$-manifolds.
One of the novelties of the current work is that it applies to all quotients $M = \Gamma \backslash X$ where $\Gamma$ is arithmetic; in particular, it applies when $\Gamma$ is isotropic.

%In \cite{Abertgrowth} the convergence of normalized Betti numbers was proved only for uniformly thick sequences of manifolds (equivalently the associated lattices were uniformly discrete, and in particular uniform). The current work allows us to dismiss this assumption and in particular, in contrast to \cite{Abertgrowth}, our new results apply also to sequences of non-compact manifolds.
\end{abstract}
\maketitle

\tableofcontents

\section{Introduction}
\label{intro}
We begin with a fair amount of general motivation, mostly from Elek \cite{elek2010betti} and Bowen \cite{Bowencheeger}. The well-versed reader can skip ahead to \S \ref{results}  for the statements of our results.

\vspace{2mm}

The \emph {normalized Betti numbers} of a space $X$ are the quotients $$b_k(X)/\vol(X), \ \text{ where } \ b_k(X) := \dim H_k(X,\BR).$$ All spaces in this paper will be either Riemannian manifolds or simplicial complexes. In the latter case, volume should be interpreted as the number of vertices.

Fix $d>0$. A simplicial complex $K$ has \emph {degree at most $d$} if every vertex in $K$ is adjacent to at most $d$ edges. In \cite{elek2010betti}, Elek shows that the normalized Betti numbers of finite simplicial complex $K$ with degree at most $d$ are \emph{testable}, meaning that there is a way to read  off approximations of the normalized Betti numbers while only looking at bounded random samples of $K$. More precisely, given $\epsilon>0$, there is some $R(\epsilon)$ as follows. Given $K$, select $R$ vertices of $K$ at random and look at the $R$-neighborhood of each in $K$.  Testability means there is a way to guess from this data what the normalized Betti  numbers of $K$ are, that is correct up to an error of $\epsilon$ with probability $1-\epsilon$.

This is really a continuity result, in the following sense. Consider the topological space 
$$\mathcal K = \big \{ \text{connected, pointed finite degree simplicial complexes } (K,p) \big \} / \sim,$$
where each $p\in K$ is a vertex, two pointed complexes are equivalent if they are isomorphic via a map that takes basepoint to basepoint, and where two complexes are close if for large $R$, the $R$-balls around their basepoints are isomorphic. Each finite (even possibly disconnected) complex $K$ induces a finite measure $\mu_K$ on $\mathcal K$, defined by pushing forward the counting measure on the vertex set $V(K)$ under the map
$$V(K) \longrightarrow \mathcal K, \ \ p \mapsto [(K_p,p)],$$ 
where $K_p\subset K$ is the connected component of $p$.
One then says that a sequence $(K_n)$ in $ \mathcal K$ \emph{Benjamini-Schramm (BS) converges}\footnote{Benjamini-Schramm convergence of graphs was first studied in their paper \cite{benjamini2011recurrence}. See also Aldous--Lyons~\cite{Aldousprocesses} for a broader picture of BS-convergence in the case of  graphs.} if the  probability measures $\mu_{K_n}/\vol(K_n)$ weakly converge to some probability measure on $\mathcal K$. One can then reformulate the testability of normalized Betti numbers above as saying:

\begin{theorem}[{Elek \cite[Lemma 6.1]{elek2010betti}}]\label{elek}
If $(K_n)$ is a BS-convergent sequence of  finite, simplicial complexes, each with degree at most $d$, the normalized Betti numbers $b_k(K_n)/\vol(K_n)$ converge for all $k$.
\end{theorem} 

Informally, the relationship with testability is that if we fix $R>0$ and take $n,m >>0$, convergence says the measures associated to the two complexes $K_n,K_m$ will be close. So by the definition of the topology on $\mathcal K$,  we will have that for large $R$, the distribution of randomly sampled $R$-balls in $K_n$ will be almost the same as that in $K_m$, so having a way to accurately guess the normalized Betti numbers from these (nearly identical) data sets means that the normalized Betti numbers of $K_n$ and $K_m$  must be close.

\vspace{2mm}

Recently, a number of authors, see e.g.\ \cite{Abertgrowth,Abertunimodular,biringer2017ends,Bowencheeger,Namazidistributional}, have studied the analogous version of BS-convergence for Riemannian manifolds. Adopting the language of \cite{Abertunimodular}, set $$\mathcal M = \{\text{pointed Riemannian } \text{manifolds } (M,p)\} / \text{pointed isometry},$$ 
 endowed with the topology of pointed smooth convergence. See \S \ref{smoothtop}. Here and below,  Riemannian  manifolds are  always assumed to be  connected and complete. Really, all of the results below hold for disconnected manifolds, just as Theorem \ref{elek} applies to disconnected complexes, but it seems unnecessarily confusing to continue working in this generality.

A finite volume (connected, complete) Riemannian manifold $M$ induces a finite measure $\mu_M$ on $\mathcal M$, by pushing forward the Riemannian measure on $M$ via the map $p \mapsto [(M,p)]$, and we say  that a sequence $(M_n)$ \emph{Benjamini-Schramm (BS) converges} if the measures $\mu_{M_n}/\vol(M_n)$ weakly converge to some probability measure.  In full generality, the Riemannian analogue of Theorem \ref{elek} is not true, since if no geometric constraints are imposed, we can pack as much homology as desired into a part of a manifold with negligible volume. For example: connect sum a small volume genus $g(n)$ surface, say with volume $1$, somewhere on a round radius-$n$ sphere. The resulting  surfaces will BS-converge to an atomic measure on the single point $[(\BR^2,p)] \in \mathcal M$, where $p\in \BR^2$ is any basepoint. But by choosing $g(n)$ appropriately, we can make the first Betti numbers whatever we like.

 In the above example, the real problem is injectivity radius. For a Riemannian manifold $M$ and a point $x \in M$ we denote the injectivity radius of $M$ at $x$ by $\inj_M (x)$. Given $\epsilon>0$,  the {\it $\epsilon$-thick part} and the {\it $\epsilon$-thin part} of $M$ are
$$M_{\geq \epsilon} = \{ x \in M \; : \; \inj_M (x) \geq \epsilon /2 \} \mbox{ and } M_{< \epsilon} = M \setminus M_{\geq \epsilon}.$$ 

One says that $M$ is \emph{$\epsilon$-thick} if $M=M_{\geq \epsilon}$. Now, under geometric constraints like curvature bounds, there is a standard way to model an $\epsilon$-thick manifold $M$ by a simplicial complex $K(M)$ with comparable volume and bounded degree: one selects an $\epsilon$-net $S$ in $M$, and lets $N(S)$ be the nerve of the covering of $M$ by $\epsilon$-balls.  One can then show:
%Using an idea suggested by Elek, Bowen \cite{Bowencheeger} used this simplicialization procedure to show:

\begin{theorem}[Elek, Bowen + ABBG\footnote{By (ABBG) we refer to the current paper.}]\label{bowenthm}
	If $(M_n)$ is a BS-convergent sequence of  compact, $\epsilon $-thick Riemannian manifolds with upper and lower curvature bounds, then the normalized Betti numbers $b_k(M_n)/\vol(M_n)$  converge.
\end{theorem}

A word is in order about the attributions: it was originally conceived by Elek, and then written up and published by Bowen  \cite[Theorem 4.1]{Bowencheeger}, but this writeup was not complete, and we (ABBG) provide a  slightly different argument that avoids this gap in \S \ref{normbmm}. Briefly, the idea is to superimpose a bunch of Poisson processes on $M_n$, discarding points that are too close together, until enough points are laid down so that the nerve complex $N_n$ associated to a collection of balls around these points sees the Betti numbers of $M_n$ up to a small error.  One then proves that the constructed sequence of (random) nerve complexes BS-converges. By (a slight generalization of) Theorem \ref{elek} above, the expected normalized Betti numbers $E[b_k(N_n))/\vol(N_n)]$ will converge, from which one can deduce convergence of the normalized Betti numbers $b_k(M_n)/vol(M_n)$.

\medskip

Theorem \ref{bowenthm} is really a special case of a more general result, see \S \ref{normbmm}. 
Indeed, the essence of the current work is that we deal with general manifolds with no assumptions on the injectivity radius. The thick part $M_{\geq \epsilon}$ is then a proper submanifold with boundary and we rely on Gelander's techniques \cite{Gelanderhomotopy} in order to associate a random simplicial compex to the thick part. As shown in \cite{bader2016homology} the boundary of the thick part corresponds to a sub-simplicial complex. This allows as to consider the thick and the thin parts separately.

\subsection{Main results}\label{results} 

Our interest in this paper is whether for certain manifolds of nonpositive curvature, one can control the thin parts well enough so that BS-convergence implies convergence of normalized Betti numbers, without any assumption of thickness.

Although almost all of the real work in this paper is done more generally,  we start as follows. Let $X$ be an irreducible symmetric space of noncompact type. An \emph{$X$-manifold} is a complete Riemannian manifold whose  universal cover is isometric to $X$.

\begin{theorem}\label{mainsym}
	 Suppose that $\dim(X)\neq 3$ and $(M_n)$ is a BS-convergent sequence of finite volume $X$-manifolds.  Then for all $k$, the sequence $b_k(M_n)/\vol(M_n)$  converges.
\end{theorem}

Here, the only three-dimensional  irreducible symmetric spaces of noncompact type are scales of $\BH^3$. In fact, the conclusion of Theorem \ref{mainsym} is  false when $X=\BH^3$. As an example, let $K \subset S^3$ be a knot such that the complement $M = S^3 \setminus K$ admits a hyperbolic metric, e.g.\ the figure-8 knot. Using meridian--longitude coordinates, let $M_n$ be obtained by Dehn filling $M$ with slope $(1,n)$;  then each $M_n$ is a homology $3$-sphere. The manifolds $M_n \to M$ geometrically, see \cite[Ch E.6]{Benedettilectures}, so the measures $\mu_{M_n}$ weakly converge to $\mu_M$ (c.f.\ \cite[Lemma 6.4]{bader2016homology}) and the volumes $\vol(M_n) \to \vol(M)$. However, $b_1(M_n)=0$ while $b_1(M) =1$, so the normalized Betti numbers of the BS-convergent sequence $M_1,M,M_2,M,\ldots$ do not converge. See also Example~\ref{3dimex} for a similar counterexample in which volume goes to infinity. In fact, there is a real sense in which the \emph{only} counterexamples come from Dehn filling. See \S \ref{pncsec}. %See also Example~\ref{3dimex},  The idea is that one can construct sequences of finite volume hyperbolic $3$-manifolds $(M_n)$ in which the number of cusps in $M_n$ is proportional to $\vol(M_n)$, and that Dehn filling these cusps does not alter a weak limit, but can drastically change the first Betti number.

To illustrate a special case of Theorem \ref{mainsym}, let's say that $(M_n)$ \emph{BS-converges to $X$} when the measures $\mu_{M_n}$ weakly converge to   the atomic probability measure on the point $$[(X,x)] \in \mathcal M,$$ where $x\in X$  is any basepoint. Now any $X$ as above admits a (compact, even) $X$-manifold $M$, by a theorem of Borel \cite[Theorem 14.1]{Raghunathandiscrete}. A theorem of Mal'cev \cite{mal1983homomorphisms} says that $\pi_1 M$ is residually finite. So, we can take a tower of regular covers $$\cdots \to  M_2 \to M_1 \to M$$ corresponding to a nested sequence of normal subgroups of $\pi_1 M$  with trivial intersection, and such a sequence $(M_n) $ will BS-converge to $X$, see \cite{Abertgrowth} for details. Moreover, if $M$ is compact then DeGeorge--Wallach \cite{DeGeorgelimit} showed that the normalized Betti numbers of $(M_n)$ converge to the \emph{$L^2$-Betti numbers} $b_k^{(2)}(X)$ of $X$. See \cite{Abertgrowth,Luckapproximating} for more information about $L^2$-Betti numbers, and for a more general result. 

In fact, any sequence of manifolds that BS-converges to $X$ can be interleaved with a tower of covers of a compact $X$-manifold as in the example above, and the result still BS-converges. So, Theorem \ref{mainsym} and the result of DeGeorge-Wallach \cite{DeGeorgelimit} above give:

\begin{corollary}\label {l2cor}
	 Suppose that $(M_n)$  is a sequence of finite volume $X$-manifolds that BS-converges to $X$. Then for all $k\in \BN$, we have
$b_k(M_n) / \vol(M_n) \to b_k^{(2)}(X).$
\end{corollary}

With Nikolov, Raimbault and Samet, we proved this in \cite{Abertgrowth} for sequences of compact, $\epsilon$-thick manifolds, using analytic methods. One could also prove it in the thick case by using  Theorem \ref{bowenthm} above (the Bowen--Elek simplicial approximation technique) and interleaving with a covering tower. In the thin case, we were able to push our analytic methods far enough to give a proof for $X=\BH^d$, see \cite[Theorem 1.8]{Abertgrowth}. Hence,  there is no problem in allowing $X=\BH^3$ in Corollary \ref{l2cor}, even though Theorem \ref{mainsym}  does not apply. 

While we were finishing this paper, Alessandro Carderi sent us an interesting preprint where, among other things, he proves the same result as Corollary \ref{l2cor} if either $k=1$, or $k$ is arbitrary and the symmetric space $X=G/K$ is of higher rank and $M_n$ is non compact, or in most cases when $X$ is of rank $1$. His proof is quite different, he considers the ultraproduct of the sequence of actions of $G$ on $G/\Gamma_n$. He then identifies the $L^2$-Betti numbers of the resulting $G$-action with the $L^2$-Betti numbers of the group $G$. 

Corollary \ref{l2cor} is particularly powerful when $X$ has real rank at least two. In this case, we proved with Nikolov, Raimbault and Samet that \emph{any}  sequence of distinct finite volume $X$-manifolds BS-converges to $X$, see \cite[Theorem 4.4]{Abertgrowth}. So, Corollary \ref{l2cor} implies:

\begin{corollary}\label {hrcor}
	 Suppose that $\rank_\BR X \geq 2$ and $(M_n)$  is any sequence of distinct finite volume $X$-manifolds. Then for all $k\in \BN$, we have
$b_k(M_n) / \vol(M_n) \to b_k^{(2)}(X).$
\end{corollary}

In the two corollaries above, we can identify the limit of the normalized Betti numbers when the BS-limit is $X$. In general, one can think of Theorem \ref{mainsym} as giving a definition of `$L^2$-Betti numbers' for arbitrary limits of BS-convergent sequences. The measures on $\mathcal M$ that arise as such limits have a special property called \emph{unimodularity}, see \cite{Abertunimodular}, and it would be interesting to find a good intrinsic definition of the `$L^2$-Betti numbers' of a unimodular measure that is compatible with Theorem \ref{mainsym}. %In the dual algebraic language of \emph{invariant random subgroups}, we do this for measures supported on $\epsilon$-thick locally symmetric spaces in \cite{Abertgrowth}, but such a theory should be available in much greater generality.

\subsection{The proof, and generalities in nonpositive curvature} To prove  Theorem \ref{mainsym}, we split into cases depending on $\rank_\BR X$. When the rank is one, we need to deal with general BS-convergent sequences, but the thin parts of rank one locally symmetric spaces are easy to understand. And when the rank is at least two, the only possible BS-limit we need to consider is $X$.  We now give two theorems that handle these two cases. We state them very generally, without any assumption of symmetry.

\begin{theorem}[Pinched negative curvature, arbitrary BS-limits]\label {pnc}
Let $(M_n )$ be a BS-convergent sequence of  finite volume Riemannian $d$-manifolds, with $d\neq 3$, and  with sectional curvatures in the interval $[-1,\delta]$, for some $-1 \leq \delta < 0$. Then the normalized Betti numbers $b_k(M_n)/\vol(M_n)$  converge for all $k$.
\end{theorem}

\begin{theorem}[Nonpositive curvature, with a thick BS-limit]\label{npc}
Let $\epsilon>0$ and let $(M_n)$  be a sequence of real analytic, finite volume Riemannian $d$-manifolds with sectional curvatures in the interval $[-1,0]$, and assume the universal covers of the $M_n$ do not have Euclidean de Rham-factors. If $(M_n)$ BS-converges to a  measure $\mu$ on $\mathcal M$  that is supported on $\epsilon$-thick  manifolds, the normalized Betti numbers $b_k(M_n)/\vol(M_n)$  converge for all $k$.
\end{theorem}

Let's see how to deduce  Theorem \ref{mainsym} from these results. Suppose $X$ is an irreducible symmetric space of noncompact type, $\dim(X)\neq 3$. When $X$ has rank one, $X$ has pinched negative curvature, so therefore Theorem \ref{mainsym} follows from Theorem \ref{pnc}. When $X$ has higher rank, \cite[Theorem 4.4]{Abertgrowth} says that any BS-convergent sequence $(M_n)$ of $X$-manifolds BS-converges to $X$, as mentioned above. Since $X$ is actually $\epsilon$-thick for any $\epsilon$, Theorem~\ref{npc} applies, and Theorem \ref{mainsym} follows.

\vspace{2mm}

The reader may wonder where we use $d\neq 3$ in the proof of Theorem \ref{pnc}. When $d=2$, one can deduce the claim from Gauss--Bonnet. In general, the point is that the boundary of a Margulis tube is homeomorphic to an $S^{n-2}$-bundle over $S^1$.  When $d\geq 4$, this bundle is not aspherical, so it can be distinguished from a cusp cross section, which prevents one from doing Dehn filling as in our problematic $3$-dimensional example. More to the point, one can show that when $d\geq 4$, Margulis tubes  with very short cores have boundaries with large volume, see Proposition \ref{shortgeodesics}, which implies that the number of Margulis tubes with short cores one can see in a manifold is sublinear in volume. Hence, the contribution of the tubes to homology cannot affect the normalized Betti numbers much.

The key to Theorem \ref{npc} is  a celebrated theorem of Gromov, see \cite[Theorem 2]{Ballmannmanifolds}, that bounds the Betti numbers of an analytic manifold with  sectional curvatures in $[-1,0]$ and no local Euclidean deRham factors linearly in terms of its volume. Delving into its proof, one can show that in the setting of Theorem \ref{npc}, the Betti numbers of the thin parts of the $M_n$ grow sublinearly with $\vol(M_n)$. One can then combine the proof of Theorem \ref{bowenthm} (the Bowen--Elek simplicial approximation argument), which handles the thick parts of the $M_n$, making use of the techniques from \cite{Gelanderhomotopy} and \cite{bader2016homology} to control the complexity of the boundary, where the thick and thin parts are glued, with Mayer--Vietoris sequence to get Theorem ~\ref{npc}.

\begin{rem}
Recently, the work \cite{Abertgrowth} has been extended by Gelander and Levit to analytic groups over non-archimedean local fields 
\cite{gelander-levitIRS}. For non-archimedean local fields of characteristic $0$ the uniform discreteness assumption holds automatically for the family of all lattices and more generally all discrete IRS. However this is not the case in positive characteristic. We conjecture that the analogue of the stronger results concerning Betti numbers obtained in the current work can be extended to general analytic groups over non-archimedean local fields.
%, but that some new techniques are required.
\end{rem}

\subsection{Acknowledgments}
M.A. was supported by the ERC Consolidator Grant 648017 and the NKFIH grant K109684. I.B. was partially supported by NSF grant DMS-1611851 and CAREER Award DMS-1654114. T.G was partially supported by ISF-Moked grant
2095/15. We would like to thank the referees for their careful reading of the paper and their many helpful suggestions.

\section{Spaces of spaces and simplicial approximation}
\label {simpsec}
In this section, we discuss the topology on $\mathcal M$ and a similar topology on the space $\mathbb M$  of all pointed metric measure spaces. We then state and prove a generalization of the Bowen--Elek theorem on the convergence of Betti numbers of thick spaces,  which was stated in a weak form in the introduction as Theorem \ref{bowenthm}.

\subsection{The smooth topology}\label{smoothtop}
In the introduction, we introduce the space
$$\mathcal M = \{\text{pointed, connected, complete Riemannian } \text{manifolds } (M,p)\} / \text{pointed isometry},$$ 
 endowed with the topology of pointed smooth convergence.  Here, a sequence $(M_n,p_n)$ converges \emph{smoothly} to $(M_\infty,p_\infty)$ if there  is a sequence of smooth embeddings
\begin{equation} \phi_n:B_{M_\infty}(p_\infty, R_n)\longrightarrow M_n \label {fig}\end {equation} with $R_n \rightarrow \infty $ and $ \phi_n (p_\infty) = p_n $,  such that $\phi_n ^*g_n\rightarrow g_\infty $ in the $C ^\infty $-topology, where $g_n $ are the Riemannian metrics on $M_n $.  We call $(\phi_n)$ a \emph{sequence of almost isometric maps  coming from smooth convergence}. Note that each  metric $\phi_n ^*g_n$  is only partially defined on $M_\infty$, but their  domains of definition exhaust $M_\infty $, so  it still makes sense to say that  $\phi_n ^*g_n\rightarrow g_\infty $  on all of $M_\infty $, even if the language is a bit abusive. 
\'Alvarez L\'opez, Barral Lij\'o and Candel \cite{Alvarezuniversal}  have shown that $\mathcal M$, with the  smooth topology, is a Polish space. See also the appendix of Abert-Biringer \cite{Abertunimodular} for a slightly simpler proof.

\subsection{Metric measure spaces}\label {mmsec}A \emph {metric measure space} (or \emph{mm}-space) is  a proper, separable metric space $M$ equipped with a Radon measure $\vol$. Let $$\mathbb M = \{\text {pointed mm-spaces } (M,\vol,p) \} / \text{pointed measure preserving isometry}.$$

 Following Bowen \cite[Definitions 28 and 29]{Bowencheeger}, an \emph{$(\epsilon,R)$-relation} between  pointed mm-spaces $\mathfrak M_1 = (M_1,\vol_1,p_1)$ and $\mathfrak M_2 = (M_2,\vol_2,p_2)$ is a pair of isometric embeddings $$ M_i \longrightarrow Z, \ \ i=1,2$$ into some common metric space $Z$  having the following properties:
\begin {enumerate}[label=(\alph*)]
	\item $d_Z( p_1,p_2) < \epsilon$,
\item $ B_{M_1}(p_1,R) \subset (M_2)_\epsilon$ and $B_{M_2}(p_,R) \subset ( M_1)_\epsilon$,
\item  for all Borel subsets $F_i \subset {B_{M_i}(p_i,R)}$, we have $$\vol_1 (F_1) < (1+\epsilon)\vol_2( \, (F_1)_\epsilon \, ) + \epsilon, \ \ \vol_2 (F_2) < (1+\epsilon)\vol_1( \, (F_2)_\epsilon \, ) + \epsilon.$$
\end {enumerate}
Here, if $F$ is a subset of a metric space, the notation $(F)_\epsilon$ refers to the $\epsilon$-neighborhood of $F$. See also \S \ref{thickapproxsec}. The multiplicative factors of $(1+\epsilon)$ in (c) are not really necessary, and are not present in \cite{Bowencheeger}. However, some of our statements, e.g.\ Lemma \ref{lipschitz Lemma} below, are simpler because of them.

For each $\mathfrak M = (M,\vol,p) \in \mathbb M$ and $\epsilon,R>0$, define the \emph {$(\epsilon,R)$-neighborhood} of $\mathfrak M$ to be the set $\mathcal N_{\epsilon,R}(\mathfrak M)$ of all $\mathfrak M' \in \mathbb M$ that are $(\epsilon',R')$-related to $\mathfrak M$ for some $\epsilon'<\epsilon$ and $R' > R$. Note that if $\mathfrak M' \in \mathcal N_{\epsilon,R}(\mathfrak M)$, then for all sufficiently small $\delta>0$ and large $r>0$, we have
\begin{equation}
\label{inside!}\mathcal N_{\delta,r}(\mathfrak M') \subset \mathcal N_{\epsilon,R}(\mathfrak M).
\end{equation}
This follows from the fact that one can `concatenate' a relation between $\mathfrak M_1$ and $\mathfrak M_2$ with one between $\mathfrak M_2$ and $\mathfrak M_3$, by gluing the two metric spaces $Z$ together along $M_2$.

If we endow $\mathbb M$  with the topology generated by all $(\epsilon,R)$-neighborhoods, then the neighborhood nesting property referenced in \eqref{inside!} implies that 
$$\mathfrak M_i \to \mathfrak M_\infty \iff \exists \epsilon_i\to 0, R_i\to \infty \text{ such that } \mathfrak M_i \text{ is } (\epsilon_i,R_i)\text{-related to } \mathfrak M_\infty.$$ 
The next lemma will help us relate smooth convergence of Riemannian manifolds to their convergence as metric measure spaces.

\begin {lemma}\label {lipschitz Lemma}
Suppose that $(M_i,p_i)$, $i=1,2 $, are  pointed Riemannian $d$-manifolds and for some $R>0$ there is an  embedding $\phi: B_{M_1}(p_1,R) \longrightarrow M_2$  with $ \phi(p_1)=p_2$ and \begin{equation}
 	(1-\delta)|v|\leq  |d\phi(v)| \leq (1+\delta)|v|, \ \ \forall v \in TB_{M_1}(p_1,R) . \label {lipschitz}
 \end{equation}
Then if $\delta=\delta( \epsilon,d)$ is small,  the triples $(M_i,\vol_i,p_i)$ are $(\epsilon,R)$-related, where  $\vol_i$  is the Riemannian measure on $M_i$.
\end {lemma}

\begin {proof}
Take $\delta<\epsilon$ and let $\phi$  be as in the statement of the lemma. We want to produce an $(\epsilon,R)$ relation between $M_1$ and $M_2$. Define the common space $Z$ as the disjoint union  $$Z = M_1 \sqcup M_2,$$ endowed with a metric that restricts to the given metrics on $M_1,M_2$, and where for $x\in M_1,y\in M_2$,
$$
 d(x,y) = \inf \{d(x,x')+\delta + d(\phi(x'),y) \ | \ x' \in B_{M_1}(p_1,R+1) \}.
$$ 
We now verify that $Z$ gives an $(\epsilon,R)$-relation. First, $d_Z(p_1,p_2)=\delta<\epsilon$. Second, if $x\in M_1 \cap B_Z(p_1,R) = B_{M_1}(p_1,R)$, then $d(x,\phi(x))= \delta<\epsilon$, so $x \in (M_2)_\epsilon$. Third, if $F_1 \subset B_Z(p_1,R)$ is a Borel subset, then we have $$\vol_1(F_1) =\vol_1(F_1 \cap M_1) \leq (1+\delta)^{d} \vol_2 (\phi(F_1 \cap M_1)) \leq (1+\delta)^{d} \vol_2 ( \, (F_1)_\delta \, ),$$
where the first inequality follows from \eqref{lipschitz}, and the second follows from the fact that $d(x,\phi(x))=\delta$. So, as long as $\delta$ is small, the right side will be at most $(1+\epsilon)\vol_2( (F_1)_\epsilon )$. The two remaining parts of  properties (a) and (b) follow similarly.
\end {proof}

 As an immediate corollary, we get the following:

\begin{corollary}
	\label {cnts}	The natural inclusion $\mathcal M \longrightarrow \mathbb M$ from the space of pointed Riemannian manifolds (with the smooth topology) to the space of  pointed mm-spaces is continuous.
\end{corollary}

\subsubsection{Extended mm-spaces}

 We will need a  slight variant of $\mathbb M$  for our work below. Let
$$\mathbb {M}^{ext} = \{(M,\vol,p,E) \ | \ (M,\vol,p) \in \mathbb{M}, \  E \supset M \text{ a super-metric space}\}/ \sim,$$
 where a super-metric space is just a proper, separable metric space that contains $M$ as a submetric space. We call  a quadruple  $(M,\vol,p,E) $ an \emph {extended}  pointed mm-space; two quadruples are identified in $\mathbb {M}^{ext}$ if there is a pointed isometry between the super-metric spaces $E$ that restricts to a measure preserving isometry from one mm-space $M$ to the other.  The topology on $\mathbb {M}^{ext}$ is similar to that on $\mathbb M$: we say that $(M_i,\vol_i,p_i,E_i)$, $i=1,2$, are \emph{$(\epsilon,R)$-related} if there are isometric embeddings $$E_i \longrightarrow Z, \ \ i=1,2$$
that restrict to give an $(\epsilon,R)$-relation  between the triples $(M_i,\vol_i,p_i)$, and where also
\begin {equation}
	 B_{E_1}(p_1,R) \subset (E_2)_\epsilon, \ \ B_{E_2}(p_2,R) \subset ( E_1)_\epsilon. \label {hausE}
\end {equation}
One then defines $(\epsilon,R)$-neighborhoods just as before and the topology on $\mathbb {M}^{ext}$ is that generated by these neighborhoods, in which $\mathfrak M_i \to \mathfrak M_\infty$ if and only if there are $\epsilon_i\to 0$ and $R_i\to \infty$ such that $M_i$ is $(\epsilon_i,R_i)$-related to $M_\infty$ for all $i$.
 
We then have the following variant of Lemma \ref{lipschitz Lemma}.

\begin {lemma}\label {lipschitz Lemma2}
Suppose that $(M_i,p_i)$, $i=1,2 $, are  pointed Riemannian $d$-manifolds with  distinguished subsets $T_i \subset M_i$ and that for some $R>0$ there is an  embedding $$\phi: B_{M_1}(p_1,R) \longrightarrow M_2$$  with $ \phi(p_1)=p_2$  that satisfies the following three properties:
\begin {enumerate}
	\item[(i)]  	$(1-\delta)|v|\leq  |d\phi(v)| \leq (1+\delta)|v|, \ \ \forall v \in TB_{M_1}(p_1,R) .$
\item[(ii)] $\phi^{-1}(T_2) \subset (T_1)_\delta,$ and $\phi(T_1 \cap B_{M_1}(p_1,R)) \subset (T_2)_\delta$,
\item[(iii)] $\vol_1 (\phi^{-1}(T_2) \, \triangle \, T_1)<\delta,$
\end {enumerate}
where  $\triangle$  is the symmetric difference.
Then if $\delta=\delta(\epsilon,d)$ is  sufficiently  small, the  quadruples $(T_i,\vol_i |
_{T_i},p_i,M_i)$ are $(\epsilon,R)$-related, where  here $\vol_i$  is the Riemannian measure on $M_i$.
\end {lemma}
\begin {proof}
 The proof is similar to that of Lemma \ref{lipschitz Lemma}. With $Z=M_1 \sqcup M_2$ and $d$  the metric defined in Lemma \ref{lipschitz Lemma}, equation \eqref{hausE} above follows exactly as before  as long as $\delta<\epsilon$.   So, we just need to verify that $Z$ gives an $(\epsilon,R) $-relation between the subsets $T_1,T_2$.   Property (a)  is immediate from the definition of the metric on $Z$. For (b), if $x\in B_{T_1}(p_1,R_1)$ then $\phi(x) \subset (T_2)_\delta$, so $d_Z(x,T_2) < 2\delta$. So, (b) holds if $\delta\leq \epsilon/2$, as the proof of the other part is similar. For (c),  suppose $F \subset B_Z(p_1,R)$ is  Borel. Then 
\begin {align*}
\vol_1|_{T_1}(F_1) &= \vol_{1} (F_1 \cap T_1) \\
&\leq  \vol_1(F_1 \cap \phi^{-1}(T_2)) + \vol_1(\phi^{-1}(T_2) \, \triangle \, T_1) \\
&<  (1+\delta)^{d}\vol_{2}( \phi(F_1) \cap T_2))+ \delta\\
 &= (1+\delta)^{d}\vol_{2}|_{T_2}( \phi(F_1))+ \delta 
\end {align*}
So since $\phi(F_1)\subset (F_1)_\delta$, (c) holds if $(1+\delta)^d \leq (1+\epsilon)$. The other part of (c) is similar.
\end {proof}

\subsubsection{Extended mm-spaces with multiple measures or distinguished subsets} Let 
\[\mathbb M^{n,ext} := \{(M,\vol_1,\vol_2,\ldots,\vol_n,p,E)\} / \sim,\]
where here $(M,p)$ is a pointed metric space embedded in some super metric space $E \supset M$, the $\vol_i$ are Radon measures on $M$, and the equivalence relation is pointed isometry that preserves all the measures. The space $\mathbb M^{n,ext}$ comes equipped with projection maps 
$$\pi_i : \mathbb M^{ext,n} \longrightarrow \mathbb M^{ext},  \ \  (M,\vol_1,\vol_2,\ldots,\vol_n,p,E) \longmapsto (M,\vol_i,p,E)$$
for each $i=1,\ldots n$, and we say that two tuples $$\mathfrak M = (M,\vol_1,\vol_2,\ldots,\vol_n,p,E), \ \  \mathfrak M' = (M',\vol_1',\vol_2',\ldots,\vol_n',p',E)$$ are $(\epsilon,R)$-related if there are fixed embeddings $E \hookrightarrow Z, E' \hookrightarrow Z$ of the two super metric spaces into some common metric space $Z$ that induce  $(\epsilon,R)$-relations between the projections $\pi_i(\mathfrak M), \pi_i(\mathfrak M')$ for all $i$. The \emph{$(\epsilon,R)$-neighborhood} $\mathcal N_{\epsilon,R}(\mathfrak M)$ of $\mathfrak M \in \mathbb M^{ext,n}$ is again defined to be the set of all  $\mathfrak M'$ that are $(\epsilon',R')$-related to $\mathfrak M$ for some $\epsilon'<\epsilon $ and $R' >R$, and we endow $\mathfrak M^{ext,n}$ with the topology generated by these neighborhoods, in which $\mathfrak M_i \to \mathfrak M$ when there are $\epsilon_i\to 0$ and $R_i\to \infty$ such that $\mathfrak M_i, \mathfrak M$ are $(\epsilon_i,R_i)$-related for large $i $. %Note that this is not the coarsest topology such that all the projection maps $\pi_i$ are continuous: indeed, if $\mu_x$ denotes an atomic measure at $x$, in the latter topology the points $(\BR,\mu_{1}, \mu_{-1},0,\BR)$ and $(\BR,\mu_{1},\mu_{1},0,\BR)$ are indistinguishable.

We also consider the space
\[\mathbb {MS}^{ext} := \{(M,\vol,p,E,S)\} / \sim,\]
of pointed, extended mm-spaces equipped with locally finite subsets $S \subset M$. The topology is defined so that the natural map $\mathbb {MF}^{ext} \longrightarrow \mathbb M^{2,ext}$ that interprets a locally finite set $S$ as the atomic Radon measure $1_S$ is a homeomorphism onto its image. Finally, we let
\[\mathbb {MF}^{ext} := \{(M,\vol,p,E,S,f)\} / \sim,\]
 be the space of pointed, extended mm-spaces equipped with locally finite subsets $S$ that come weighted with functions $f : S \longrightarrow [0,1]$. We topologize $\mathbb {MF}^{ext}$ so that the natural map $\mathbb {MF}^{ext} \longrightarrow \mathbb M^{3,ext}$ is a homeomorphism onto its image; here, the three measures on the image of $(M,\vol,p,E,S,f)$ are $\vol$, the atomic Radon measure $1_S$ determined by $S$, and the atomic Radon measure $1_f$ where points $s\in S$ have mass $f(s)$ instead of unit weight. Note that the natural projection $\mathbb {MF}^{ext} \longrightarrow \mathbb {MS}^{ext}$ is continuous, and that there is also an embedding $\mathbb {MS}^{ext} \longrightarrow \mathbb {MF}^{ext}$ obtained by letting $f$ be the constant function whose values are all $1$. With this embedding in mind, we state most results below just for $ \mathbb {MF}^{ext}$, knowing that they also apply to the subspace $\mathbb {MS}^{ext}$. Finally, an \emph{$(\epsilon,R)$-relation} between two elements of $\mathbb {MS}^{ext}$, or between two elements of $\mathbb {MF}^{ext},$ is just an $(\epsilon,R)$-relation between their images in $\mathbb M^{2,ext}$, or in $\mathbb M^{3,ext}$. The topologies on $\mathbb {MS}^{ext}$ and $\mathbb {MF}^{ext}$ can then also be described via these relations, just as above.
 
 \medskip
 
 One difficulty that arises when working with $(\epsilon,R)$-relations is that you have a different pair of embeddings for each relation. In order to work with probability measures on sets of pointed mm-spaces, it is more convenient to have all our spaces be subsets of a fixed metric space. So, let $Z$ be some proper separable metric space. A pointed, extended mm-space, possibly with a distinguished discrete set and a function, is \emph{embedded in $Z$} if the extended space $E$ is a subset of $Z$, and we write the associated spaces of such spaces as
 $$\mathbb M^{ext}(Z), \ \mathbb {MS}^{ext}(Z), \ \mathbb {MF}^{ext}(Z).$$ We say that two spaces are \emph{$(\epsilon,R)$-related within $Z$} if their inclusions into $Z$ induce an $(\epsilon,R)$-relation, and we equip the spaces of spaces above with the topologies generated by $(\epsilon,R)$-relations within $Z$. In particular, we say that \emph{$\mathfrak M_i \to \mathfrak M_\infty$ within $Z$} if for any $(\epsilon,R)$ we have that for large $i$, $\mathfrak M_i,\mathfrak M_\infty$ are $(\epsilon,R)$-related within $Z$. 
 
A sequence of Radon measures $\mu_i$ on $Z$ \emph{weak* converges} to $\mu_\infty$ if $ \int f \, d\mu_i \to \int f \, d\mu_\infty$ for all continuous functions $f : Z \longrightarrow \BR$ with compact support\footnote{In this paper weak* convergence involves integrating against continuous functions with compact support, while weak convergence integrates against bounded continuous functions. Bowen uses weak* convergence in \cite{Bowencheeger} when defining Benjamini-Schramm convergence on $\mathbb M$, but it really should be weak convergence. Indeed, $\mathbb M$ is not locally compact at any point, so there are no nonzero continuous functions with compact support on $\mathbb M$.}.  When a sequence of mm-spaces with weighted subsets is embedded in a single $Z$, convergence of the weighted subsets can be interpreted as weak* convergence. 
 
 \begin{lemma}[c.f.\ Lemma A.2 of \cite{Bowencheeger}]\label{weakrelations}
Suppose that $\mathfrak M_i = (M_i,\vol_i,p_i,E_i,S_i,f_i)  \in \mathbb {MF}^{ext}(Z)$, where $i=1,2,\ldots,\infty$. Then $\mathfrak M_i \to \mathfrak M_\infty$ within $Z$ if and only if the embedded extended pointed mm-spaces $(M_i,\vol_i,p_i,E_i) \to (M_\infty,\vol_\infty,p_\infty,E_\infty)$ and the measures $1_{S_i}$ and $1_{f_i}$ converge in the weak* topology to $1_{S_\infty}$ and $1_{f_\infty}$.
 \end{lemma}

In fact, every convergent sequence in $\mathbb {MF}^{ext}$ can be embedded in some $Z$.

\begin{lemma}\label{Zspace}
Suppose that $\mathfrak M_i = (M_i,\vol_i,p_i,E_i,S_i,f_i)  \in \mathbb {MF}^{ext}$, where $i=1,2,\ldots,\infty$, and $\mathfrak M_i \to \mathfrak M_\infty$. Then there is a proper, separable metric on $$Z = \bigsqcup_{i=1,2,\ldots,\infty} E_i$$
such that $\mathfrak M_i \to \mathfrak M_\infty$ within $Z$. Furthermore, we can assume that for all $i,j$, $$d_Z(E_i,E_j) \geq 1/i + 1/j.$$\end{lemma}

Note that this lemma also applies to sequences of extended mm-spaces without weighted subsets, just by taking $S_i = \emptyset$. The proof is a modification of Lemma B.2 in \cite{Bowencheeger}.

\begin{proof}
	For each $i$, choose an $(\epsilon_i,R_i)$-relation between $\mathfrak M_i$ and $\mathfrak M_\infty$, where $\epsilon_i\to 0$ and $R_i\to \infty$. Instead of writing this relation as a pair of embeddings of $E_i,E_\infty$ into some third metric space, we can consider it as a pseudometric on the disjoint union $E_i \sqcup E_\infty$ that restricts to the original metrics on $E_i$ and $E_\infty$. We can then change each such pseudometric into a metric $d$ by adding $1/i$ to the distance between any point in $E_i$ and any point in $E_\infty$, and combine all of them into a single partially defined metric $d$ on the disjoint union
	$$Z = \bigsqcup_{i \in \BN\cup \{\infty\}} E_i,$$
	which we extend to a (fully defined) metric $d$ by setting $$d(x_i,x_j) := \inf_{x_\infty \in E_\infty} d(x_i,x_\infty)+d(x_j,x_\infty)$$
	for all finite $i,j$ and $x_i\in E_i, x_j\in E_j$. The reader can verify the desired properties.
\end{proof}

In order to talk about convergence of measures on $\mathbb {MF}^{ext}(Z)$ we will need an explicit basis of neighborhoods. Of course, one could just take the sets $\mathcal N_{\epsilon,R}(\mathfrak M)$ of all $\mathfrak M'$ that are $(\epsilon',R')$-related to $\mathfrak M$ within $Z$, but then it is a little unclear exactly what condition this places on the weighted discrete subsets. The following system of neighborhoods is more convenient in that respect.

\begin{lemma}\label{bijectionlem}
Suppose that $\mathfrak M_1 := (M_1,\vol_1,p_1,E_1,S_1,f_1)  \in \mathbb {MF}^{ext}(Z)$. For $\epsilon,R>0$, let $\mathcal B_{\epsilon,R} \subset \mathbb {MF}^{ext}$ be the set of all $\mathfrak M_2 := (M_2,\vol_2,p_2,E_2,S_2,f_2) \in \mathbb {MF}^{ext}(Z)$ such that
\begin{itemize}
	\item the pointed extended mm-spaces $(M_i,p_i,\vol_i,E_i)$ are $(\epsilon',R')$-related in $Z$, for some $\epsilon'<\epsilon$ and $R<R'$, and where
	\item there is a bijection $$ \phi : S_1 \cap B_{M_1}(p_1,R) \longrightarrow S_2 \cap B^\bullet_{M_2}(p_2,R)$$ 
	such that $d_Z(s,\phi(s))<\epsilon$ and $|f_1(s) - f_2(\phi(s))| < \epsilon$ for all $s \in S_1 \cap B_{M_1}(p_1,R)$.
\end{itemize}
Then there is a family of `admissible' pairs $(\epsilon,R)$ such that the sets $\mathcal B_{\epsilon,R}$ form a basis of open neighborhoods of  $\mathfrak M_1 \in \mathbb {MF}^{ext}(Z)$. Moreover, for every $R_0$, there is some $R>R_0$ such that $(\epsilon,R)$ is admissible for all sufficiently small $\epsilon$.
\end{lemma}

 Here, $B^\bullet_*(*,*)$ denotes the \emph{closed} ball of the given center and radius, while $B_*(*,*)$ is the open ball. A pair $(\epsilon,R)$ is \emph{admissible} if the following conditions hold:
 \begin{enumerate}
\item $d(s,t) > 3\epsilon$ for all $s,t\in S_1 \cap B(p_1, R)$, and
\item there are no points $s\in S_1$ with $d(p_1,x) \in (R-2\epsilon,R+2\epsilon).$
 \end{enumerate}
Since $S_1$ is locally finite, for any given $R$ condition (1) holds whenever $\epsilon$ is sufficiently small, and if we perturb $R$ so that there are no $s\in S_1$ with $d(p_1,s)=R$ and then shrink $\epsilon $ further, we can ensure that (2) holds as well. This justifies the last line of the lemma.

Note also that if we drop the condition on $f_1,f_2$ from $\phi$, then the above gives a description of a neighborhood basis for a point in $\mathbb {MS}^{ext}(Z)$ rather than in $\mathbb {MF}^{ext}(Z)$.

 \begin{proof}
 	Suppose $(\epsilon,R)$ is an admissible pair as defined above. Below, one should consider all relations as taken within $Z$.

 	We first want to show that $\mathcal B_{\epsilon,R}$ is open. If $\mathfrak M_2 := (M_2,\vol_2,p_2,E_2,S_2,f_2) \in \mathcal B_{\epsilon,R}$, it suffices to find some $\delta, T$ such that any $\mathfrak M_3$ that is $(\delta,T)$-related to $\mathfrak M_2$ lies in $\mathcal B_{\epsilon,R}$. So, let $E_i \hookrightarrow Z$, $i=1,2$, $\epsilon',R'$ and $\phi$ be the data witnessing that $\mathfrak M_2 \in \mathcal B_{\epsilon,R}$, and let $\delta$ be very small and $T$ be very large. Take some $\mathfrak M_3$ that is $(\delta,T)$-related to $\mathfrak M_2$.

Given $s_1 \in  S_1 \cap B_{M_1}(p_1,R)$, as long as $T$ is large we can apply the definition of a $(\delta,T)$-relation to get that $$1=|\{\phi(s_1)\} \cap S_1| < (1+\delta)| S_3 \cap (\{\phi(s_1)\})_{\delta}|+\delta.$$ As long as $\delta< 1$ this  implies that there is at least one element $s_3 \in S_3$ that is within $\delta$ of $\phi(s_1)$. Since $d(s_1,\phi(s_1))<\epsilon$, we then have $d(s_1,s_3)<\epsilon$ as well as long as $\delta$ is small. (The set of all $s_1$ is finite, so $\delta$ can be chosen small enough that this works for all $s_1$ simultaneously.) Now if we had two elements $s_3,s_3' \in S_3$ within $\epsilon$ of $s_1$, we have $2= |S_3 \cap \{s_3,s_3'\}| < (1+\delta) |S_2 \cap (\{s_3,s_3'\})_{\delta}| + \delta$, so there are at least two elements $s_2,s_2'$ of $S_2$ within $\delta+\epsilon$ of $s_1$. As long as $\delta $ is small, property (2) in the definition of admissibility implies that these two points lie in the image of $\phi$, so $\phi^{-1}(s_2),\phi^{-1}(s_2')$ both lie within $\delta + 2\epsilon$ of $s_1$, contradicting property (1) of admissibility. So, if we let $\psi(s_1)$ be the unique element of $S_3$ with $d(s_1,\phi(s_3))<\epsilon$, we get a map $$\psi : S_1 \cap B_{M_1}(p_1,R) \longrightarrow S_3$$ such that $d_Z(s_1,\psi(s_1))<\epsilon$ for all $s\in S_1 \cap B_{M_1}(p_1,R)$. By property (2) of admissibility, the image of $\psi$ lies in $S_3 \cap B_{M_3}(p_3,R)$.   The argument above shows that $\phi$ is an injection. And if $s_3 \in B^\bullet_{M_{3}}(p_{3},R)$, we have that $$(1+\delta)|S_2 \cap (\{s_3\})_{\delta}|+\delta \geq |S_3 \cap \{s_3\}|=1,$$ implying there's some $s_2\in S_2$ within a distance of $\delta$ of $s_3$. If $\delta$ is very small relative to the minimum distance from an element of $S_2 \setminus B^\bullet_{M_2}(p_2,R)$ to $B^\bullet_{M_2}(p_2,R)$, we can assume that this $s_2 \in S_2 \cap B^\bullet_{M_2}(p_2,R)$, so that $s_2 = \phi(s_1)$ for some $s_1$. Taking $\delta$ small again, we have $d_Z(s_3,s_1)<\epsilon$, so $s_3$ is in the image of $\psi$ as desired. This proves $\psi$ is a bijection. The fact that $|f_1(s_1) - f_2(\psi(s_1))|<\epsilon$ if $\delta$ is small follows from similar techniques. This verifies that $\mathfrak M_3 \in \mathcal B_{\epsilon,R}$, so the set $\mathcal B_{\epsilon,R}$ is open. Note also that condition (2) in the definition of admissibility implies that $\mathfrak M_1 \in \mathcal B_{\epsilon,R}$, so $\mathcal B_{\epsilon,R}$ is an open neighborhood as required.

 Next, we need to show that the sets $\mathcal B_{\epsilon,R}$ with $(\epsilon,R)$ admissible form a neighborhood basis for $\mathfrak M_1 $. For this, it suffices to fix $(\delta,T)$ and show that for sufficiently small $\epsilon$ and large $R$, any $\mathfrak M_2 \in \mathcal B_{\epsilon,R}$ is $(\delta,T)$-related to $\mathfrak M_1$. By choosing $\epsilon< \delta$ and $T<R$, we get automatically that the embeddings $E_i \hookrightarrow Z$ that verify that $\mathfrak M_2 \in \mathcal B_{\epsilon,R}$ induce $(\delta,T)$-relations of the corresponding pointed extended mm-spaces. If $F \subset B_{M_1}(p_1,T)$ is Borel, then for any $s_1\in F \cap S_1$ we have $\phi(s_1) \in (F)_\delta \cap S_2$ and hence 
 $$|S_1 \cap F| \leq |S_2 \cap (F)_\delta|.$$

 Moreover, as long as $\epsilon < \delta/|S_1 \cap B_{M_1}(p_1,T)|,$ we have
 $$\sum_{s_1 \in S_1 \cap F} (f_1(s_1) - \epsilon) < \sum_{s_2 \in S_2 \cap (F)_\delta} f_2(s_2) \implies \sum_{s_1 \in S_1 \cap F} f_1(s_1) < \sum_{s_2 \in S_2 \cap (F)_\delta} f_2(s_2) + \delta. $$
The two inequalities associated to a subset $F \subset B_{M_2}(p_2,T)$ are proved similarly, using $\phi^{-1}$ instead of $\phi$, so we have a $(\delta,T)$-relation between $\mathfrak M_1$ and $\mathfrak M_2$.
 \end{proof}

\subsubsection{Poisson processes on mm-spaces}

The reason we introduce so many spaces of spaces above is that we need to make precise the notion that the Poisson process on a pointed mm-space varies continuously with the space. 

Let $(M,\vol) $ be a mm-space and let $\mathcal S$ be the set of all locally finite subsets of $M$. Regarding a locally finite subset $S$ as an atomic Radon measure $\mu_S$ on $M$, we endow $\mathcal S$ with the weak* topology, where measures are tested against continuous functions with compact support, as discussed before Lemma \ref{weakrelations}. The \emph {Poisson process of $M $} (of intensity $1$) is the unique Borel probability measure $\rho_M$ on $\mathcal S$ such that the following hold.
\begin {itemize}
\item When  $A_1, \ldots, A_n$ are disjoint Borel subsets of $M $, the random variables that record the sizes of the intersections $S\cap A_i$ are independent.
\item If $A \subset M$ is Borel, the size of $S \cap A$ is a random variable having a Poisson distribution with expectation $\vol(A)$.
\end {itemize}
For a finite volume subset $A\subset M$ and $n\in \mathbb{N}$, we have 
\begin {equation} \label {jnnos}  \mathrm {Prob}\left (\substack{\text {for } (x_1,\ldots,x_n)\in A^n, \text{ we have } D\cap A =\{x_1,\ldots,x_n\}, \\  \text { given  that } D\cap A \text { has } n \text { elements.}}\right)=d\vol^n(x_1,\ldots,x_n).\end {equation}
In other words, if $D$ is chosen randomly, the elements of $D\cap A$ are distributed within $A$ independently according to $\vol $. 
See  \cite[Example 7.1(a)]{Daleyintroduction} for details on Poisson processes in $\RR^n$. The general case is similar. In fact, every mm-space is measure-isomorphic modulo null sets to the union of an interval in $\RR$ with a countable set of atoms, c.f.\ \cite{rokhlin1949fundamental}, so as the definition of the Poisson process is totally measure theoretic, most analyses of it can be performed on the latter space.

Suppose now that $\mathfrak M = (M,\vol,p,E) $ is a pointed, extended mm-space. Push forward the Poisson process on $M$ to a measure $\rho_{\mathfrak M}$ on $\mathbb {MS}^{ext}$, using the map \begin{equation}\label{pmap}
	\mathcal S \longrightarrow \mathbb {MS}^{ext}, \ \  S \subset M \longmapsto (M,\vol,p,E,S).
	\end{equation}
Note that the map in \eqref{pmap} is continuous: if $S $ is weakly close to $ S'$, the identity inclusions $E \hookrightarrow E$ generate an $(\epsilon,R)$-relation between $(M,\vol,p,E,S)$ and $(M,\vol,p,E,S')$.

The following is the main result of this subsection. A variant of it is claimed, but not proved, in the proof of Claim 1 on pg 582 in Bowen \cite{Bowencheeger}.

\begin{lemma}[Poisson processes vary continuously with the mm-space]\label{poissonlemma}
	The map $$\mathbb M^{ext} \longrightarrow \mathcal P(\mathbb {MS}^{ext}), \ \ \mathfrak M \longmapsto \rho_{\mathfrak M}$$ is continuous.
\end{lemma}

\begin{proof}
Suppose that we have $\mathfrak M_i = (M_i,\vol_i,p_i,E_i) \in \mathbb M^{ext}$ and $\mathfrak M_i \to \mathfrak M_\infty$. By Lemma \ref{Zspace}, we can assume that all $E_i$ are embedded in some fixed $Z$, and that the convergence happens within $Z$. Let $\mathcal S(Z)$ be the set of all $S \subset Z$ that are locally finite subsets of $M_i$ for some $i=i(S)$; endow $\mathcal S(Z)$ with the weak* topology. Then for each $i$, the Poisson process on $M_i$ can be considered as a probability measure $\rho_i $ on $\mathcal S(Z)$. By Lemma \ref{weakrelations}, the map \begin{equation}\label{zsmap}
	\mathcal S(Z) \longrightarrow \mathbb {MS}^{ext}, \ S \longmapsto (M_{i(S)},\vol_{i(S)},p_{i(S)},E_{i(S)},S_{i(S)}), 
\end{equation}
is continuous, and each $\rho_i$ pushes forward under this map to the measure $\rho_{\mathfrak M_i}$ on $\mathbb {MS}^{ext}$. So, to prove the lemma it suffices to show that $\rho_i\to \rho_\infty$ weakly.

Let $T \in \mathcal S(Z)$ with $T \subset M_\infty$, let $\epsilon,R>0$ and let $\mathcal B_{\epsilon,R}(T)$ be the set of all $S\in \mathcal S(Z)$ such that there is a bijection $$ f : T \cap B_{M_\infty}(p_\infty,R) \longrightarrow S \cap B^\bullet_{M_{i(S)}}(p_{i(S)},R)$$
such that $d(t,f(t))<\epsilon$ for all $t$. Lemmas \ref{Zspace} and \ref{bijectionlem} imply that for admissible pairs $(\epsilon,R)$, the sets $\mathcal B_{\epsilon,R}(T)$ form a basis of neighborhoods for $T \in \mathcal S(Z)$. So by the Portmanteau theorem and the fact that $\rho_\infty$ is supported on subsets of $M_\infty$, it suffices to show that 
\begin{equation}\liminf_i \rho_i(\mathcal B) \geq \rho_\infty(\mathcal B)\label{portman}\end{equation}
for all $\mathcal B := \mathcal B_{\epsilon,R}(T)$, where $T,\epsilon,R$ are as  above. 

Fixing some such $\mathcal B$, let $t \in T \cap B_{M_\infty}(p_\infty,R)$ and define $$V_i(t) :=  \vol_i(B_{Z}(t,\epsilon)).$$ By the definition of the Poisson process  and the fact that the points of $T \cap B_{M_\infty}(p_\infty,R)$ are $3\epsilon$-separated when $(\epsilon,R)$ is an admissible pair, we have \begin{align}\label{firstpart}\rho_i(\mathcal B) \ &= \left (\prod_{t } V_i(t) e^{-V_i(t)} \right )\cdot e^{- \big (\vol_i(B^\bullet_{M_i}(p_i,R)) - \sum_t V_i(t)\big ) }\\ &=  \left (\prod_{t } V_i(t) \right )\cdot e^{- \vol_i(B^\bullet_{M_i}(p_i,R))},\nonumber
\end{align}
where $t \in T \cap B_{M_\infty}(p_\infty,R)$ and $V_i(t) = \vol_i(B_{Z}(t,\epsilon) \cap M_i)$. For a $\rho_i$-random $S$, the product in the first line of \eqref{firstpart} is the probability that there is exactly one point of $S$ within $\epsilon$ of each $t$, and the second factor is the probability that there are no points of $S \cap B^\bullet_{M_i}(p_i,R)$ other than those within $\epsilon$ of the various $t$.

Recall that the inclusions of $E_i$ and $E_\infty$ into $Z$ form an $(\epsilon_i,R_i)$ relation where $\epsilon_i\to 0$ and $R_i\to \infty$. 
Pick any $0< \epsilon'<\epsilon$ and apply property (c) in the definition of an $(\epsilon_i,R_i)$-relation to $B_{M_\infty}(t,\epsilon')$. Then if $i$ is large enough so that $\epsilon'+\epsilon_i< \epsilon$, we have for all $t$ that
\begin{align*}
\vol_\infty({B_{M_\infty}(t,\epsilon')}) < (1+\epsilon_i)\vol_i( \, ({B_{M_\infty}(t,\epsilon')})_{\epsilon_i} \, ) + \epsilon_i  < (1+\epsilon_i)\vol_i( B_{M_\infty}(t,\epsilon) ) + \epsilon_i 
\end{align*}
By taking  $\epsilon'$ close enough to $\epsilon$, we can make $\vol_\infty({B_{M_\infty}(t,\epsilon')})$ arbitrarily close to $\vol_\infty({B_{M_\infty}(t,\epsilon)}$. Combining this with the fact that $\epsilon_i\to 0$, we get that 
\begin{equation}\label{liminf!}
\vol_\infty({B_{Z}(t,\epsilon)}) =  \vol_\infty({B_{M_\infty}(t,\epsilon)}) \leq \liminf_i \vol_i( B_{Z}(t,\epsilon) ) .
\end{equation}
We now apply property (c) in the definition of an $(\epsilon_i,R_i)$-relation to $B^\bullet_{M_{i}}(p_i,R)$, giving
$$\vol_i(B^\bullet_{M_{i}}(p_i,R)) < (1+\epsilon_i) \vol_\infty((B^\bullet_{M_{i}}(p_i,R))_{\epsilon_i}) + \epsilon_i .$$
But since $d(p_i,p_\infty)<\epsilon_i$, we have $(B^\bullet_{M_{i}}(p_i,R))_{\epsilon_i} \cap M_\infty \subset B_{M_{\infty}}(p_\infty,R+2\epsilon_i),$
which implies 
$$\vol_i(B^\bullet_{M_{i}}(p_i,R)) < (1+\epsilon_i) \vol_\infty(B_{M_{\infty}}(p_\infty,R+2\epsilon_i)) + \epsilon_i .$$
As $i\to \infty$, the right hand side converges to $\vol_\infty(B^\bullet_{M_{\infty}}(p_\infty,R))$, so we get
\begin{equation}\label{limsup!}
\limsup_i \vol_i(B^\bullet_{M_{i}}(p_i,R)) \leq \vol_\infty(B^\bullet_{M_{\infty}}(p_\infty,R)).
\end{equation}
Combining \eqref{liminf!} and \eqref{limsup!} proves \eqref{portman}, so we are done.
\end{proof}

\subsection{Normalized Betti numbers of mm-spaces} \label{normbmm} If $(M,\vol)$ is a finite volume mm-space, let  $\mu_{(M,\vol)}$  be the measure on  $\mathbb M$  obtained by pushing forward  the $\vol$ under $$M \longrightarrow \mathbb M, \ p \longmapsto (M,\vol,p).$$
A sequence of finite volume mm-spaces $(M_n,\vol_n)$ \emph {Benjamini-Schramm (BS) converges} if the  associated  sequence of probability measures $\mu_{(M_n,\vol_n)}/\vol_n(M_n)$ weakly converges to some  limit probability measure on $\mathbb M$.

An mm-space $M$ is \emph {special} if $M$ has finitely many path components\footnote{Bowen requires $M$ to be path connected in his definition of special, but finitely many components suffices everywhere below.}, the measure $\vol$ is non-atomic and fully supported, and metric spheres have measure zero. In \cite{Bowencheeger}, Bowen claims the following result, and justifies it by fleshing out an argument of Elek.

\begin{hypothesis}[Compare {\cite[Theorem 4.1]{Bowencheeger}}]\label{Bowennnn}
	Suppose $(M_n,\vol_n)$  is a BS-convergent sequence of finite volume special mm-spaces  and that there are constants $r,v_0,v_1$  such that
\begin{enumerate}
	\item all $r/2$-balls in $M_n$ have volume at least $v_0$, 
\item all $20r$-balls have volume at most $v_1$,
\item all $\rho$-balls in $M_n$ with $\rho<10r$ are \emph {strongly convex},  meaning that for any two points $x,y$ in a $\rho$-ball $B$, there is a unique point $z\in B$ with $d(x,z)=d(y,z)=1/2d(x,y).$
\end{enumerate}
 Then the normalized Betti numbers $b_k(M_n)/\vol(M_n)$  converge  for all $k$.
\end{hypothesis}

  As mentioned in the introduction, Bowen's proof of the Theorem \ref{Bowennnn} is not quite complete. Briefly, the Elek/Bowen argument is to construct, for each $n$, a random $\epsilon$-net $S_n\subset M_n$, i.e.\ a set of point that are $\epsilon/2$-separated in $M_n$ and where every point in $M_n$ is within $\epsilon$ of a point of $S_n$. Letting $N_n$  be the nerve complex associated to the cover of $M_n$ by $\epsilon$-balls centered at the points of $S_n$, they then say that the random complexes $N_n$ BS-converge, and then they use Elek's Theorem \ref{elek} to conclude that the expected normalized Betti numbers of the $N_n$ converge. By the strong convexity in   condition (3) above and the Nerve Lemma (c.f.\ \cite[Corollary 4G.3]{Hatcheralgebraic}), each $N_n$  is homotopy equivalent to $M_n$, so $b_k(N_n)=b_k(M_n)$. One can also relate the number of vertices of $N_n$ to the volume of $M_n$, so this implies the convergence of the normalized Betti numbers of $M_n$.

 Above, the random nets $S_n$ are constructed  as subsets of the union of infinitely many randomly chosen discrete subsets of $M_n$, each of which is chosen according to a \emph{Poisson process}.  In order to ensure separation of the net, Elek/Bowen enumerate all the discrete subsets and their points, and add them into $S_n$ one by one, throwing out the points that are too close to the previously added points.  The problem with this is that it is very hard to prove that such random nets vary continuously when the underlying space is changed, which is essential for BS-convergence of the associated nerve complexes. In \cite{Bowencheeger}, this issue is not really addressed. The construction of these subsets is the content of Lemma 4.2 of \cite{Bowencheeger}, and the last line of the proof (see the end of the first paragraph of pg 584) seems to indicate that continuity of the $\epsilon$-nets follows immediately from continuity of the `almost nets' one would obtain by superimposing only a fixed number of Poisson processes,  instead of infinitely many of them. However, this is not true; it is like saying that the limit of continuous functions is always continuous. The question of whether the Elek/Bowen random nets do vary continuously with the underlying mm-space seems quite subtle in general, and while we do not have a counterexample, we think that a proof of this would be more difficult than the alternative approach we take in the current paper.

Below, we will prove a slightly more (and less) general result, Theorem \ref{Bowen2}. While it does not strictly imply Theorem \ref{Bowennnn}, it can be used in all Bowen's applications.  The proof essentially follows the Elek/Bowen argument, but we get around the continuity issue by only superimposing a fixed finite number of Poisson processes, creating an $\epsilon/2$-separated `almost net' $S_n\subset M_n$. While $S_n$ may not be a net, we show that it can be completed to a net using a small number of points, so the Betti numbers of the associated nerve complex still approximates that of  $M_n$, allowing us to run the rest of the Elek/Bowen argument.

To motivate the  statement of the more general result,  look again at the  statement of Theorem \ref{Bowennnn}. Condition (3)  is only used to say that the nerve complex is homotopy  equivalent to $M_n$, so we should be able to  state a version of Theorem \ref{Bowennnn} in which (3) is omitted, if we talk about the Betti numbers of the nerve complexes  directly instead of the Betti numbers of the $M_n$.  Next, to make a result that is compatible with the machinery of Gelander described in \S \ref{thickapproxsec}, it is also important for us to take nets in the $M_n$, but  construct the corresponding nerves using balls in larger spaces $E_n$.  In other words, we need to work with the \emph {extended mm-spaces} of \S \ref {mmsec}.  %Finally, we will need to work with {convex combinations} of mm-spaces instead of just with mm-spaces themselves, since in \S \ref{pncsec} it is easier to get a sequence of convex combinations to weakly converge.

 To that end, we say that an  extended mm-space $\mathfrak M = (M,\vol,E)$ is \emph {finite volume} or \emph{special} if the mm-space $M$ is. When $\mathfrak M$  has finite volume,  we can construct a finite measure $\mu_{\mathfrak M}$ on $\mathbb M^{ext}$ by pushing forward $\vol$ under the map $$p\in M \mapsto (M,\vol, p,E).$$
If $\mathfrak M_n=(M_n,\vol_n,E_n)$ is a sequence of  extended mm-spaces, then we say that $(\mathfrak M_n)$ \emph{BS-converges} if the sequence of measures $\mu_{\mathfrak M_n}/\vol_n(M_n)$ weakly converges.

We define an \emph{$(r_0,r_1)$-net} in $\mathfrak M$ to be a subset  $S \subset M$  such that 
\begin{enumerate}
\item $S$ is \emph{$r_0$-separated,} i.e.\ $d(x,y)> r_0$  for all $x\neq y\in S$,
\item  $S$ \emph{$r_1$-covers} $M$, i.e.\  for every $p\in M$,  there is some $x\in S$ with $d(p,x)<r_1,$
\end{enumerate}
and an \emph{$[r_2,r_3]$-weighted $(r_0,r_1)$-net} is a $(r_0,r_1)$-net $S$ with a function $$\rho : S \longrightarrow [r_2,r_3],$$ where here $r_0 < r_1 \leq r_2 < r_3.$
Given any weighted net $(S,\rho)$ in $\mathfrak M$, we let $N_{E}(S,\rho)$ be the nerve complex associated to the collection of $E$-balls $B_{E}(x,\rho(x)),$ where $x\in S$. 

% Finally, a \emph{convex combination} of  extended mm-spaces is a finite formal sum $$\sum_i t^i \cdot \mathfrak M^i, \ \text{ where } t^i \in [0,1] \ \forall i, \ \ \sum_i t^i =1.$$
%We say that  a sequence $\sum_i t^i_n \cdot \mathfrak M^i_n$  of such combinations \emph {weakly  converges} if the   measures \begin{equation}\label {quoooootient!}
% 	\frac{\sum_i t^i_n \cdot \mu_{\mathfrak M_n^i}}{\sum_i t^i_n \cdot \vol_n^i(M_n^i)}
% \end{equation}
% weakly converge. Note that there are always finitely many indices $i$  for each fixed $n$, but this number may be unbounded as $n\to \infty$. In \eqref{quoooootient!}, note that we are dividing by the total weighted sum of the volumes, rather than taking the weighted sum of the quotients $\mu_{\mathfrak M_n^i} / \vol_n^i(M_n^i)$. This makes sense: if the $t^i_n$ are integers, we can view each $\sum_i t^i_n \cdot \mathfrak M^i_n$ as a disconnected space containing $t^i_n$ copies of each $\mathfrak M^i_n$, and  \eqref{quoooootient!} is the  natural generalization of the probability measure on $\mathbb M^{ext}$ associated to a connected extended mm-space.
%
% 
\begin{theorem}\label{Bowen2}Fix $k$, let $\mathfrak M_n=(M_n,\vol_n,E_n)$  be a BS-convergent sequence of extended finite volume special mm-spaces and suppose we have constants $v_{min}>0$, $r_1 >r_0 >0$, and  $r_3 >r_2 \geq 2r_1, $ a function $v_{max} : \BR_+ \longrightarrow \BR_+$  such that 
\begin{enumerate}
	\item all $r_0/2$-balls in every $M_n $ have volume at least $v_{min}$, 
\item for all $r \in \BR_+$, every $r$-ball in $M_n$ has volume at most $v_{max}(r)$.
\end{enumerate}
Now suppose that we have a sequence $B_n$ of positive numbers such that
\begin{enumerate}
\item[(3)] for any sequence of  $[r_2,r_3]$-weighted $(r_0,2r_1)$-nets $(S_n,\rho_n)$ in $ M_n$, $$\frac{\big | b_k(N_{E_n}(S_n,\rho_n)) - B_n \big | }{\vol_n(M_n)}\to 0.$$
\end{enumerate}
Then the ratios $B_n/\vol_n(M_n)$  converge.
\end{theorem}

In our two applications, Theorem \ref{pnc} and Theorem \ref{npc}, the numbers $B_n$ will be the Betti numbers $b_k(M_n)$ and the Betti numbers $b_k(E_n)$, respectively. We state it as above to have a single unified statement that applies in both situations. Note that when applying Theorem~\ref{Bowen2}, one has to show that the Betti numbers of the nerve complexes associated to \emph{all} nets in (3) are approximated by a single sequence $B_n$. This usually requires an argument that goes through the Nerve Lemma at some point.

Given a sequence $(M_n,\vol_n)$ of finite volume special mm-spaces, we can apply Theorem~\ref{Bowen2} to the extended mm-spaces $(M_n,\vol_n,M_n)$, with $B_n = b_k(M_n)$, to get a slightly weaker version of Theorem~\ref{Bowennnn}. The difference is that hypothesis (2) in Theorem~\ref{Bowen2} is formally stronger than it is in Theorem~\ref{Bowennnn}, but in basically all applications, upper bounds on ball volumes come from curvature lower bounds, which imply both versions of (2). Note, however, that by the Nerve Lemma the nerve of any covering of $M_n$ by strongly convex balls is homotopy equivalent to $M_n$, so condition (3) in Theorem \ref{Bowennnn} implies condition (3) in Theorem \ref{Bowen2} with $B_n = b_k(M_n)$, after adjusting the constants appropriately.

\medskip

  %Most of the proof generally follows Bowen's  outline for \cite[Theorem 4.1]{Bowencheeger}.  However, fixing the significant error requires restructuring  the entire argument, so we will have to go through basically every step in his proof.  When possible, though, we will just cite lemmas from his paper instead of  repeating their proofs verbatim.  See  Remark \ref{error}  at the end of the section for a detailed explanation of the error in Bowen's paper, and the differences between our argument and his.

 Before starting the proof, we also record two brief lemmas. First,  as mentioned above Elek's Theorem \ref{elek} is crucial in the proof below. Here is a formal consequence of his result with a more general  sounding statement. 

\begin{lemma}\label {elekimproved}
Suppose that for $n=1,2,\ldots$, we have  a probability measure $\eta_n$  on the space of pointed complexes $\mathcal K$  that is of the form $$\eta_n = \frac{\sum_{m=1}^{M_n} t_{n,m} \mu_{X_{n,m}} }{\sum_{m=1}^{M_n} t_{n,m} \vol(X_{n,m})},$$
where $X_{n,m}$ are finite complexes with  universally bounded degree, and $\mu_{X_{n,m}}$ is the measure on $\mathcal K$ obtained by pushing forward the counting measure on the vertex set of $X_{n,m}$, as in \S \ref{intro}.  Then if  the measures $\eta_n$ weakly converge, the ratios $$\frac{\sum_{m=1}^{M_n} t_{n,m} b_k(X_{n,m})}{\sum_{m=1}^{M_n} t_{n,m} \vol(X_{n,m})},$$ converge for all $k$.
\end{lemma}

 The reader can compare this with Lemma 2.2 in Bowen \cite{Bowencheeger}, although that lemma is incorrectly stated\footnote{In Lemma 2.2 of \cite{Bowencheeger}, Bowen sets $\eta_i =\sum t_{i,j} \mu_{K_{i,j}}$, where in his paper the $\mu_{K_{i,j}}$ are the \emph{normalized probability measures} associated to the $K_{i,j}$, but in order to get his conclusion you need to normalize to a probability measure \emph{after} taking the convex combination, like we do in our lemma.  Also, it is worth noting that in his proof of Lemma 2.2, Bowen wedges complexes together in order to create a connected complex, but as Elek's theorem actually applies to disconnected complexes, he  could have just taken the disjoint union instead of the wedge, which makes his hypothesis on the sizes of the complexes unnecessary.}.  %Informally, one should consider Lemma 2.6 a version of Elek's Theorem for random (unpointed) complexes. That is, if we have some random way of constructing finite (unpointed) complexes, and a sequence of such random complexes BS-converges, then the ratios of the expected Betti numbers to the expected volumes converge as well.

\begin{proof}
 As in Bowen's proof of \cite[Lemma 2.2]{Bowencheeger},  it suffices to prove the lemma when the coefficients $t_{n,m}$ are rational, so we can assume we have integers $D_n$ such that $D_n t_{n,m} \in \BN$ for all $n,m.$ Then we can create a complex $Y_n $ by taking the disjoint union of $D_n t_{n,m}$ copies of each $X_{n,m}$. Since $\mu_{Y_n}/\vol(Y_n) = \eta_n$, these $Y_n$ BS-converge, so by Elek's Theorem \ref{elek} their  normalized Betti numbers converge, and the conclusion of the lemma follows. 
 %Note that here it is important that Elek's Theorem holds for disconnected complexes, which is why we wrote it that way in the introduction.
\end{proof}

Note that in the argument above it is important that Elek's Theorem holds for disconnected complexes, which is why we wrote it that way in the introduction.

%
%\marginpar{\tiny FYI, the statement of Lemma 2 in the lemma.tex file  isn't quite right. (This Lemma replaces it.) Consider for instance $A=\cup_{i=1}^n [i-1/n,i+1/n]$, $n>>0$. Then $\vol(A)=2$,  but for any fixed $c$, for large $n$  there are no points of $A$  with $\vol(B_r(x)\cap A) \geq c$.  One problem is that when you are computing $\nu(W)$, you say $\lambda(B_{r/2}(x) \cap A) \geq \lambda(B_{r/2}(x))$,  but it's the other way around.
%}

 Finally, we record the following elementary measure theoretic lemma.

\begin{lemma}
Suppose that $(M,\vol)$ is a mm-space and that every ball in $M$ with radius in the interval $[r_0/2,r_0]$ has  volume between $v_{min}$ and $v_{max}$. Set $c=v_{min}^2/(2v_{max}),$ $c' =   v_{min}^2/(2v_{max}^2)$. Then for every measurable  subset $A\subseteq M$, we have that
\[
\vol \big (\left\{  x\in A\  \Big | \  \vol(B_{r_0}(x)\cap A)\geq c \cdot \frac{\vol(A)}{\vol(M)} \right\}  \big ) \geq c' \cdot \frac{\vol(A)}{\vol(M)} \cdot \vol(A).%
\]
\label {lemlem}
\end{lemma}

Here, the ball $B_{r_0}(x)$  is the metric ball in $M$. Note that if $\vol(A)/\vol(M)$  is bounded away from zero, the lemma says  that a definite proportion of $A$  is taken up by points $x\in A$ such that $A$  takes up a definite proportion of $B_{r_0}(x)$. 
 
\begin{proof}
Note that 
\begin{align*}
	\int_A\vol(B_{r_0}(x) \cap A) \, dx &= \vol( \{(a,b) \in A^2 \ | \ d(a,b)<r_0\}) \\
& \geq \frac 1{v_{max}} \vol( \{(x,a,b) \in M \times A^2 \ | \ d(a,x)<r_0/2 \text{ and } d(a,b) < r_0\}) \\
& \geq  \frac 1{v_{max}} \vol( \{(x,a,b) \in M \times A^2 \ | \ d(a,x)<r_0/2 \text{ and } d(x,b) < r_0/2\}) \\
& =\frac 1{v_{max}} \int_M \vol(B_{r_0/2}(x) \cap A )^2 \, dx \\
&\geq  \frac 1{v_{max}\vol(M)} \left (\int_M \vol(B_{r_0/2}(x) \cap A)  \, dx \right )^2\\
&=  \frac 1{v_{max}\vol(M)} \left (\int_A \vol(B_{r_0/2}(x))  \, dx \right )^2\\
&\geq  \frac {v_{min}^2\vol(A)^2}{v_{max}\vol(M)} \\
&= 2 \cdot \left (c \cdot \frac {\vol(A)}{\vol(M)}\right ) \cdot \vol(A).
\end{align*}
From this the lemma follows immediately, since if $f : A \longrightarrow [0,max]$ is a function then
$$\int_A f \geq 2\cdot \epsilon \cdot \vol(A) \implies \vol\{x \in A\ | \ f(x)\geq \epsilon \}\geq \frac \epsilon{max}\vol(A) .\qedhere $$

\end{proof}

\subsection{The proof of Theorem \ref{Bowen2}} We now begin the proof. Recall from the previous section  that the general strategy is to show how to construct a random `almost net' in a given finite volume space such that the associated nerve complex $N_n$ has Betti numbers close to those of the nerve complex associated to an actual net. Then one uses Elek's Theorem (or rather, Lemma \ref{elekimproved}) to show that the normalized Betti numbers of these nerve complexes $N_n$ converge, and finally one deduces from this and property (3) in the statement of Theorem~\ref{Bowen2} that the ratios $B_n/\vol_n(M_n)$ converge.

\subsubsection{Random `almost nets'.} \label{almostnets}Fix a finite volume special mm-space $(M,\vol)$ and real numbers $0<r_0<r_1$.  For each $j \in \BN$, let $P^j$ be a Poisson process on $M$ with intensity $1$, and let $f^j : P^j \longrightarrow [0,1]$ be a random function whose values are chosen independently according to Lebesgue measure. 
%(Just for convenience in notation, let's assume below that almost surely, $P^i \cap P^j = \emptyset$ for $i\neq j$: this is the case if the measure $\vol$ is atomless, which is always the case in our applications.)
Each function $f^j$ is almost surely injective, and when it is, it induces a linear order $``<"$ on $P^j$ via $s < t \iff f^j(s) < f^j(t).$ Set $$P^j({<s}) = \{t \in P^j \ | \ t < s\}.$$

Pick some $r$ with $r_0<r<r_1$ and choose a continuous function $$\phi : [0,\infty) \longrightarrow [0,1], \text{ where }\  \phi(t) = 0 \text{ if } t \leq r_0, \  \phi(t)=1 \text{ if } t \geq r.$$ Let $P=\sqcup_j P_j$ be the disjoint union, and for each pair $s,t\in P$, let $X(s,t)$ be a Lebesgue-random element of $[0,1]$, where the $X(s,t)$ are independent as $s,t$  are varied.
%Finally, define for $s \in P^i$ and $t \in P^j$ that `$s \lhd t$' if $$ f^i(s) \leq f^j(t) \text{ and } \phi(d(s,t)) \geq X(s,t).$$
%Note that $X(s,t) > 0$ almost surely, so $s\lhd t$ implies $d(s,t)<r_1$  almost surely. Note that the condition on $d(s,t)$ is obviously symmetric, while the condition that $f^i(s) \leq f^j(t)$ is essentially a random choice of whether $s \lhd t$ or $t \lhd s$.
 We now recursively define subsets $$S^j\subset P^j, \ \ \ S^{\leq j} := S^1 \cup \cdots \cup S^j, \ \ \ S^{<j} := S^1 \cup \cdots \cup S^{j-1},$$ where given $S^1,\ldots,S^{j-1}$, the rule is that for $s\in P^j$, we say that $s \in S^j$ if 
$$ \text{ for all } t \in P^j({<s}) \cup   S^{<j}, \ \ \phi(d(s,t))\geq X(s,t).$$
In other words, go through all the elements $s$ of a Poisson process $P^1$ one by one, in some random order. For each $s$, backtrack through all previously considered  $t$, flip for each a $[0,1]$-valued coin, and add $s$ to $S^1$ if for each $t$, the value $\phi(d(s,t))$ is bigger than the result of the coin flip. After finishing with all available $s$, switch over to a new Poisson process, and add points to $S^2$  using a similar rule, comparing them against previous points in $P^2$ and also against all points in $S^1$. Then  repeat this with a third Poisson process to define $S^3$, and a fourth to define $S^4$, etc...   

For later use, we record:

\begin{claim}\label {boundbelow}
There is some $c=c(r_0,r_1,v_{min},v_{max})>0$  such that for all $j$, if $M$  satisfies conditions (1) and (2)  in the statement of the theorem, we have $E[|S^{\leq j}|] \geq c\cdot \vol(M)$.
\end{claim}
\begin{proof}
	 Certainly, it suffices to set $j=1$. Let $B_1, \ldots, B_k$ be a  maximal collection of disjoint $r_0$-balls in $M$, and note that $k \geq c_1 \cdot \vol(M)$ for some  uniform $c_1$, by (1) and (2). For each $i$,  there is a probability bigger than some fixed constant that the Poisson process $P_1$ will intersect the $r_1$-neighborhood of $B_i$ in a single point $x$ that lies in $B_i$.  When this happens, this $x$ will automatically be included in $S^1$. So,  $E[|S^1 \cap B_i|] \geq c_2$  for some uniform $c_2>0$. Hence, $E[|S^1|] \geq c_1 \cdot c_2 \cdot \vol(M)$ by linearity of expectation. 
\end{proof}
By the definition of $\phi$, we will almost never add $s$ to $S^j$ if $d(s,t) \leq r_0$  for some previously  considered $t$. So, for each $j$, the subset $S^{\leq j}$ is almost surely $r_0$-separated.   On the other hand, we cannot ensure that any particular $S^{\leq j}$ is an $(r_0,r_1)$-net, since it may not $r_1$-cover $M$. (You do get a net if you take $j=\infty$, as in Bowen's proof.) However, set 
$$n^j = \min \, \big \{ \, |T \setminus S^{\leq j} | \ \big | \ T \text{ is a } (r_0,2r_1)\text{-net in } M \text{ with } T \supset S^{\leq j}\big \}, $$
a random  integer associated to each  choice of $M$ and $j$. We prove:
\begin{proposition}\label{uncoveredprop}
	Given $\epsilon>0$, there is some $j=j(\epsilon,r_0,r_1,v_{min},v_{max})$ such that for any $M$  satisfying (1) and (2) in the statement of the theorem, we have $$\frac{E[n^j]}{\vol(M)} < \epsilon.$$ 
\end{proposition}
\begin{proof}
For each $j$, write $R^{j}$ for the complement of the $r_1$-neighborhood of $S^{\leq j} \subset M$, and let $R^j_2$ be the  complement of the $2r_1$-neighborhood.  If $X$ is any maximal $r_0$-separated set in the space $R^j_2$, then the $M$-balls $B_{r_0}(x), x\in X$ are disjoint and contained in $R^j$, so
$$v_{min} \cdot |X| \leq   \vol (R^j).$$ Since the union $X \cup S^j$ is a $(r_0,2r_1)$-net in $M$, this means $n^j \leq \vol(R^j) / v_{min},$  so to prove the  proposition it suffices (after adjusting $\epsilon$) to find $j$ such that 
\begin{equation} \label{whatwewant}\frac{E[\vol(R^j)]}{\vol(M)} < \epsilon.	
\end{equation}

\begin{claim}
 There is some $\delta=\delta(\epsilon,r_0,r_1,v_{min},v_{max})<1$ as follows. Suppose a fixed $S^{\leq j} $, and hence $R^j$, is given, and that $\vol(R^j)/\vol(M) \geq \epsilon/2$. Then \label{shrinkclaim}
$$E[\vol(R^{j+1}) \, | \, R^j] \leq \delta \cdot \vol(R^{j}).$$

\end{claim}
Here, we write $E[\vol(R^{j+1}) \, | \, R^j]$  to indicate that this is the expected volume of $R^{j+1}$, conditioned on our particular choice of a fixed $R^j$.  This is to remove some ambiguity when we apply the claim later. In the \emph{proof} of Claim \ref{shrinkclaim}, though, we will always consider $R^j$ as fixed and just write $E[\ \cdot \ ]$, omitting any reference to $R^j$. Also, to  avoid a proliferation of constants in the following proof, we  will use the notation $x \preceq y$ to mean that $x \leq C y $ for some constant $C > 0$ depending only on $\epsilon,r_0,r_1,v_{min},v_{max}.$
\begin{proof}
Fix $c,c'$ as in Lemma \ref{lemlem}. Let $R^{j}_\circ$  be the subset of $R^j$  consisting of all points $x\in R^j$  such that $\vol(B_{r_1}(x) \cap R^j)\geq c \cdot \epsilon/2$. Then Lemma \ref{lemlem} says that \begin{equation}
 	\vol(R^j_\circ) \succeq \vol(R^j).\label{succcceq}
 \end{equation}
Let $R^{j}_{\circ\circ}$ be the  further subset consisting of all points $x\in R^j_{\circ}$  such that $\vol(B_{r_1}(x) \cap R^j_{\circ})\geq c \cdot \epsilon/2$.  Applying Lemma \ref{lemlem} and using \eqref{succcceq}, we also have 
$$\vol(R^j_{\circ\circ}) \succeq \vol(R^j_\circ) \succeq \vol(R^j). $$
Fix a maximal $r_0$-separated (say) subset  $Z \subset R^j_{\circ\circ}$.  By  assumption,  there is an upper bound for the volumes of all the $r_0$-balls around points $z\in Z$, so since $R^j_{\circ\circ}$ is contained in the union of all such balls, our lower bound on the volume of $R^j_{\circ\circ}$  implies that
\begin {equation}
	|Z| \succeq \vol(R^j). \label{fofofo}
\end {equation}
  For each $z\in Z \subset R^j_{\circ\circ}$, the volume of $B_{r_1}(z) \cap R^j_{\circ}$ is  bounded below and the volume of $B_{2r_1}(z)$ is bounded above and below. So if $P^{j+1}$  is the Poisson  process used in defining $S^{j+1}$, we have \begin{equation}
 	P^{j+1} \cap B_{r_1}(z) \cap R^j_{\circ} \neq \emptyset \text{ and } |P^{j+1} \cap B_{2r_1}(z)|=1 \label {connndition}
 \end{equation}
 with probability bigger than some fixed constant. So by linearity of expectation and \eqref{fofofo}, 
$$E\Big [ \, \big |\{z \in Z \ | \ \eqref{connndition} \text { holds for } z \} \big | \, \Big ] \succeq \vol(R^j),$$
where the expectation is taken over the Poisson process $P^{j+1}$. 
But for every $z$ such that \eqref{connndition} holds, the single point of $P^{j+1}$ that  is in $B_{r_1}(z) \cap R^j_{\circ}$ is at least $r_1$-away from every other point of $P^{j+1}$, and is also at least $r_1$-away from $S^{\leq j}$, since $z \in R^j$. Hence, this single point lies not just in $P^{j+1}$, but in $S^{j+1}$.  It follows that
\begin{equation}
	E\big [ \, | \{ x \in S^{j+1} \cap R^j_{\circ} \ | B_{r_1}(x) \cap S^{j+1}=\{x\} \}\, | \big ] \succeq \vol(R^j).\label{finall}
\end{equation}

Note that $R^{j+1} = R^j \setminus \bigcup_{x \in S^{j+1}} B_{r_1}(x).$ By  definition of $R^j_\circ$, the $r_1$-ball around each $x \in  S^{j+1}\cap R^j_{\circ}$ intersects $R^j$ in a  set with volume bounded below, and if  we only look at those $x$  where $B_{r_1}(x)\cap S^{j+1} = \{x\}$, all the balls $B_{r_1}(x)$  are disjoint. So by \eqref{finall}, $$E\left [ \vol \left (R^j \, \cap \bigcup_{x \in S^{j+1}} B_{r_1}(x) \right) \right ] \succeq \vol(R^j),$$
 and the claim follows. 
\end{proof}

We  now complete the proof of the proposition. Let $\sigma^j$ be the law of $S^{\leq j}$. Then conditioning on whether $\vol(R^j)/\vol(M) \geq \epsilon/2$ or not, we have 
\begin{align}\label{condition}
	\frac{E[\vol(R^{j+1})] }{\vol(M)}
&\leq  \frac{\epsilon}{2} +  \int_{\vol(R^j)/\vol(M) \geq  \epsilon/2} 
\frac{E[ \vol(R^{j+1}) \, |\, R^j]}{\vol(M)} \, d\sigma^{j}.% \\ &\leq  \frac{\epsilon}{2} +  \delta \cdot \int_{\vol(R^j)/\vol(M) \geq  \epsilon/2}  \frac{\vol(R^j)}{\vol(M)} \, d\sigma^{j}, \text{ by Claim } \ref{shrinkclaim},
\end{align}
Let's call the second term on the right in \eqref{condition}  $X^{j+1}$. Then by Claim \ref{shrinkclaim}, we have
$$X^{j+1} \leq  \delta \cdot \int_{\vol(R^j)/\vol(M) \geq  \epsilon/2}  \frac{\vol(R^j)}{\vol(M)} \, d\sigma^{j} \leq \delta \cdot \int_{\vol(R^{j-1})/\vol(M) \geq  \epsilon/2} \frac{E[\vol(R^j)]}{\vol(M)} \, d\sigma^{j-1} \leq \delta X^j  $$
for all $j$, where the middle inequality uses  the inclusion $R^j \subset R^{j-1}$ to say that the condition on $\vol(R^j)$ is  \emph {at least} as restrictive as the condition on $\vol(R^{j-1})$.  Since $\delta <1$  is fixed, there is some uniform $j=j(\epsilon,r_0,r_1,v_{min},v_{max})$  such that $X^j < \epsilon/2$, and then $$\frac{E[\vol(R^{j})] }{\vol(M)} < \epsilon/2 + \epsilon/2 = \epsilon$$
as desired in \eqref{whatwewant}.
\end{proof}

%\begin{claim}\label {independence claim}
%	 Suppose that $A,B \subset M$ and that $d(A,B) \geq 2\cdot (j+1) \cdot r_1$. Then the random subsets $S^{\leq j} \cap A$ and $S^{\leq j} \cap B$ are independent.
%\end{claim}
%\begin {proof}
%For $s,t$ as in \eqref{defeqqq}, if $d(s,t) \geq r_1$ then $\phi(d(s,t))=1 \geq X(s,t) \in [0,1]$. So, \emph{the decision whether to include $s$ in $S^{ j}$ only depends upon  the points $t$ that lie in an $r_1$-neighborhood of $s$.}   So, suppose $d(A,B)\geq (j+1) \cdot r_1$. The intersections $P^{j} \cap N_{r_1}(A)$ and $P^{j} \cap N_{r_1}(B)$  are independent, since $P^{j}$ is a Poisson process, and the order of the elements in the first set is also independent of the order in the second, since the order on $P^{j}$ is determined by picking  independent, random values in $[0,1]$ for each $s \in P^{j}$.  Assuming inductively that $S^{<j} \cap N_{r_1}(A)$ and $S^{<j} \cap N_{r_1}(B)$ are independent, the result follows for $j$.
%\end {proof}
%

\subsubsection{Continuity of the random almost nets $S^{\leq j}$}\label{contnet}
If $\mathfrak M=(M,\vol,p,E)$ is an extended special, pointed mm-space and $j$ is fixed, we can choose a random element of $\mathbb {MS}^{ext} $ by choosing a random $S^{\leq j} \subset M$ as above. The law of this random element is a measure $\sigma^{\leq j}= \sigma^{\leq j}(\mathfrak M)$ on $\mathbb {MS}^{ext} $ depending on $\mathfrak M$, and this defines a function 
\begin{equation}
\mathbb M_{sp}^{ext} \longrightarrow \mathcal P(\mathbb {MS}^{ext}), \ \ \mathfrak M \longmapsto \sigma^{\leq j}(\mathfrak M) \label{sigmamap}
\end{equation}
where as usual $\mathcal P( \cdot )$ denotes the space of probability measures. 

\begin{lemma}\label{contclaim}
The function in \eqref{sigmamap} is continuous for all $j$.
\end{lemma}

 This proof is suggested by Claim 2 on pg 583 of \cite{Bowencheeger}, but Bowen only does the $j=1$ case, so for completeness we will give the argument here. Note that Bowen also does not prove continuity of the Poisson process, which we did above.

\begin{proof}
In Lemma \ref{poissonlemma}, we showed that the map 
\begin{equation}\label{poissonmap}
\mathbb M^{ext} \longrightarrow \mathcal P(\mathbb {MS}^{ext}), \ \  \mathfrak M \longmapsto \rho_{\mathfrak M}
\end{equation} that associates to an extended mm-space its Poisson process is continuous. Consider the map \begin{equation}\label{unif}
	\mathbb {MS}^{ext} \longrightarrow \mathcal P(\mathbb {MF}^{ext})
\end{equation}
 that takes a tuple $\mathfrak M = (M,\vol,p,E,P)$ to the measure $\nu_{\mathfrak M}$ whose random element is of the form $(M,\vol,p,E,P,f)$, where the values of $f : P \longrightarrow \BR$ are chosen independently and uniformly from $[0,1]$. We claim that the map \eqref{unif} is continuous. For suppose $(M_i,\vol_i,p_i,E_i,P_i)\to (M,\vol,p,E,P)$ in $\mathbb {MS}^{ext}$. Then Lemma \ref{Zspace} allows us to realize the convergence within  $\mathbb {MS}^{ext}(Z)$ for some $Z$. Fixing some $f : P \longrightarrow [0,1]$ and an admissible pair $(\epsilon,R)$ for $(M,\vol,p,E,P,f)$, consider the neighborhood $\mathcal B_{\epsilon,R}$ given by Lemma \ref{bijectionlem}. For large $i$, the hypothesized convergence gives bijections $$ \psi_i : P \cap B_{M}(p,R) \longrightarrow P_i \cap B^\bullet_{M_i}(p_i,R)$$ 
such that $d_Z(P,\psi(s))<\epsilon$ for all $s$, and a tuple $(M_i,\vol_i,p_i,E_i,P_i,f_i)$ is in $\mathcal B_{\epsilon,R}$ exactly when $|f(s) - f_i(\psi(s))| < \epsilon$ for all $s$. This event has $\nu_{\mathfrak M_i}$-measure $\epsilon^{|P \cap B_{M}(p,R)|}$ for $i=1,2,\ldots,\infty$, so by the Portmanteau theorem we have $\nu_{\mathfrak M_i} \to \nu_{\mathfrak M}$ as desired.

It follows from the above that the composition
$$\mathbb M^{ext} {\longrightarrow} \mathcal P(\mathbb {MS}^{ext}) \longrightarrow \mathcal P(\mathcal P(\mathbb {MF}^{ext})) \longrightarrow \mathcal P(\mathbb {MF}^{ext})$$
is continuous, where the first map is \eqref{poissonmap}, the second is the weakly continuous map induced by \eqref{unif}, and the third is the expectation map. Given $\mathfrak M \in \mathbb M^{ext}$, a random element of the associated measure on $\mathcal P(\mathbb {MF}^{ext})$ is exactly a $[0,1]$-weighted Poisson process $(P,f)$. 

Let 
$\mathbb {MF}^{j,ext}$ be the set of all pointed, extended mm-spaces $(M,p,\vol,E)$ that come equipped with $j$ weighted locally finite subsets $(P^i,f^i)$, $i=1,\ldots, j$, which we topologize by regarding it as a subset of $\mathbb {M}^{2j+1,ext}$. Consider the map 
\begin{equation}\label {somanymaps}
\mathbb M^{ext} \longrightarrow \mathcal P(\mathbb {MF}^{j,ext})
\end{equation} that takes $(M,\vol,p,E)$ to the measure whose random element is given by selecting $j$ weighted Poisson processes randomly on $M$, as above. Then the same arguments as above show that \eqref{somanymaps} is continuous. Let $\mathbb {MF}^{j,ext}_{inj} \subset \mathbb {MF}^{j,ext}$ be the subset consisting of tuples where all the $f^i$ are injective and consider the map 
\begin{equation}\label{sjconstrmap}
\mathbb {MF}^{j,ext}_{inj} \longrightarrow \mathcal P(\mathbb {MS}^{ext})
\end{equation}
that takes $(M,\vol,p,E, P^1,f^1,\ldots,P^j,f^j)$ to the measure whose random element is the tuple $(M,\vol,p,E,S^{\leq j})$ constructed from the given data as in the previous section. (This last element of randomness comes from the need to pick a random value $X(s,t)\in [0,1]$ for every pair of elements $s,t \in P := \cup_i P^i$.) Since \eqref{somanymaps} is continuous and any measure in its image gives   $\mathbb {MF}^{j,ext}_{inj}$ full mass, it suffices now to show that \eqref{sjconstrmap} is continuous.

So, suppose that $\mathfrak M_i = (M_i,\vol_i,p_i,E_i, P^1_i,f^1_i,\ldots,P^j_i,f^j_i) \in \mathbb {MF}^{j,ext}_{inj}$, where $i=1,\ldots,\infty$, and that $\mathfrak M_i \to \mathfrak M_\infty$. Applying the appropriate analogue of Lemma \ref{Zspace}, we can assume that this convergence happens within $\mathbb {MF}^{j,ext}_{inj}(Z)$ for some fixed $Z$. Then the construction in \eqref{sjconstrmap} gives a sequence of probability measures $\nu_i$ on $\mathbb{MS}^{ext}(Z)$, where a $\nu_i$-random element is $(M_i,\vol_i,p_i,E_i,S^{\leq j}_i)$, with $S^{\leq j}_i$ constructed from $(P^j_i,f^j_i)$ as in the previous section. We want to show that the measures $\nu_i$ weakly converge.

Set $P_i := \cup_{k=1}^j P_i^k$. After discarding finitely many $i $,  Lemma \ref{bijectionlem} says that there is a sequence $(\epsilon_i,R_i)$ of pairs, where $\epsilon_i\to 0$ and $R_i\to \infty$ and where each pair is admissible with respect to $(M_\infty,\vol_\infty,p_\infty,E_\infty,P_\infty)$, and bijections
$$ \phi_i : P_\infty \cap B_{M_\infty}(p_\infty,R_i) \longrightarrow P_i \cap B^\bullet_{M_i}(p_i,R_i)$$ 
such that $d_Z(s,\phi_i(s))<\epsilon_i$ and $|f_\infty(s) - f_i(\phi_i(s))|<\epsilon_i$ for all $s$. The following claim is in some sense the heart of the proof of this lemma\footnote{The proof of Claim \ref{probability claim} is what fails if you let $j=\infty$ as in Bowen's paper \cite{Bowencheeger}. For in that case, the decision to include an element $s \in S^{\leq \infty}_i$ does not depend just on the part of $P_i$ that lies in a neighborhood of fixed radius around $s$, but potentially on the entire $P_i$.}.

\begin{claim} \label {probability claim} For any fixed $s_\infty \in P_\infty$, the $\nu_i$-probability that $\phi_i(s_\infty) \in S^{\leq j}_i$ converges to the $\nu_\infty$-probability that $s_\infty \in S^{\leq j}_\infty$. 
\end{claim}

\begin {proof}
Let us recall the definition of $S^{j}_i$. Fix some function $\rho : [0,\infty] \longrightarrow [0,1]$ with  $\rho(t) = 0 \text{ if } t \leq r_0, \  \rho(t)=1 \text{ if } t \geq r_1$; we called this function $\phi$ before. Select random elements $X(s,t)\in [0,1]$ for all $s,t\in P_i$. Then $S^{j}_i \subset P_i^j$ and $S^{\leq j}_i = S_i^1 \cup \cdots \cup S_i^j$ are defined recursively with respect to $j$, where an element $s \in P_i^j$ is added to $S^j_i$ exactly when $$\rho(d(s,t))\geq X(s,t)$$ for all $t \in P^j$ with $f_i(t)<f_i(s)$, and for all $t\in S^{\leq j-1}_i$. Note that since $ \rho(t)=1 \text{ if } t \geq r_1$, the decision to include $s \in S^j_i$ only depends on points in the $r_1$-neighborhood of $s$. So, the decision to include a point $s$ in the set $S^{\leq j}_i$ depends only on the $j\cdot r_1$-neighborhood of $s$. 

Fix some $R> d_Z(s_\infty,p_\infty) + j\cdot r_1$. Since $f_\infty$ is injective and $P_\infty^R := P_\infty \cap B_{M_\infty}(p_\infty,R)$ is finite, for all large $i$ we have
$$f_\infty(s) < f_\infty(t) \ \ \iff \ \ f_i(\phi_i(s)) < f_i(\phi_i(t))$$
for all $s,t\in P_\infty^R$. Then the probability that $s_\infty$ in $S^{\leq j}_\infty$ can be calculated using the same computation from the order on and the distance between elements of $P_\infty^R$. For large $i$ the probability that $\phi_i(s_\infty)$ in $S^{\leq j}_i$ is calculated in the same way from 
$\phi_i(P_\infty^R)$, which contains all points of $P_i$ within a distance of $j\cdot r_1$ from $\phi_i(s_\infty)$. The order on $P_\infty^R$ agrees with that on its $\phi_i$-image, and for large $i$ the distances between points of $P_\infty^R$ almost agree with the distances between their $\phi_i$-images. From this, the claim follows.
\end {proof}

Returning to the proof of the lemma, the measure $\nu_\infty$ is supported on points of the form \begin{equation}(M_\infty,\vol_\infty,p_\infty,E_\infty,S_\infty) \in \mathbb{MS}^{ext}(Z), \ \ S_\infty \subset P_\infty. \label{supported on points}\end{equation}
If $(\epsilon,R)$ is an admissible pair for such a point, let $\mathcal B_{\epsilon,R}(S_\infty)$ be the neighborhood that is constructed in Lemma \ref{bijectionlem}, i.e.\ the set of all $(M,\vol,p,E,S) \in \mathbb{MS}^{ext}(Z)$ such that there is a bijection from $S_\infty \cap B_{M_\infty}(p_\infty,R)$ to $S \cap B^\bullet_{M}(p,R)$ that matches points that are within $\epsilon$ of each other in $Z$. As $(\epsilon,R)$ varies, these sets form a neighborhood basis for the point in \eqref{supported on points}. So by the Portmanteau Theorem, to prove $\nu_i \to \nu_\infty$ it suffices to show that \begin {equation}
\nu_i(\mathcal B_{\epsilon,R}(S_\infty)) \to \nu_\infty(\mathcal B_{\epsilon,R}(S_\infty)), \label{nuconv}
\end{equation}
But for large $i$, we have that $$(M_i,\vol_i,p_i,E_i,S_i) \in \mathcal B_{\epsilon,R}(S_\infty) \ \ \iff \ \ S_i \cap B^\bullet_{M_i}(p_i,R) = \phi_i(S_\infty \cap B_{M_\infty}(p_\infty,R)),$$ since for such tuples and large $i$ the bijection in the definition of $\mathcal B_{\epsilon,R}(S_\infty)$ must be the restriction of $\phi_i$. This is equivalent to saying that for all $s_\infty \in P_\infty \cap B_{M_\infty}(p_\infty,R)$, we have $s_\infty \in S_\infty$ if and only if $\phi_i(s_\infty) \in S_i$. But then \eqref{nuconv} follows from Claim \ref{probability claim} via inclusion-exclusion.
 \end{proof}

\subsubsection{Random nerve complexes.} \label{randomnerve}Above, we defined the space $\mathbb{MS}^{ext}$ of pointed, extended special mm-spaces  with distinguished discrete subsets. Here, we explain how to construct a random nerve complex from certain elements of $\mathbb{MS}^{ext}$, in a way that depends continuously on the input. Consider the subset $$(\mathbb {MS}^{ext})' \subset \mathbb{MS}^{ext}$$ consisting of all tuples $(M,\vol,p,E,S)$  such that  \emph{there is a unique element of $S$  that is closest to $p$}. Given some  such tuple, construct a simplicial complex by choosing independently and Lebesgue-randomly a number $\rho(x)\in[r_2,r_3]$ for each $x\in S$ and taking the nerve $N_E(S,\rho)$ of the collection of balls $B_{E}(x,\rho(x))$, where $x\in S$. Note that the balls are in $E$, not in $M$. The unique element of $S$  closest to $p$  is a natural base point for the nerve, and if $K$ is  the \emph{connected component} of $N_E(S,\rho)$  containing $p$ then we have a map
\begin{equation}
	\label {eq22} (\mathbb {MS}^{ext})' \longrightarrow \mathcal P(\mathcal K), %\ \ (M,\vol,p,E,S) \mapsto \nu_{(M,\vol,p,E,S)},
\end{equation}
where $\mathcal K$  is the space of all pointed, finite degree simplicial complexes, $\mathcal P(\cdot)$ denotes the space of probability measures, and the map sends $(M,\vol,p,E,S)$ to the (law of the) random pointed complex $(K,p) $ described above. Note that properties (1) and (2) in the statement of the theorem imply that there is a universal degree bound for all constructed $K$, that is independent of $n$.

\begin{claim}\label{eq22cont}
	The map in \eqref{eq22} is continuous.
\end{claim}

\begin{proof} 
In his Claim 1 at the beginning of the proof of Theorem 4.1, Bowen shows that a variant of \eqref{eq22} is continuous. Here are the discrepancies with our version. First, Bowen works only with $[r_2,r_3]=[5r_0,6r_0],$ but his argument obviously generalizes.  He also does not use extended mm-spaces, and so the map he constructs is from the subset $\mathbb {MS'} \subset\mathbb{MS}$ of all $(M,p,\vol,S)$ where there is a unique closest point in $S$ to $p$, and the balls he uses in constructing the nerve are just in $M$, not in some larger space.  However, the quick proof of continuity works verbatim  in our  extended setting: basically all that is used is that the topology comes from pointwise Hausdorff convergence of the mm-spaces $M$, so since we are putting the same topology on the super-sets $E$, the argument extends. 
\end{proof}

\subsubsection{Convergence of  normalized Betti numbers.} \label {conversion section} We now begin on the main argument for the proof of Theorem \ref{Bowen2}. Most of the ideas below are from \cite{Bowencheeger}, altered so that we use our $S^{\leq j}$  instead of the Elek/Bowen random nets.

 Let $ \mathfrak M _n=(M_n ,\vol_n ,E_n )$ be as in the theorem statement, and fix $\epsilon>0$. We will show 
\begin{equation}\label{thegoal}\limsup_{n\to \infty} \frac{B_n}{\vol_n(M_n)} - \liminf_{n\to \infty}  \frac{B_n}{\vol_n(M_n)} \leq C \cdot \epsilon	
\end{equation}
for some $C$ depending only on the constants in the statement of the theorem. Since $\epsilon$ is arbitrary,  this will suffice to prove convergence of $B_n/\vol_n(M_n)$.

 Pick $j=j(\epsilon,r_0,r_1,v_{min},v_{max})$  as in Proposition \ref{uncoveredprop}, and as in \S \ref{contnet} let  $\sigma_n$ be the law of the random almost-net  $S^{\leq j} \subset M_n$  constructed in \S \ref{almostnets}.  For simplicity in notation, we'll drop the superscript $j$ below and just write $S_n$ for a $\sigma_n$-random subset of $M_n$.  Note that by our choice of $j$, the almost-net $S_n$  can  always be extended to a $(r_0,2r_1)$-net $T_n \subset M_n$  in such a way that  for every $n$, we have \begin{equation}\label{morepoints}E\big [|T_n\setminus S_n| \big ] \leq \epsilon \cdot \vol_n(M_n).\end{equation}

Fix some almost-net $S_n$ and an extension $T_n$ as above. If $\rho_n$ is a function that assigns some $\rho_n(x)\in [r_2,r_3]$ to each $x\in S_n$, let $N_{E_n}(S_n,\rho_n)$ be the  nerve complex of the collection of balls $B_{E_n}(x,\rho_n(x)),$ where for each $x\in S_n$ the radius  is chosen Lebesgue-randomly, just as in \S \ref{randomnerve}.  Extending $\rho_n$ to $T_n$ arbitrarily, define $N_{E_n}(T_n,\rho_n)$  similarly. Then $N_{E_n}(S_n,\rho_n)$ is a full subcomplex of $N_{E_n}(T_n,\rho_n)$. By conditions (1) and (2) in the statement of the theorem, \emph{the degrees of these complexes are bounded by some universal constant $D$}. (Indeed, any $x\in T_n$ represents a vertex of the nerve that is connected only to vertices $y\in T_n$ such that $d(x,y) \leq 2r_3$. Since points in $T_n$ are $r_0$-separated, the degree of $T_n$ is then bounded by the number of $r_0/2$-balls one can pack into a $(2r_3+r_0/2)$-ball in $M_n$. The volume bounds in (1) and (2)  imply that this packing constant is  universally bounded.) Since $S_n $  is a full subcomplex of $T_n$, and $deg(T_n)$ is universally bounded,  it follows that the  total number of \emph{simplices} in $T_n\setminus S_n$ is at most $C\cdot  |T_n\setminus S_n|$, for some fixed constant $C$ depending only on the constants in the statement of the theorem. Hence,  we also have
\begin{equation}
\label{closebetti}	\big | b_k(N_{E_n}(T_n,\rho_n)) - b_k(N_{E_n}(S_n,\rho_n)) \big | \leq C\cdot |T_n\setminus S_n|
\end{equation}
 for such a  constant $C$, by Mayer-Vietoris.

Now condition (3) in the statement of the theorem implies there are $\delta_n \to 0$ such that \begin{equation}\label{new3}\big | b_k(N_{E_n}(T_n,\rho_n)) - B_n \big | \leq \delta_n \vol_n(M_n)\end{equation}
for every sequence of nets $T_n$ as above, and any radii $\rho_n$. (A priori, maybe this looks like a stronger statement, but  if  it were not true, we could apply (3) to a sequence of weighted nets with Betti numbers maximally different from $B_n$ to get a contradiction.) It follows from \eqref{closebetti} and \eqref{new3} that for any $S_n$ and any choice of radii $\rho_n$, we have
\begin{equation}
\label{almost!}
	|b_k(N_{E_n}(S_n,\rho_n))-B_n| \leq C\cdot |T_n\setminus S_n| + \delta_n \vol_n(M_n).
\end{equation}
 So, if we $\sigma_n$-randomly chooses the almost-nets $S_n$ and randomly choose the radii $\rho_n(x) \in [r_2,r_3]$  independently and uniformly for each $x\in S_n$, then we have
\begin{align*}\label{expestimate} 
\big | E\big [b_k(N_{E_n}(S_n,\rho_n))\big ] -B_n \big | &\leq  E\big [\big | b_k(N_{E_n}(S_n,\rho_n)) -B_n\big |\big ] \\ 
& \overset{}{\leq} C \cdot E[|T_n\setminus S_n|] +\delta_n \vol_n(M_n) \\
& \leq C\cdot \epsilon \cdot \vol_n(M_n)+\delta_n \vol_n(M_n). \end{align*}
where the second inequality is \eqref{almost!} and the third is \eqref{morepoints}. So, to prove that the liminf and limsup of $B_n/\vol_n(M_n)$ are within $C\epsilon$ of each other,  which is the goal we set in \eqref{thegoal}, it suffices to prove the following claim.

\begin{claim}\label {thegoal2}
The ratio $E\big [b_k(N_{E_n}(S_n,\rho_n))\big ]/\vol_n(M_n)$ converges as $n\to \infty$.
\end{claim} 

We now show how to use Elek's Theorem \ref{elek} to reduce Claim \ref{thegoal2} to two other convergence claims.  For simplicity, let $\tau_n$ be  the law of a $\sigma_n$-random $S_n \subset M_n$ equipped with uniformly and independently chosen radii $\rho_n(x)\in [r_2,r_3]$ at each $x\in S_n$.  So, the expectation in \eqref{almost!}  is taken with respect to $\tau_n$.  As in the introduction, let $\mathcal K$  be the space of pointed complexes and  consider the probability measure \begin{equation}
 	\eta_n:= \frac{ \int \mu_{N_{E_n}(S_n,\rho_n)} \, d\tau_n}{\int |S_n| \, d\sigma_n} \in \mathcal P(\mathcal K).\label{etan}
 \end{equation}
 Note that an $\eta_n$-random pointed complex is \emph{not} produced by $\tau_n$-randomly choosing $(S_n,\rho_n)$ and then choosing  a base point for the corresponding nerve complex uniformly randomly.  For the  cardinalities $|S_n|$ can vary  depending on the particular subset, and the nerve complex of a particular $S_n$  appears more often with respect to $\eta_n$ if $|S_n|$ is larger.  Intuitively, one should think that $\eta_n$ assigns `equal weights' to all the vertices in all the different complexes, rather than weighting the (unpointed) complexes themselves `equally'.

Now, the randomly produced nerve complexes $N_{E_n}(S_n,\rho_n)$ all have volume at most some universal constant times $\vol_n(M_n)$. (Indeed, the points of $S_n$ are $r_0$-separated and by condition (2) in the statement of the theorem, we have a lower bound on the  volume of every $r_0/2$-ball in $M_n$.)  So, for a fixed $n$ the nerve complex takes on only finitely many  isomorphism types.  In other words, the measure $\eta$  is really a \emph{finite} linear combination of the measures $\mu_K$ associated to certain unpointed complexes $K$,  as in Lemma \ref{elekimproved}, which was a slightly more general version of Elek's Theorem \ref{elek}.

\begin{claim} \label {thing1}
 The probability measures $\eta_n$ weakly converge in $\mathcal P(\mathcal K)$.
\end{claim}

 Assuming this for a moment, we conclude from Lemma \ref{elekimproved} that the ratios
$$\frac{E[b_k(N_{E_n}(S_n,\rho_n))]}{E[|S_n|]}$$
converges for all $k$.  We will also show the following.

\begin{claim}\label {thing2}
	The ratio $E[|S_n|]/\vol_n(M_n)$ converges.
\end{claim}

Assuming both Claims \ref{thing1} and \ref{thing2}, we then have that the ratio
$$\frac{E\big [b_k(N_{E_n}(S_n,\rho_n))\big ]}{\vol_n(M_n)} = \frac{E[b_k(N_{E_n}(S_n,\rho_n))]}{E[|S_n|]} \cdot \frac{E[|S_n|]}{\vol_n(M_n)} $$
 converges, as desired, proving Claim \ref{thegoal2} and hence Theorem \ref{Bowen2}.

\medskip

Claims \ref{thing1} and \ref{thing2} will be proved simultaneously. For Claim \ref{thing1}, the point is to use that the extended mm-spaces $\mathfrak M_n$ BS-converge, and to translate that into convergence of the measures $\eta_n$.  Since BS-convergence relies on randomly picking base points, one would like to relate randomly chosen basepoints in $M_n$ to randomly chosen points of $S_n$, which can be used as base points of the associated nerve complexes.  Intuitively, the idea is just to associate to a point $p\in M_n$ the point of $S_n$ closest to $p$. However, there might not be a unique such closest point, and even if there is, the set of points in $M_n$ closest to some $p\in S_n$ may have much different volume from the set of points closest to some other $q\in S_n$, so a Lebesgue-random basepoint in $M_n$ may not correspond to a uniformly random point in $S_n$. The idea, then, is to show that we still obtain a weakly convergent sequence of measures if instead of choosing Lebesgue-random basepoints from $M_n$, we only choose them from small fixed-volume balls around the points of a random $S_n$. This convergence will translate directly into convergence of the measures $\eta_n$ above. 

To formalize this idea above, we work measure theoretically instead of probabilistically. For each $n$, let $\sigma_n$ be the law of $S_n$ and let $\lambda_n $  be the measure on $\mathbb {MS}^{ext}$ (see \S \ref{contnet}) obtained by   pushing forward the product measure $(\vol_n /\vol_n(M_n)) \times \sigma_{n }$ under the map \begin{equation}
 \label {productmeas}	M_n  \times \{ \text{ discrete } S  \subset M_n \ \} \longrightarrow \mathbb {MS}^{ext}, \ \ \  (p,S) \mapsto (M_n,\vol_n,p,E_n,S).
 \end{equation}
This $\lambda_n $ can also be obtained by pushing forward $\mu_{\mathfrak M_n } / \vol_n(M_n)$ via the continuous map $$\ \mathbb M_{sp}^{ext} \longrightarrow \mathcal P(\mathbb {MS}^{ext})$$ from \eqref{sigmamap} then taking the expected value. In symbols,
$$\mu_{\mathfrak M_n }\in  \mathcal P(\mathbb M_{sp}^{ext}) \overset{\eqref{sigmamap}}{\longrightarrow} \mathcal P( \mathcal P(\mathbb {MS}^{ext})) \overset{E[\cdot ]}{\longrightarrow} \mathcal P(\mathbb {MS}^{ext}) \ni  \lambda_n  $$
 Since both the maps in this composition are continuous and $(\mu_{\mathfrak M_n})$ is weakly convergent, it follows that $(\lambda_n)$ converges weakly to some  probability measure $\lambda_\infty$ on $\mathbb {MS}^{ext}$.

Fix some $v$ with $0 < v < v_{min}$, where $v_{min}$ is a lower bound for the volume of any $r_0/2$-ball in any $M_n$, as in condition (1) of the theorem. If $S \subset M$ is a discrete set, let $S(v)$ be the union of all volume-$v$  closed balls in $M$ that are centered at points of $S$. Let $$\mathbb{MS}^{ext}(v)\subset \mathbb{MS}^{ext}$$ be the closed subset consisting of all  tuples $(M,\vol,p,S,E)$ such that $p\in S(v)$. 

Since $v<v_{min}$, the volume-$v$ balls around the points of any \emph{$r_0$-separated} set like $S_n \subset M_n $  are all disjoint, so we have $\vol_n(S_n(v)) = v |S_n |.$ Given $S_n$,  the probability then that a randomly chosen $p\in M_n$ ends up in $S_n(v)$ is $$v|S_n|/\vol_n(M_n).$$  Integrating over $S_n$, we have for $n=1,2,\ldots$ that
\begin{equation}\label{lnmeas}
	\lambda_n (\mathbb{MS}^{ext}(v)) = v\cdot \frac {E[ |S_n | ]}{\vol_n(M_n)},
\end{equation}
where $S_n$ is chosen $\sigma_n$-randomly.  
\begin{claim}\label {primeconverge}
We have $\lambda_n (\mathbb{MS}^{ext}(v)) \to \lambda_\infty(\mathbb{MS}^{ext}(v))$.
\end{claim}

\begin{proof} Fix some very small $\epsilon>0$. As mentioned above, the sets $S_n$ are $r_0$-separated, so the ratios $|S_n|/\vol_n(M_n)$ are bounded above by some $C$ depending on the constants in properties (1) and (2) in the theorem. So if we set $a = v + \epsilon/C$, it follows from \eqref{lnmeas} that  \begin{align*}
\lambda_n\big (\mathbb{MS}^{ext}(a) \big ) &\leq a \frac{E[ |S_n | ]}{\vol_n(M_n)} \\ & \leq \frac{v_{min}}{2} \frac{E[ |S_n | ]}{\vol_n(M_n)} + \epsilon \\ 
&= \lambda_n\big (\mathbb{MS}^{ext}(v) \big ) + \epsilon.
\end{align*}
Since $\mathbb{MS}^{ext}(v)$ is closed and the interior of $\mathbb{MS}^{ext}(a)$ contains $\mathbb{MS}^{ext}(v)$, it then follows from the Portmanteau theorem that
\begin{align*}\limsup \lambda_n\big (\mathbb{MS}^{ext}(v) \big ) &\leq \lambda_\infty(\mathbb{MS}^{ext}(v)) \\ & \leq \liminf \lambda_n\big (\mathbb{MS}^{ext}(a) \big )\\ & \leq \liminf \lambda_n\big (\mathbb{MS}^{ext}(v) \big ) + \epsilon.
\end{align*}	
And then the claim follows, since $\epsilon$ was arbitrary.
\end{proof}

Now let $\lambda'_n$  be the probability measure on $ \mathbb{MS}^{ext}(v)$  that we get by normalizing the restriction of $\lambda_n$.  %Since  the restriction of $\lambda_n$ to $\mathbb{MS}^{ext}(v)$ is obtained by integrating the Lebesgue measure on $S_n(v)$ against $\sigma_n$, and $\vol_n(S_n(v))=v|S_n|$, a $\lambda_n' $-random $5$-tuple $(M_n,\vol_n,p,S_n,E_n)\in \mathbb{MS}^{ext}(v)$ is obtained by choosing a $|\cdot |\sigma_n$-random $S_n\subset M_n$ and then choosing a point Lesbesgue-randomly from $S_n(v)$. Here, $|\cdot |\sigma_n$  is the measure on  the set of discrete subsets of $M_n$ obtained by integrating the cardinality function against $\sigma_n$. 
If $C \subset \mathbb{M S}^{ext}(v)$ is closed, then $C$ is also closed in $\mathbb{MS}^{ext}$. So by Claim~\ref{primeconverge} and the Portmanteau Theorem,  we have that
$$\limsup \lambda'_n( C ) = \limsup \frac{\lambda_n( C )}{\lambda_n(\mathbb{MS}^{ext}(v))} \leq \frac {\lambda_\infty(C)}{\lambda_\infty(\mathbb{MS}^{ext}(v))} = \lambda_\infty'(C).$$
Applying the Portmanteau Theorem again, we see that $\lambda_n' \to \lambda_\infty'$ weakly.

Since all the $S_n$ are $r_0$-separated, if $v$ is sufficiently small relative to the constants in properties (1) and (2) in the theorem, the basepoint $p$ of $\lambda_n'$-almost every $5$-tuple $$(M_n,\vol_n,p,S_n,E_n)\in \mathbb{MS}^{ext}(v)$$ is closest in $M_n$ to the element  $q\in S_n$ that is the center of the volume-$v$ ball in which $p$ lies. So, if $\mathcal K$ is the space of pointed complexes, the random nerve complex map $$\mathbb {MS}^{ext}(v) \longrightarrow \mathcal P(\mathcal K)$$
 one gets by restricting the map in \eqref{eq22} takes the $5$-tuple above to a measure whose random element is obtained by picking a random $\rho_n$ and then taking the connected component of $N_{E_n}(S_n,\rho_n)$ that is rooted at the vertex $q\in S_n$ in whose volume-$v$ ball $p$ lies. 

Now, the restriction of $\lambda_n$ to $\mathbb{MS}^{ext}(v)$ is obtained by integrating the Lebesgue measure on $S_n(v)$ against $\sigma_n$. So, the image of $\lambda_n'$ under the composition
\begin{equation}\label{compppp}
	\mathcal P(\mathbb {MS}^{ext}(v)) \longrightarrow \mathcal P(\mathcal P(\mathcal K)) \overset{E[\cdot]}{\longrightarrow} \mathcal P(\mathcal K)
\end{equation}
is a probability measure on $\mathcal K$ obtained by integrating the counting measure on the vertices of $N_{E_n}(S_n,\rho_n)$ against the measure $\tau_n$ that is the law of the random weighted nets $(S_n,\rho_n)$. In other words, $\lambda_n'$ pushes forward to the measure $\eta_n$ from \eqref{etan}. By Claim \ref{eq22cont}, the composition \eqref{compppp} is continuous, so the fact that the $\lambda_n'$ weakly converge means that the $\eta_n$ also weakly converge. This proves Claim \ref{thing1}. Claim \ref{thing2} follows immediately from \eqref{lnmeas} and Claim \ref{primeconverge}, so our proof of Theorem \ref{Bowen2} is done. \qed

\subsection{A variation of Theorem \ref{Bowen2}}
The following variant of Theorem \ref{Bowen2} will serve us in the sequel:

\begin{corollary}
\label{Bowen3}Fix $k$, let $\mathfrak M_n=(M_n,\vol_n,E_n)$  be a sequence of extended finite volume special mm-spaces and  assume that for some  sequence of constants $V_n$, the measures $\mu_{\mathfrak M_n}/V_n$ weakly converge to some finite measure $\mu$ on $\mathbb{M}^{ext}$. Pick constants $v_{min}>0$, $r_1 >r_0 >0$, and  $r_3 >r_2 \geq 2r_1, $ and a function $v_{max} : \BR_+ \longrightarrow \BR_+$  such that 
\begin{enumerate}
	\item all $r_0/2$-balls in every $M_n $ have volume at least $v_{min}$, 
\item for all $r \in \BR_+$, every $r$-ball in $M_n$ has volume at most $v_{max}(r)$,
\item for every sequence of  $[r_2,r_3]$-weighted $(r_0,2r_1)$-nets $(S_n,\rho_n)$ in $ M_n$, $$\frac{\big | b_k(N_{E_n}(S_n,\rho_n)) - B_n \big | }{V_n}\to 0.$$
\end{enumerate}
Then the ratios $B_n/V_n$  converge.
\end{corollary}

Note that here, the measures $\mu_{\mathfrak M_n}/V_n$ and their limit may not be probability measures.

\begin{proof}[Proof of  Corollary \ref{Bowen3} given Theorem \ref{Bowen2}]
 Let $\mu$  be the weak limit of $\mu_{\mathfrak M_n}/V_n$.  Then $$\lim_{n\to \infty} \frac { \vol_n(M_n)}{V_n} = \lim_{n\to \infty} \int 1 \ d(\mu_{\mathfrak M_n}/V_n) = \mu(\mathbb{M}^{ext}) \in [0,\infty).$$

Suppose first that $\mu(\mathbb M^{ext})=0$. By (1) and (2),  the number of points in any $(r_0,2r_1)$-net in $M_n$  is comparable to $\vol_n(M_n)$, so the Betti numbers in (3) are  $O(\vol_n(M_n))$. Combining (3) and the triangle inequality,   we have $B_n/V_n \to 0$.

If $\mu(\mathbb M^{ext})>0$, then $V_n/\vol_n(M_n)$ has a finite limit, so the  probability measures $$\frac{\mu_{\mathfrak M_n}}{\vol_n(M_n) }= \frac{V_n}{\vol_n(M_n)}  \cdot \frac{\mu_{\mathfrak M_n}}{V_n}\to \frac \mu{\mu(\mathbb M^{ext})} .$$  In other words, the extended mm-spaces $\mathfrak M_n$ BS-converge. Also, $\vol_n(M_n)$ and $V_n$ are of bounded ratio, so (3) holds with $\vol_n(M_n)$  instead of $V_n$. Theorem \ref{Bowen2} then  says that $B_n/\vol_n(M_n)$  converges, from which it follows that $B_n/V_n$  converges too.
\end{proof}

\section{Pinched negative curvature and Theorem \ref{pnc}}
\label{pncsec}

In this section, we consider only $d$-manifolds $M$   with  sectional curvature $$-1 \leq K \leq -a^2 <0,$$ and we let $\epsilon(d)$ be the corresponding  $d$-dimensional Margulis constant. For any $\epsilon\leq \epsilon(d)$, each component of the $\epsilon$-thin part $(M_n )_{\leq \epsilon}$ is either:
\begin{itemize}
\item a {\it Margulis tube}, which is (topologically) a tubular neighborhood of a closed geodesic, and so is homeomorphic to a ball bundle over the circle, or
\item a {\it cusp neighborhood}, which is homeomorphic to $S \times [0,\infty)$  for some compact aspherical $(d-1)$-manifold $S$ with virtually nilpotent fundamental group.
\end{itemize}
See for instance \cite[\S 8]{Ballmannmanifolds} for a proof.

 In the introduction, we explained how to produce  BS-convergent sequences $(M_n)$ of hyperbolic $3$-manifolds where the normalized Betti numbers do not converge, using Dehn filling.  In the example we gave, the volumes $\vol(M_n)$  were bounded, but one can construct similar examples with unbounded volumes by filling  the complements of links with unboundedly many components, instead of a fixed knot complement.  Instead of doing the details of this approach, though, we'll briefly describe a similar example in which the BS-limit is easier to understand.

\begin {example}\label {3dimex}
Let $M$ be the mapping torus of a homeomorphism $\phi : S \longrightarrow S$,  where $\phi$ is a pseudo-Anosov homeomorphism of some closed surface $S$ with genus at least $2$.  So, $M$  comes with a fibration $M \longrightarrow S^1$. Identify $S$  with a fiber  of this fibration, and let $\gamma $ be a simple closed curve on $S$. By Thurston's Hyperbolization Theorem \cite{Kapovichhyperbolic}, the manifold $$M(\infty) :=M \setminus \gamma$$ admits a finite volume hyperbolic metric. 

Let $M(k)$ be  the closed $3$-manifold obtained from $M(\infty)$ by $(1,k)$-Dehn filling\footnote{Here, we use meridian-longitude coordinates to parametrize the boundary of a cusp neighborhood, where the meridian is the curve that was homotopically trivial before we drilled out $\gamma$.}. For large $k$,  Thurston's Dehn Filling Theorem \cite{Benedettilectures}  implies that $M(k)$  admits a hyperbolic metric;  moreover, as $k\to\infty$ the manifolds $M(k)\to M(\infty)$ geometrically.  Note that since we are doing $(1,k)$ filling, each $M(k)$ is also a genus $g$ mapping torus. Indeed,  if $T_\gamma$  is a Dehn twist around $\gamma$, the monodromy map of $M(k)$ is $T_\gamma^k \circ \phi$.

For $k\in \BN\cup \{\infty\}$, let $M_n(k)$ be the degree $n$ cyclic cover of $M(k)$  corresponding to the subgroup of $\pi_1 M(k)$ that is the preimage of $n\BZ \leq \BZ \cong \pi_1(S^1)$  under the map induced by $$M(k) \hookrightarrow M \longrightarrow S^1.$$ Then for every $n$ and $k<\infty$, the manifold $M_n(k)$ is a  mapping torus over a genus $g$ surface, and hence $$b_1(M_n(k))\leq 2g+1.$$
On the other hand, setting $k=\infty$ the manifold $M_n(\infty)$ has $n$  cusps, so we have $$b_1(M_n(\infty)) \geq n.$$  

Now set $k=n$. As $n\to \infty$, the sequences $M_n(n)$ and $M_n(\infty)$ both BS-converge to the same limit measure $\mu$ on $\mathcal M$. This $\mu$ is supported on pointed manifolds isometric to the {infinite} cyclic cover $M_\infty(\infty)$ of $M(\infty)$  corresponding to the kernel of the map on fundamental groups induced by $M(\infty)\longrightarrow S^1$; more carefully, $\mu$ is the push forward of the  normalized Riemannian  measure on $M(\infty)$ under the map
$$M(\infty) \longrightarrow \mathcal M, \ \ p \longmapsto [(M_\infty(\infty),p_\infty)],$$
where $p_\infty$  is any point that projects to $p$  under the covering map $M_\infty(\infty) \longrightarrow M(\infty)$. (This is a special case of the construction in Example 2.4 in \cite{Abertunimodular}.) 
However, 
$$b_1(M_n(n)) /\vol(M_n(n)) \to 0, \ \ b_1(M_n(\infty)) /\vol(M_n(\infty)) \not \to 0.$$
\end {example}

 Essentially, the reason why Dehn filling is problematic is that from the perspective of most points in a manifold, a Margulis tube with very small core length can look nearly identical to a rank two cusp.  (One can only see the difference if one is close enough to be able to distinguish the core geodesic of the tube, and when the core length is small, the set of points a bounded distance from the core has very small volume.)   This coincidence is particularly three-dimensional, though.  For instance, note that the  boundary of a $d$-dimensional Margulis tube is a $S^{d-2}$-bundle over $S^1$, while the boundary of a cusp neighborhood is a Euclidean $(d-1)$-manifold. If $d=3$, the torus $T^2$  satisfies both descriptions, but when $d\geq 4$, $S^{d-2}$-bundles over $S^1$ are not aspherical, so cannot be Euclidean.

\vspace{2mm}

 The plan for the rest of \S \ref{pncsec} is as follows. In \S \ref{lower bounds}  we show that Margulis tubes with short cores have large volume in dimension at least $4$.  In \S \ref{thickapproxsec} we adapt some of  Gelander's work in \cite{Gelanderhomotopy}, showing that one can approximate (shrinkings of) the $\epsilon$-thick parts of manifolds with pinched negative curvature with certain nerve complexes. And then finally, in \S \ref{pfsec} we  prove  Theorem \ref{pnc}.

\subsection{Lower volume bounds for Margulis tubes} \label{lower bounds} As mentioned above, the basic idea in Theorem \ref{pnc} is to show that the number of Margulis tubes with very short cores that appear in a manifold with pinched negative curvature is a very small fraction of its volume.  To  verify this, we  will use the following proposition.

\begin{prop}[Short geodesics imply large volume] Let $d\geq 4$
and let $M $ be a  complete Riemannian $d$-manifold with  sectional curvatures in the interval $[-1,-a^2]$, where $a>0$. Suppose that $T \subset M_{\leq \epsilon}$ is a component of the $\epsilon$-thin part of $M $ whose core geodesic has length $\ell$. Then $\vol(T) \geq C:=C(d,a,\epsilon,\ell),$
where $C \to \infty$ as $\ell \to 0$.\label {shortgeodesics}
\end{prop}

By Wang's finiteness theorem \cite{Wangtopics}, for $d\geq 4$ a finite volume hyperbolic $d$-manifold $M$ can only have a very short geodesic if its volume is very large.   So  for hyperbolic  manifolds, one can think of the above as a strengthening of this statement that says that the large volume has to come from the Margulis tube around the short geodesic.

One could probably  prove (at least a version of)  the proposition  using a geometric limit argument informed by the above discussion on Dehn filling. We assume there is a sequence of manifolds and Margulis tubes $T_n \subset M_n$ where the core length $\ell_n \to 0$, but where $\sup \vol(T_n) < \infty.$ Take base points $p_n \in \partial T_n$ and extract pointed Gromov-Hausdorff limits of everything, giving $T_\infty \subset M_\infty$ and $p_\infty \in \partial T_\infty$. Since $\ell_n \to 0$, this $T_\infty$  is a cusp neighborhood, rather than a Margulis tube. And since $\sup \vol(T_n) < \infty$, one can argue that the  diameter of $\partial T_n$ is bounded, which means that $\partial T_\infty$ should actually be homeomorphic to $\partial T_n$. But as mentioned above, this is impossible since the boundary of a cusp neighborhood is always aspherical, but the boundary of a Margulis tube is not if $d\geq 4$. 

 We chose not to use the geometric limit approach because pushing through the limiting arguments requires control over higher order derivatives of the metric tensors, which we do not necessarily  want to include in the statement of Theorem \ref{pnc}.  Also, the proof we give below is attractive in that one could use it to write down an explicit formula for $C$.

\vspace{2mm}

 Before starting the proof of Proposition \ref{shortgeodesics}, we establish the following simple lemma.

\begin {lemma}\label {onlem}
Suppose that $n\geq 3$ and $A \leq O(n)$ is an abelian subgroup, which we consider as acting on the unit sphere $S^{n-1}$ by isometries. Then $\diam(G \backslash S^{n-1}) \geq \pi/2$. \end {lemma}

 Here, the distance between two points in the quotient is the minimal distance in $S^{n-1}$ between points in their preimages. Note that $G\backslash S^{n-1}$ is a path metric space.

\begin {proof}
The subgroup $A$ is contained in a subgroup $\BT \leq O(n)$ of the form 
$$\BT = \begin{pmatrix}
	O(2)  & &  \\
& \ddots &   \\
& & O(2)   
\end{pmatrix} \text{ or } \BT = \begin{pmatrix}
	O(2) & & &  \\
& \ddots & &   \\
& & O(2) &  \\
& & & \pm 1  
\end{pmatrix},$$
written in suitable orthonormal coordinates $(x_1,\ldots,x_n)$ for $\BR^{n}$,  depending on whether $n$ is even or odd. But in these coordinates, the action of $A$ preserves the intersection $I$ of $S^{n-1}$ with the $x_1x_2$-coordinate plane, and it also preserves the intersection $J$ of $S^{n-1}$ with either the $x_3x_4$-coordinate plane or the $x_3$-axis, depending on whether $n\geq 4$ or $n=3$. The distance in $S^{n-1}$ between $I$ and $J$ is $\pi/2$, so the lemma follows.\end{proof}

\begin {proof}[Proof of Proposition \ref{shortgeodesics}]
Pick a universal covering map $\tilde M \longrightarrow M$ and lift the core geodesic $\gamma\subset T$ to a complete geodesic $\tilde \gamma \subset \tilde M$. Let $\tilde T$ be the component of the preimage of $T$ that contains $\tilde \gamma$, and let $g : \tilde M \longrightarrow M$ be a nontrivial deck transformation stabilizing $\tilde \gamma$ that is primitive in the deck group. So, $g$  is determined up to inversion,  the cyclic group $\langle g \rangle$ is the stabilizer of $\tilde T$, and any deck transformation not in $\langle g \rangle$ moves $\tilde T$ completely off itself. 

Pick a point $\tilde p\in \tilde \gamma$ and isometrically identify the fiber $N^1(\tilde\gamma)_{\tilde p}$ of the  unit normal bundle of $\tilde \gamma$ with $S^{n-2}$. Parallel transport then determines a global trivialization
$$\tilde \gamma \times S^{n-2} \longrightarrow N^1(\tilde \gamma),$$
and we can then write the action of $g$ on $N^1(\tilde \gamma)$  in these coordinates as
$$g = \tau \times r,$$
where $\tau$  is a translation by $\ell$ along $\tilde \gamma$ and $r \in O(d-1)$. 

Since $O(d-1)$ is a compact manifold, there is some $c>0$ such that if $S \subset O(d-1)$ is any set of $m^{\dim O(d-1)}$ points, there are $s,t \in S$  such that $$d(s(\xi),t(\xi)) \leq c/ m, \ \ \forall \xi\in S^{d-2}.$$ %(The more natural denominator is $m$, but we can replace it with $m+1$ after doubling $c$, and the latter is more convenient because of the forthcoming $\lfloor$floor$\rfloor$.)
Setting $m=\lfloor \ell^{\frac {-1}{\dim O(d-1) + 1}} \rfloor$, we get that there is some power $r^k, \ k \leq m^{\dim O(d-1)}$ with \begin{equation} d(r^k(\xi),\xi) \leq c/m \leq c \cdot \ell^{\frac 1{\dim O(d-1) + 1}}, \ \ \forall \xi\in S^{d-2}. \label {firstdist}\end{equation}
Note that  we also have \begin{equation}
	\label {seconddist}
 d(\tau^k(\tilde x),\tilde x) \leq \ell \cdot m ^{\dim O(d-1)} \leq \ell \cdot \ell ^{\frac{-\dim O(d-1)}{\dim O(d-1) + 1}} =\ell^{\frac 1{\dim O(d-1) + 1}}, \ \ \forall \tilde x\in \tilde \gamma.\end{equation}

Since $\tilde M$ has sectional curvatures in $[-1,-a^2]$, it follows from the triangle comparison theorems that if two unit speed geodesic segments $\alpha,\beta$ in $\tilde M$ share an endpoint $\alpha(0)=\beta(0)$ at which they intersect with angle $\theta$, then we have that 
\begin {equation}\label {3rd}\theta t \leq d(\alpha(t),\beta(t)) \leq \theta \sinh(t), \ \ \forall t>0.\end {equation} Similarly, by an application of Berger's extension of Rauch's comparison theorem \cite[Theorem 1.34]{Cheegercomparison}, if $\alpha,\beta$ start out with $\alpha(0),\beta(0)\in \tilde \gamma$, and both are perpendicular to $\tilde \gamma$, then \begin {equation}\label {4th}d(\alpha(t),\beta(t)) \leq d(\alpha(0),\beta(0)) \cdot \cosh(t), \ \ \forall t>0.\end{equation} Combining \eqref{firstdist} and \eqref{seconddist}  with the upper bounds in \eqref{3rd}  and \eqref{4th},  and using the decomposition $g=\tau \times r$, we get that for a point $\tilde x\in \tilde M$  that lies at distance $t$ from $\tilde \gamma$, 
$$d(g^k(\tilde x),\tilde x) \leq \ell^{\frac 1{\dim O(d-1) + 1}}\big (c \cdot \sinh(t) + \cosh(t) \big ) \leq 2 \ell^{\frac 1{\dim O(d-1) + 1}}\cdot c \cdot \cosh(t). $$
So if $\ell$ is small, we can set
$L = \cosh^{-1}(\epsilon/(6c\ell^{\frac 1{\dim O(d-1) + 1}})),$ and 
then the $L$-neighborhood of $\tilde \gamma$  will be contained in the subset $\tilde T_{\epsilon/3} \subset \tilde T$  that is the lift of $T \cap M_{\leq \epsilon/3}$.

Since $d\geq 4$, Lemma \ref{onlem} implies that $$\diam(\langle r \rangle \backslash S^{n-2}) \geq \pi /2.$$
Since the quotient is a path metric space, we can then choose  for any $\theta>0$ a set $S$ of $\lfloor \pi/(2\theta)\rfloor$ points in  $S^{n-2}$,  with the property that the $\langle r \rangle$-orbits of  any two distinct points in $S$ are at  least  a distance of $\theta$ from each other in $S^{n-2}$.  Identify $S^{n-2}$  with the  fiber $N^1(\tilde \gamma)_{\tilde p}$  of the unit normal bundle as above, let $\exp_{\tilde p}$ be the Riemannian exponential map, and let $$\mathcal S = \{ \exp_{\tilde p}(L \cdot \xi ) \ | \ \xi \in S\}.$$  
By the lower bound in \eqref{3rd},  we get that the distance between the $\langle g \rangle$-orbits of any two distinct points in $\mathcal S $  is at least $\theta L$.  So taking $\theta = {\epsilon}/{2L},$ the $\langle g \rangle$-orbits of points in $\mathcal S $ are at least $\epsilon/3$ apart in $\tilde M$.

 Now $\tilde T_{\epsilon/3}$ is star-shaped with respect to the geodesic $\tilde \gamma$, so we can now project $\mathcal S \subset \tilde T_{\epsilon/3}$ radially from $\tilde \gamma$ to a subset $\mathcal S' \subset \partial T_{\epsilon/3}$. Since $\tilde M$ has negative curvature, this radial projection cannot decrease the distance between any two points in $\tilde T_{\epsilon/3}$ that are the same distance from $\tilde \gamma$. Hence,  the $\langle g \rangle$-orbits of points in $\mathcal S' $ are still at least $\epsilon/3$ apart in $\tilde M$.  It follows that the covering map $\tilde M \longrightarrow M$ restricts to an embedding on the union of the $\epsilon/3$-balls in $\tilde M$ around the points of $\mathcal S'$. (It is an embedding on each individual ball by definition of $\tilde T_{\epsilon/3}$.) Each of these $\epsilon/3$-balls is contained in $\tilde T$, so volume of $T$ is bounded below by the sum of the volumes of these balls. Each ball has volume at least some $V=V(\epsilon,d,a)$, by the usual comparison arguments, and there are $\lfloor \pi/(2\theta) \rfloor$ balls in total. By our definitions of $\theta$ and $L$,  the number of balls goes to infinity with $\ell$, and the proposition follows. \end{proof}

\subsection {Simplicial approximation of the thick part} \label {thickapproxsec}  Suppose that $M$ is a metric space and  $A \subset M$. Following \cite{Gelanderhomotopy}, we denote the metric $\xi$-neighborhood of $A$ by $(A)_\xi$, and we define the \emph{$\xi$-shrinking} of  $A$ to be the subset 
$$)A(_\xi \ := \ M\setminus  (M\setminus A)_\xi \subset A.$$ 
Fix now $\epsilon,\xi>0$, with $\epsilon$ less than the Margulis constant  $\epsilon(d)$, and let $M$ be a Riemannian  $d$-manifold with curvatures in  $[-1,-a^2]$. The main result of this section is the following, which is an application of techniques of Gelander \cite{Gelanderhomotopy}. Informally, it says that  the shrinking $)M_{\geq \epsilon}(_\xi$ of the $\epsilon$-thick part of $M$ can be simplicially modeled (up to homotopy) by the nerve complex associated to a certain open cover.

\begin{proposition}[(\cite{Gelanderhomotopy})$_\epsilon$]
For any sufficiently small $\epsilon>\epsilon'>0$ and any $c' \geq c \geq 1$, there is a constant $b=b(d,a,\epsilon)>0$ and some small $\delta_0=\delta_0(d,a,c,\epsilon')>0$ such that  the following holds for every $d$-manifold $M$  with curvatures in $[-1,-a^2]$, and all $\delta<\delta_0$.

Set $\xi=\epsilon/2+\delta$ and let $  S$ be a $(\delta,c\delta)$-net in $)M_{\geq \epsilon}(_\xi$. Let $$\rho : S \longrightarrow [(b+c)\delta,(b+c')\delta]$$ be any function and let $N(S,\rho)$ be the nerve of the collection of balls $B_{M}(x,\rho(x)) $, where $x\in S$. Then $N(  S,\rho)$ is homotopy equivalent to $M_{\geq \epsilon}$.\label {gelanderprop}
\end{proposition}

Here, recall from \S \ref{normbmm} that $S $ is a $(\delta,c\delta)$-net if it is $\delta$-separated and $c\delta$-covers. The key to the above is the following restatement of a result from \cite{Gelanderhomotopy}.

\begin{lemma}[essentially Lemma 4.1 in \cite{Gelanderhomotopy}] \label{tsachiklem} Let $M$ be a complete Riemannian $d$-manifold with  sectional curvatures in  the interval $[-1,0]$, let $M' \subset M $ be a  connected submanifold with boundary  and let $\epsilon,b,c,c'>0,$ with $c<c'$,  be fixed. Suppose that 
\begin{enumerate}
	\item $M'$ is contained in the $\epsilon $-thick part of $M$,
\item $M'$ is homotopy equivalent to $)M'(_{\epsilon/2}$,
\item the preimage $\tilde X$ of $X=M \setminus M'$ under a universal covering map $\tilde M \longrightarrow M$  is a locally finite union $\tilde X = \cup_\gamma \tilde X_\gamma$ of convex open sets  with smooth boundary,
\item  for any point $x \in \tilde M \setminus \tilde X$ with $d(x,\tilde X)\leq \epsilon$,  there is a unit  tangent vector $n(x) \in T_x(\tilde M)$  such that for each $\gamma$ with $d(x,\tilde X)=d(x,\tilde X_\gamma)$, we have
$$n(x) \cdot \nabla d( \cdot, \tilde X_\gamma) |_x \geq 1/b.$$
\end{enumerate}
 Then there is some small $\delta_0=\delta_0(\epsilon,b,d)>0$ such that the following holds for all $\delta<\delta_0$. Let $S$ be any $\delta$-separated subset of $)M'(_{\epsilon/2+\gd}$ that $c\delta$-covers $)M'(_{\epsilon/2+\gd}$, and let $$\rho : S \longrightarrow [(b+c)\delta,(b+c')\delta]$$ be a function. Then the nerve of the collection of balls $$\mathcal C = \{B_M(x,\rho(x)) \ | \ x\in S\}$$ is homotopy equivalent to $M'$.\end{lemma}

 We should say that Lemma 4.1 in \cite{Gelanderhomotopy} is not quite stated as above.  The biggest difference is that \emph{maximal} $\delta$-separated subsets of $)M'(_{\epsilon/2+\gd}$ are used in \cite{Gelanderhomotopy} instead of subsets that $c\delta$-cover, and the radii of the balls in the collection $\mathcal C $ are all chosen to be $(b+1)\delta$. However, the proof works just as well if all the $1$'s are replaced by numbers between $c$ and $c'$.  (So for instance, one should use $c'$ instead of $1$ in  Proposition 4.7 of \cite{Gelanderhomotopy}, and allow the radius to vary in  Proposition 4.8.) A purely cosmetic difference is that Lemma 4.1 in \cite{Gelanderhomotopy} is stated for locally symmetric spaces, but local symmetry is not used in its proof.  Finally, the  conclusion of Lemma 4.1 in \cite{Gelanderhomotopy} is that $M'$ is homotopy equivalent to an unnamed simplicial complex, but if one looks at \cite[Proposition 4.8]{Gelanderhomotopy}, one will see that this unnamed complex is just the nerve mentioned above. (The statement of Proposition 4.8 references the cover of $)M'(_{\epsilon/2}$ given by the collection of intersections $C \, \cap \, )M'(_{\epsilon/2}$, rather than the cover by  $C \in \mathcal C$, and a priori the difference matters when constructing the nerve complex.  However, the proof of Proposition 4.8 shows that when a finite subset of $\mathcal C$ has a nonempty intersection, this intersection intersects $)M'(_{\epsilon/2}$, so one gets the same nerve whether one considers the collection $\mathcal C$ referenced in our statement of Lemma \ref{tsachiklem}, or the collection consisting of the intersections of its elements with $)M'(_{\epsilon/2}$, as in \cite{Gelanderhomotopy}.)

\begin {proof}[Proof of Proposition \ref{gelanderprop}]
Set $M'=M_{\geq \epsilon}$.  First, note that $M'$ is homotopy equivalent to its $\epsilon/2$-shrinking, since its complement components are star-shaped neighborhoods of either a closed geodesic or a point at infinity, so we can deformation retract $M'$ to its shrinking by flowing outwards. See the proof of Claim 8.5 of \cite{Gelanderhomotopy} for more details. So, by the Nerve Lemma it suffices to show that $M'$  satisfies the conditions of the lemma above.

 Conditions (1) and (2) are immediate from the definition of $M'$, where $$\tilde X_\gamma = \{ x\in \tilde X \ | \ d(x,\gamma(x)) < \epsilon\}, \ \ \gamma \in \pi_1 M.$$ 
 For condition (3), we define the vector $n(x)$ in two cases. As long as $\epsilon$ is small, we can assume that any $x$ in condition $(3)$ is contained in the preimage of the $\epsilon(d)$-thin part of $M$, where $\epsilon(d)$  is the Margulis constant. If $x$ lies in a component of this preimage that covers a Margulis tube, we define $n(x)$ exactly as in the proof of Lemma 7.4 of \cite{Gelanderhomotopy}, i.e.\ by using Lemma 7.3 with $b=b(d)$ and a unit vector $n(x)$ whose inner products with the gradients $\nabla d( \cdot, \tilde X_\gamma)$ are all at least $1/b$. If $x$ lies in a component that covers a cusp neighborhood, we let $n(x)$ point away from the point at infinity to which the lifted cusp neighborhood accumulates, just as in Section 6 of \cite{Gelanderhomotopy}. In \cite{Gelanderhomotopy} the control on the associated constant $b$ makes use of the fact that only locally symmetric manifolds are considered. Instead, here we use the following:% but in light of the following claim any $b\geq\epsilon \cdot a/2$ works for this case.

\begin{claim}[Moving away from the cusp]
Suppose $\tilde M$  is a simply connected Riemannian  manifold with curvatures in $[-1,-a^2]$ and let $\xi \in \partial_\infty \tilde M$. Let $\gamma$ be a parabolic isometry of $\tilde M$ with $\gamma(\xi)=\xi$, let $x \in \tilde M$ be  a point with $d(x,\gamma(x)) \geq \epsilon$ and let $c : \BR \longrightarrow \tilde M$  be a  unit speed geodesic with $c(-\infty)=\xi$ and $c(0)=x$. Then we have
$$\frac d{dt} d( c(t), \tilde X_\gamma)  |_{t=0} \geq \epsilon \cdot a/2.$$\label {derivative!}
\end{claim}
\begin{proof}
Since the geodesics $c(t)$ and  $\gamma \circ c(t) $ are asymptotic to $\xi$ as $t\to -\infty$ and always lie on the same horospheres, \cite[Proposition 4.1]{heintze1977geometry} says that for any fixed $s$,
$$d(c(t),\gamma \circ c(t)) \leq d(c(s),\gamma \circ c(s))\cdot e^{a(t-s)}, \ \ \forall t\leq s.$$
Since the two sides are equal at $t=s$ and we are saying that $t\leq s$,  it follows that
$$\frac d{dt} d(c(t),\gamma \circ c(t)) |_{t=s} \geq \frac d{dt} d(c(s),\gamma \circ c(s)) \cdot e^{a(t-s)} |_{t=s}.$$
Apply this to the (unique) value $s \leq 0$ such that $c(s) \in \partial \tilde X_\gamma.$ Then 
\begin{equation}\label {arggh}\frac d{dt} d(c(t),\gamma \circ c(t)) |_{t=s} \geq 
 \epsilon \cdot a.\end{equation}
On $\partial \tilde X_\gamma$,  the gradients of $d_\gamma$ and $d(\cdot, \tilde X_\gamma)$ are parallel.  As $$d_\gamma(y) \leq d_\gamma(x)+2d(x,y) \ \ \forall x,y,$$ we have $|\nabla d_\gamma| \leq 2$, while $\nabla d(\cdot, \tilde X_\gamma)$  is a unit vector. So, this and \eqref{arggh} imply:
$$\frac d{dt} d( c(t), \tilde X_\gamma)|_{t=s}  = \nabla d(\cdot, \tilde X_\gamma) \cdot c'(0) \geq \frac 12 \nabla d_\gamma \cdot c'(0) \geq 
\frac 12 \, \epsilon \cdot a. $$
Finally, as curvature is nonpositive and $X_\gamma$ is a convex set, $d(c(t), \tilde X_\gamma)$ is a convex function and hence has  increasing derivative. As $s\leq 0$,  the claim follows.
\end{proof}

So,  to finish the proof of Proposition \ref{gelanderprop}, we just take $b$ to be at least the constant $b=b(d)$ from the Margulis tube case, and at least $\epsilon \cdot a /2$. With this $b$ and $n$, the conditions of Lemma \ref{tsachiklem} are satisfied, so the proposition follows.
\end {proof}

 Finally, we prove the following estimate on the volumes of balls in the shrunk thick parts  $)M_{\geq \epsilon}(_{\delta}$, which  is necessary if we want to invoke Theorem \ref{Bowen2}.

\begin {lemma}\label {volumeballs}
 Suppose that $M$  is a complete Riemannian $d$-manifold with sectional curvatures in $[-1,-a^2]$. Fix  $\epsilon < \epsilon(d)$, let $\delta,r<\min\{\epsilon,\epsilon(d) -\epsilon\}/4$ and set $N:=\, )M_{\geq \epsilon}(_{\delta}$. Then there is some $c=c(d,\epsilon,a)>0$ such that 
$$\vol(B_{N}(p,r)) \geq c r^d, \ \forall p\in N.$$
\end {lemma}

 Note that since $M$ is non-positively curved, the volume of any  embedded metric ball $B \subset M$  is at least the volume of a  ball with the same radius in $\BR^d$, see e.g.\ \cite[Theorem~3.101]{gallot1990riemannian}. So, as long as we choose $\rho < \epsilon$, the lemma is trivial for balls $B_{N}(p,r)$ that do not intersect $\partial N$.  The point of the lemma, then, is that the boundary of $N$ is moderate enough  that  balls centered near $\partial N$ still have a  definite amount of volume that is contained in $N$.

\begin {proof}
As described in the paragraph above, it suffices to consider only  points $p \in N $  that are within $r$ of $\partial N$. The fact that $\delta < (\epsilon(d)-\epsilon)/4$ ensures that the radius $r$ ball $B_M(p,r)$ in $M$  around $p$ will be  an embedded ball contained in the $\epsilon (d)$-thin part $M_{< \epsilon(d)}$.  Choose a universal cover$$\pi : \tilde M \longrightarrow M,$$ components $\tilde T_{< \epsilon} \subset \tilde T_{< \epsilon(d)} \subset \tilde M$  of the preimages of $M_{<\epsilon}, M_{<\epsilon(d)}$,  and  a point $$\tilde p \in \tilde T_{< \epsilon(d)} \setminus ( \tilde T_{< \epsilon} )_{\delta}, \ \ \pi(\tilde p)=p.$$
 Then we can write $\tilde T_{<\epsilon}$  as the union$$\tilde T_{<\epsilon} = \cup_{\gamma} \tilde X_\gamma,$$  where  $\gamma$  ranges over the nontrivial elements in the group of deck transformations stabilizing $\tilde T_{<\epsilon}$,  and
$$ \tilde X_\gamma:= \{\tilde x\in \tilde M \ | \ d(\gamma(\tilde x),\tilde x)<\epsilon\}.$$

We claim that there is a unit vector $n \in T\tilde M_{\tilde p}$ and some $b=b(d,a,\epsilon)$  such that
\begin {equation}\label {1b} n \cdot \nabla d(\cdot, \tilde X_\gamma)
|_{\tilde p}  \geq 1/b >0, \ \ \forall \gamma.\end{equation}
Now if $\tilde T_{<\epsilon}$  covers a Margulis tube, any two $\gamma,\gamma'$ as above commute, so we have $$\nabla d(\cdot, \tilde X_\gamma) \cdot \nabla d(\cdot, \tilde X_\gamma') \geq 0$$
by the argument of \cite[Lemmas 7.1 and 7.2]{Gelanderhomotopy} (see also \cite[Lemma 3.5]{bader2016homology}), and then one can construct $n$  as in \cite[Lemma 7.3]{Gelanderhomotopy} (or \cite[Lemma 3.12]{bader2016homology}). If $\tilde T_{<\epsilon}$ covers a cusp neighborhood, we can just let $n$ be the unit vector that points away from the point  $\xi \in \partial_\infty \tilde M$  to which $\tilde T_{<\epsilon}$ accumulates (i.e.\ let $n=c'(0)$ where $c$  is a unit speed geodesic with $c(-\infty)=\xi$  and $ c(0)=\tilde p$) and then the claim follows from Claim \ref{derivative!} above, after setting $b=\epsilon a/2$.

 It follows from \eqref{1b}  that for every $v\in T\tilde M_{\tilde p}$  with $|v-n|<1/b$  we have \begin {align*} v \cdot \nabla d(\cdot, \tilde X_\gamma)   &= n \cdot \nabla d(\cdot, \tilde X_\gamma) + (v-n) \cdot \nabla d(\cdot, \tilde X_\gamma)	 \\
& = 1/b - |v-n|   \\ &>0,
 \end {align*}
 so $v$  points out of the convex subset $(\tilde X_\gamma)_{d(\tilde p,\tilde X_\gamma)} \subset \tilde M$ on whose boundary $\tilde p$  lies. And since 
$$\tilde p \not \in ( \tilde T_{< \epsilon} )_{\delta} \ \implies \ ( \tilde T_{< \epsilon} )_{\delta} \subset \cup_\gamma (\tilde X_\gamma)_{d(\tilde p,\tilde X_\gamma)},$$
we  then have that  for all $v\in T\tilde M_{\tilde p} $ with $|\, v/|v| -n\, |<1/b$, the  Riemannian exponential $$\exp_{\tilde p} (v) \not \in ( \tilde T_{< \epsilon} )_{\delta}.$$  Now as explained in the beginning of the proof, $r$ is small enough so that $B_M(p,r)$  is an embedded ball in $ M_{<\epsilon(d)}$. So, if we let 
$$V = \{ v \in T\tilde M_{\tilde p} \ \ | \ \ |v| < r, \ | \, v/|v| - n \, | <1/b \},$$ the composition $\pi \circ \exp_{\tilde p}$  of the universal covering map and the Riemannian exponential map embeds $V$  as a subset of $N$, where $N=)M_{\geq\epsilon}(_\delta$.  The ratio of the Euclidean volume of $V$ to $r^d$ is certainly bounded below by some constant depending only on $b=b(d,a,\epsilon)$, so nonpositive curvature implies that the same is true of $B_N(p,r)$, \cite[Corollary 11.4]{lee2006riemannian}.
\end {proof}

\subsection{The proof of Theorem \ref{pnc}}\label {pfsec}%The general strategy is as follows. Starting with a  weakly convergent sequence $(M_n)$ of finite volume $d$-manifolds with sectional curvatures in $[-1,-a^2]$, we will show that for a good choice of $\epsilon$, the measures $\mu_{( )(M_n)_{\geq \epsilon}(_\xi, M_n)} / \vol(M_n)$ on $\mathbb M^{ext}$ weakly converge. Theorem~\ref{Bowen2} will show convergence of the expected normalized Betti numbers of these thick parts, which Proposition \ref{gelanderprop} will then relate back to the normalized Betti numbers of the $M_n$.

\vspace{2mm}
 
We will assume everywhere below that  $d\geq 4$, since the theorem follows trivially from the Gauss--Bonnet theorem when $d=2$ and we have assumed that $d\neq 3$. 

Let $\mathcal M^d_a \subset\mathcal M$ be the subset consisting of all pointed Riemannian $d$-manifolds with sectional curvatures in $[-1,-a^2]$. Fix a sequence of finite volume $d$-manifolds $(M_n)$ with curvatures in $[-1,-a^2]$ and  for each $n$, let $\mu_n$ be the measure on $\mathcal M^d_a$ obtained by pushing forward the Riemannian measure on $M_n$ under $p \mapsto (M_n,p)$. By assumption, the sequence  $(\mu_n/\vol(M_n))$  converges weakly to some  probability measure $\mu$  on $\mathcal M^d_a$. 

\begin{claim}
For some $\epsilon_{max}>0$,  we have that $\mu(\mathcal E_\epsilon) =0$  for all but countably many $\epsilon \in (0,\epsilon_{max})$, where here $ \mathcal E_\epsilon$ is the set of all $(M,p) \in \mathcal M^d_a$ such that $M$ has a  primitive closed geodesic with length  exactly $2\epsilon$.\label {good epsilon}
\end{claim}

\begin{proof}
Take $\epsilon_{max} $  less than the Margulis constant $\epsilon(d)$. Using Proposition \ref{shortgeodesics},  we may assume that $\epsilon_{max}$ is small enough so that if $\epsilon\in (0,\epsilon_{max})$, then  any $\epsilon(d)$-Margulis tube with core length $2\epsilon$ has volume at least $1$.  For each $\epsilon\in (0,\epsilon_{max})$ and $R>0$, consider the set
$\mathcal E_{ \epsilon,R}$ of all $ (M,p) \in \mathcal M^d_a$ such that there is an $\epsilon(d)$-Margulis tube with core length $2\epsilon$  that is completely contained in the radius $R$ ball around $p$. 
In any manifold $M$ with sectional curvatures at least $-1$, the radius $R$ ball around any point has volume at most some constant $V(d,R)$, see \cite[Theorem 3.101]{gallot1990riemannian}. So, it follows that  for fixed $R$, any $(M,p)\in \mathcal M^d_a$  can be contained in $\mathcal E_{ \epsilon,R}$ for at most $V(d,R)$-many choices of $\epsilon$. Hence, we have
$$\sum_{\epsilon} \mu(\mathcal E_{ \epsilon,R}) \leq V(d,R),$$
 implying that $\mu(\mathcal E_{ \epsilon,R})\neq 0$ for at most  countably many $\epsilon$. But letting $R\in \BN$, there are then only  countably many \emph {pairs} $(\epsilon,R)$ such that $\mu(\mathcal E_{ \epsilon,R})\neq 0$, and hence only countably many $\epsilon$ such that $\mu(\mathcal E_{ \epsilon,R})\neq 0$ for \emph{some} $R\in \BN$. Since $\mathcal E_\epsilon = \cup_{R \in \BN} \mathcal E_{ \epsilon,R} $, the claim follows.\end{proof}

Fix now some small $\epsilon,\xi>0$, to be determined later, such  that $\mu(\mathcal E_\epsilon)=0$.   Using the notation  and terminology of \S \ref{mmsec}, consider the extended mm-space 
$$
 \mathfrak M_n :=( \ ) (M_n)_{\geq \epsilon}(_\xi,  M_n ),
$$
and let $\mu_{ \mathfrak M_n} $  be the associated measure on $\mathbb M^{ext}$.   Then if
$$\mathcal T = \{ (M,p) \in \mathcal M^d_a \ | \ d(p,M_{<\epsilon})>\xi\},$$
the measure $\mu_{ \mathfrak M_n} $ is just  the push forward of  the restriction $\mu_n |_{\mathcal T}$  under the map
$$\mathcal T \longrightarrow \mathbb M^{ext}, \ \ (M,p) \longmapsto ( \ ) M_{\geq \epsilon}(_\xi, p, M ).$$

\begin{lemma}\label {convergence Lemma}
	The measures $\mu_{ \mathfrak M_n}/ \vol(M_n)$ weakly converge.
\end{lemma}

Note that these are not probability measures. 

\begin{proof}
Let $f : \mathbb M^{ext} \longrightarrow\BR$ be a  bounded, continuous function and define
$$F : \mathcal M^d_a \longrightarrow \BR, \ \ F(M,p) = \begin{cases} f( \ ) M_{\geq \epsilon}(_\xi, p, M ) & (M,p)\in \mathcal T \\ 0 \text{ otherwise}.\end{cases}$$
We have $\int f \, d\mu_{ \mathfrak M_n}= \int F \, d\mu_n$, so it suffices to show that the limit
$$\lim_{n\to \infty} \frac 1{\vol(M_n)} \int F \, d\mu_n$$
exists. Recall that the measures $\mu_n/\vol(M_n) \to \mu$ weakly. So by the Portmanteau theorem,  it suffices to show that $F$ is continuous on a subset of $\mathcal M^d_a$  that has full $\mu$-measure. 

\begin{claim}
	The map $F$  is continuous on the difference $\mathcal M^d_a \setminus (\mathcal E_\epsilon \cup \mathcal D)$, where $$\mathcal D := \{ (M,p) \in \mathcal M^d_a \ | \ d(p,M_{<\epsilon}) = \xi \}$$ and $\mathcal E_\epsilon$ is as in Claim \ref{good epsilon}. \label {convergence claim} \end{claim} 
\begin{proof}[Proof of Claim \ref{convergence claim}]
Suppose that we have a convergent sequence $$(N_n,p_n) \to (N,p) \in \mathcal M^d_a \setminus (\mathcal E_\epsilon \cup \mathcal D).$$

Assume first that $d(p,N_{<\epsilon}) < \xi$. By a result of Ehrlich \cite{ehrlich1974continuity}, injectivity radius is continuous under smooth convergence, so  it follows that $d(p_n,(N_n)_{<\epsilon}) < \xi$ as well for large $n$.  In  this case, the continuity of $F$ along our sequence is obvious, since for large $n$, $$0 = F(N_n,p_n) \to F(N,p)=0.$$

So, assume that $d(p,N_{<\epsilon}) > \xi$, i.e.\ that $(N,p) \in \mathcal T$. First, we claim that $(N_n,p_n) \in \mathcal T$  for large $n$. If not, then after passing to a subsequence there would be points $q_n \in N_n$ with $d(p_n,q_n)\leq\xi$ and $\inj_{N_n}(q_n) \leq \epsilon$. Again by continuity of injectivity radius, we can take a subsequential limit of the $q_n$ to  produce some $q \in N$ with $d(p,q)\leq\xi$ and $\inj_N(q)\leq \epsilon$.
 If $\inj_N(q)$ is \emph{less than} $\epsilon$, then this contradicts that $d(p,N_{<\epsilon}) > \xi$. So assume $\inj_N(q)=\epsilon$. Since $(N,p) \not \in \mathcal E_\epsilon$, the point $q$ cannot lie on a closed geodesic of length exactly $2\epsilon$, so $q$ can be perturbed to a point $q'$ with $\inj_N(q')<\epsilon$. Taking the perturbation small enough so that $d(p,q')  < d(p,N_{<\epsilon})$,  we have a contradiction.

 In order to avoid a debauch of parentheses, set $T_n = ) (N_n)_{\geq \epsilon}(_\xi $ and define $T \subset N$ similarly. To prove that $F$ is continuous along $(N_n,p_n) \to (N,p)$, it suffices to show that \begin {equation}(T_n , p_n, N_n ) \ \to \ ( T  , p, N ) \in \mathbb M^{ext}.
 \label {conv--dope}	
 \end {equation}
Fixing some large $R>0$, choose a sequence of embeddings
$$\phi_n : B_N(p,R) \longrightarrow N_n, \ \ \phi_n(p)=p_n,$$
 such that the pullback metrics $\phi_n^*(g_n) \to g$ in the smooth topology, as described in the appendix of \cite{Abertunimodular}.  To prove \eqref{conv--dope}, we would like to apply Lemma \ref{lipschitz Lemma2} to say that for a given $\alpha>0$, the triples in \eqref{conv--dope} are $(\alpha,R)$-related for large $n$.  This requires proving that for an arbitrary $\delta>0$, conditions (1)--(3) in Lemma \ref{lipschitz Lemma2} hold for large $n$.

Condition (1) in Lemma \ref{lipschitz Lemma2} is immediate, since the maps $\phi_n$ are nearly isometries when $n$ is large.   The proof of condition (2) in Lemma~\ref{lipschitz Lemma2} is similar to the first two paragraphs of the current claim. Namely, suppose that the first part of condition (2) fails for infinitely many $n$. Then  for infinitely many $n$, there are points $$q_n  \in  T_n  \cap \phi(B_N(p,R)), \ \ d(\phi_n^{-1}(q_n),T)>\delta.$$ Passing to a subsequence, we can assume that $\phi_n^{-1}(q_n)\to q\in B_N(p,R)$,  and  by continuity of injectivity radius we have $q\in T$, a contradiction.  The second part of condition (2) is similar, although as we did above one has to use that there are no closed geodesics of length exactly $2\epsilon$ in $N$. So, it  remains to prove condition (3) of Lemma~\ref{lipschitz Lemma2}, i.e.\ that $$\vol(\phi_n^{-1}(T_n) \triangle T)<\delta$$ for large $n$.  Pick a neighborhood $U \supset \partial T \cap B_N(p,R)$ with  volume less than $\delta.$ If $n$ is large, then the same arguments as above show that 
$\phi_n^{-1}(T_n) \triangle T \subset U,$
so we are done.
\end{proof}

By our choice of $\epsilon$,  we have $\mu(\mathcal E_\epsilon) = 0.$ So, to prove Lemma \ref{convergence Lemma} it suffices to show that $\mu(\mathcal D)=0$.  Essentially, the point is that $d(p,M_{<\epsilon})=\xi$ is  a measure zero condition within each fixed $M$, and as a weak limit of measures constructed using Riemannian measures on finite volume manifolds, $\mu$ is distributed on each `leaf' $$\mathcal L_M = \{(M,p) \ | \ p\in M\} \subset \mathcal M^d_a$$ 
 according to the Riemannian measure of $M$. (This is not quite precise, the leaves may be highly singular, but one can  make this argument work in the foliated `desingularization' of $\mathcal M$  constructed in \cite[Theorem 1.6]{Abertunimodular}). However, an easier approach is to use that $\mu$  satisfies the \emph {mass transport principle}, see \cite[(1)]{Abertunimodular}.  Namely, define a Borel function
$$\varphi : (\mathcal M^d_a)_2 \longrightarrow \{0,1\}, \ \  \varphi(M,p,q)= \begin{cases}
	1 & d(p,M_{<\epsilon})=\xi \text{ and } d(p,q) \leq \epsilon \\ 
0 & \text {otherwise}
\end{cases},$$
where $(\mathcal M^d_a)_2$  is the space of \emph{doubly pointed} $d$-manifolds with curvature in $[-1,-a^2]$,  endowed with the natural version of smooth convergence, see \cite{Abertunimodular}.   Note that
$$d(p,M_{<\epsilon})=\xi \ \implies\  \int_{q\in M} \varphi(M,p,q) \, d\vol \geq \vol \,  B_{\BR^d}(0,\epsilon),$$
since embedded $\epsilon$-balls in a $d$-manifold of nonpositive curvature have volume at least that of an $\epsilon$-ball in $\BR^d$, c.f.\ \cite[Theorem~3.101]{gallot1990riemannian}.  Then
\begin{align*}
	\mu(\mathcal D) &\leq 1/B_{\BR^d}(0,\epsilon) \cdot  \int_{(M,p)\in \mathcal M^d_a} \int_{q\in M} \varphi(M,p,q) \, d\vol_M \, d\mu  \\
&= 1/B_{\BR^d}(0,\epsilon) \cdot\int_{(M,p)\in \mathcal M^d_a} \int_{q\in M} \varphi(M,q,p) \, d\vol_M \, d\mu  \\
&= 0,
\end{align*}
 where the first equality is the {mass transport principle} \cite[(1)]{Abertunimodular}, and the last equality is because  for small $\epsilon$, the set of points exactly at distance $\xi$ from the $\epsilon$-thin part has measure zero in any manifold with  negative curvature.
\end{proof}

 We now know that  the sequence of measures $\mu_{ \mathfrak M_n}/ \vol(M_n)$ weakly converges, and we would like to apply Theorem \ref{Bowen2}, or really  Corollary \ref{Bowen3}. Lemma \ref{volumeballs} will give the lower bound on ball volumes needed in Theorem \ref{Bowen2} (1), and the upper bound needed in (2) comes from the uniform lower sectional curvature bound, see e.g.\ \cite[Theorem 3.101 on p.\ 169]{gallot1990riemannian}. 
   The key, though, is to use our work in \S \ref{simpsec} to  define the  appropriate $r_0,r_1,r_2,r_3$.  Namely, take $\epsilon,\delta>0$ small enough so that they work in Proposition \ref{gelanderprop}, set $\xi = \epsilon/2+\delta$, and let $b$ be as given in Proposition \ref{gelanderprop}  for $c=3,c'=4,$ say. If $$r_0=\delta, \ r_1=3\delta, \ r_2 = (b+6)\delta, \ r_3= (b+7)\delta,$$
then Proposition \ref{gelanderprop} says that 
the nerve $N_{M_n}(S_n,\rho_n)$ in $M_n$  associated to any $[r_2,r_3]$-weighted $(r_0,r_1)$-net $(S_n,\rho_n)$ in $)(M_n)_{\geq \epsilon}(_{\xi}$ is homotopy equivalent to $(M_n)_{\geq \epsilon}$. So, applying Corollary \ref{Bowen3} to the sequence  of extended mm-spaces $\mathfrak M_n$, with $B_n = b_k((M_n)_{\geq \epsilon})$, $V_n = \vol(M_n)$  and $r_0,r_1,r_2,r_3$ as above, we get that the  limit
\begin{equation}\lim_{n\to \infty} \frac{b_k((M_n)_{\geq \epsilon})}{  \vol (M_n)} =L \in [0,\infty).
\label {secconv} \end{equation}
But Proposition \ref{shortgeodesics} says that  the number of components of the $\epsilon$-thin part of $M_n$ is at most $\vol(M_n)/C$, where $C=C(\epsilon,d,a)\to \infty$ as $\epsilon\to 0$. Removing a cusp neighborhood from $M_n$ does not change the homotopy type, and by Mayer--Vietoris removing a Margulis tube can only change Betti numbers by $1$. So, we get that for each $n$ and $k$, 
$$
 |b_k((M_n)_{\geq \epsilon}) - b_k(M_n)| \leq \vol(M_n)/C.
$$
Combining this with \eqref{secconv}, we get that
$$ L-1/C \leq \liminf_{n\to \infty} \frac{b_k(M_n)} { \vol M_n} \leq \limsup_{n\to \infty} \frac{b_k(M_n)} { \vol M_n} \leq L+1/C,$$
so sending $\epsilon\to 0$, and hence $C\to \infty$,  proves the theorem.

%
%
%
%
%\section{Non-uniform arithmetic manifolds}
%
%This interesting particular case is straightforward from:
%
%\begin{itemize}
%\item Theorem \ref{TEB}.
%\item The fact that in the non-uniform arithmetic case $\mathcal{R}_{M_n}$ is homotopy equivalent to $M_n$ \cite[Theorem ??]{hv}.
%\end{itemize}
%
%
%
%

\section{Manifolds of nonpositive curvature and Theorem \ref{npc}}

In this section we prove Theorem \ref{npc}, i.e.\ the convergence of normalized Betti numbers for BS-convergent sequences of analytic $d$-manifolds of nonpositive curvature without Euclidean factors, when the limit is thick. 
For that purpose it will be more convenient to work not with the standard thick thin decomposition but a close variant of it, introduced in \cite{Ballmannmanifolds}, which we call a `stable' thick thin decomposition:

\subsection{A stable thick thin decomposition}\label{stablesec}
Suppose that $M $ is a finite volume,  real analytic $d$-manifold with sectional curvatures in the interval $[-1,0]$ and that the universal cover $X$ of $M$ has no Euclidean deRham factors. 
Write $M=\Gamma \backslash X$.
Then $\Gamma$ operates freely and the displacement functions $d_\gamma$ ($\gamma \in \Gamma$) are analytic. In particular the convex sets 
$$\mathrm{Min}(\gamma ) = \{ x \in X \; | \; d_\gamma (x) = \mathrm{min} (d_\gamma ) \}$$
are complete submanifolds. An element $\gamma \in \Gamma$ is called \emph{$J$-stable} if we have $$\Min(\gamma^i) = \Min(\gamma), \ \ \forall i=1,\ldots,J.$$
%Intuitively, isometries of $X$ can have translational and normal $O(n)$ parts, and being stable means that no small power of the $O(n)$ part fixes a point on the sphere. For instance, if $\gamma $ is a screw motion in $\BR^3$ with angle $2\pi / n$, then $\gamma$ is $J$-stable exactly when $n < J$. 

Let $\epsilon$ be less than the Margulis constant, and $I$ the index constant in the Margulis lemma, and fix also the constants $ \delta, I_\delta, $ and $J = I_\delta \cdot I$ defined at the beginning of \cite[\S 13.4]{Ballmannmanifolds}, \emph{but using $\epsilon$ instead of the actual Margulis constant}. The interested reader can refer to \cite{Ballmannmanifolds} if necessary, but it is not necessary to know what these constants are to read our proof below. (As at the top of pg 141 of \cite{Ballmannmanifolds}, though, we note that $0 < \delta < \epsilon / I_\delta$.) As in  \cite{Ballmannmanifolds}, let
\begin{align*}\Delta_0 &:= \{\gamma \in \Gamma \setminus \{1\} \ | \ \gamma \text{ is } J\text{-stable and }  \inf_{x\in X} d_\gamma(x) \leq \delta\}, \text { and } \\
&\ \ \ \ \ \ \ \ \ \ \ \ \ \ \Delta := \{ \gamma^1, \ldots, \gamma^{I_\delta} \  | \ \gamma \in \Delta_0 \}.
\end{align*}

We now define two subsets $X_\pm \subset X$ with $X_- = \overline{X \setminus X_+}$ by
$$X_+ := \{ x\in X \ | \ d_\gamma(x) \geq \epsilon \ \  \forall \gamma \in \Delta \}, \ \ X_- := \{ x\in X \ | \ d_\gamma(x) \leq \epsilon \  \text{for some } \gamma \in \Delta \},$$
and we define the \emph{stable $\epsilon$-thick part} $M_+$ and the \emph{stable $\epsilon$-thin part} $M_-$ by
$$M_+ := \Gamma \backslash X_+ , \ \ M_- = \Gamma \backslash X_- $$

\begin{lemma}\label{itsamanifold}
$M_+,M_-$ are topological submanifolds of $M$ and their common boundary $\partial M_- = \partial M_+$ is compact.
\end{lemma}
\begin{proof}
 We work  mostly with $X_\pm$,  and address $M_\pm$  at the end. By definition, $X_-$ is the union over all $\gamma \in \Delta$ of the sets 
$$U_{\gamma,\epsilon}:= \{x \in X \ | \ d_\gamma(x) \leq \epsilon \}.$$
Since $X$ is a Hadamard manifold, the second variational formula implies that the distance function $d : X \times X \longrightarrow \BR$ is convex. Moreover, each $d_\gamma$ is a submersion except along $\mathrm{Min}(\gamma)$, see \cite[Lemma, pg 96]{Ballmannmanifolds}. This implies that $U_{\gamma,\epsilon}$ is a smooth, convex codimension zero submanifold of $X$.

Let $N$ be the frontier of $X_+$ and $X_-$ in $X$.  We claim that $N$ is a topological submanifold of $X$. So, pick some $p\in N$. By discreteness of $\Gamma$, there is a small open neighborhood $W \subset X$ of $p$ and a finite subset $\mathcal F \subset \Delta$  such that $W \cap U_{\gamma,\epsilon}\neq \emptyset$ only when $\gamma \in \mathcal F$. By the Margulis lemma, $\mathcal F$ generates a subgroup of $\Gamma$ that has a nilpotent subgroup with index at most $I$. Now $p \not \in \mathrm{Min}(\gamma)$ for any $\gamma \in \Gamma$, since $$\inf d_\gamma(x) \leq \delta \cdot I_\delta < \epsilon / I_\delta \cdot I_\delta = \epsilon $$
while the fact that $p\in F$ implies that $d_\gamma(p) \geq \epsilon$ for all $\gamma \in \Gamma$. As every $\gamma \in \Delta$ is $I$-stable, this means that $p\not \in \mathrm{Min}(\gamma^i)$ for any $i\leq I$, so \cite[Lemma, part (2), pg 96]{Ballmannmanifolds} says that there is some $v \in TX_p$ such that \begin{equation} \label{goodv} \langle \nabla d_\gamma , v \rangle >0, \ \ \forall \gamma \in \mathcal F.\end{equation}
Note that at $p$, the gradient $\nabla d_\gamma$ is just the outward normal vector to the set $U_{\gamma,\epsilon}$. 

Shrinking $W$ if necessary, pick a chart $\phi : W \longrightarrow \BR^d = \BR^{d-1} \times \BR$ with $\phi(p)=(0,0)$ and such that $d\phi(v) = (0,1)$. After shrinking $W$ further, the implicit function theorem and \eqref{good epsilon} imply that for each $\gamma\in\mathcal F$, we have $$\phi(W \cap U_{\gamma,\epsilon}) = \{ (x,t) \in \phi(W) \subset \BR^{d-1} \times \BR \ | \ t \leq f_\gamma(x) \},$$
where $f_\gamma $ is a smooth function defined on a neighborhood of $0 \in \BR^{d-1}$. Hence, $$\phi(W \cap N) = \{ (x,t) \in \phi(W) \subset \BR^{d-1} \times \BR \ | \ t = \max_{\gamma \in \mathcal F} f_\gamma(x) \},$$
which is the graph of a continuous function. Hence, $N$ is a submanifold of $X$.

The frontier $\partial M_- = \partial M_+$ is the projection of $N$ to $M$, and hence is a topological submanifold of $M$. It follows that $M_\pm$ are  topological submanifolds with boundary.  Finally, frontiers are always closed, and since $M$ has finite volume $M_+ \subset M_{\geq \epsilon'} $ is compact, so $\partial M_\pm$ is compact.  \end{proof}

By \cite[Corollary 12.5]{Ballmannmanifolds}, there is some integer $m=m(J)$ such that for any $\gc\in\gC\setminus\{ 1\}$, there is some $j \leq m$ such that $\gc^{j}$ is $J$-stable. So, if $\gep'=\epsilon/m$ we have
 \begin{equation}\label{weirdthickcontained}
 	M_{<\gep'}\subset M_-\subset M_{<\epsilon}, \ \ M_{\geq \epsilon}\subset M_+\subset M_{\geq' \gep}.
 \end{equation}

\vspace{2mm}

 The following proposition is a modification of \cite[Theorem 13.1]{Ballmannmanifolds}.

\begin{proposition}\label{thinbettis}
Suppose that $M $ is a finite volume,  real analytic $d$-manifold with sectional curvatures in the interval $[-1,0]$ and that the universal cover $X$ of $M$ has no Euclidean de Rham factors. Then there is some  $C=C(d,\epsilon)$ such that for all $k \in \BN$, both 
$b_k(M_-)$ and $ b_k(\partial M_-)$ are less than or equal to $C \vol(M_{<2 \epsilon}).$
\end{proposition} 

It is necessary to assume here that $M$ is analytic and that $X$ has no Euclidean de Rham factors. If $M = N \times S^1$ for some $(d-1)$-manifold $N$, we can scale the $S^1$-factor so that $M=M_{\leq \epsilon}=M_-$ and $\vol(M)\approx 0$. And unless we assume analyticity (or some weaker alternative, see \cite[\S A2]{Ballmannmanifolds}) there are finite volume manifolds with sectional  curvatures in $[-1,0]$ where the thin parts have infinite Betti numbers, see \cite[\S 11.1]{Ballmannmanifolds}.

 Before starting the proof,  we record a brief algebraic topology lemma.

\begin{lemma}\label{atlem}
	 Suppose that $N$  is a (possibly noncompact) topological manifold with compact boundary. If $b_k(N,\BR) \leq C$ for all $k$, then $b_k(\partial N,\BR) \leq 2C$ for all $k$.
\end{lemma}

\begin{proof}
By Poincar\'e-Lefschetz duality, using cohomology with compact support, 
\begin {equation}\label {pl}
C \geq b_*(N) = \dim H_*(N;\BR) = \dim H^{d-*}_c(N,\partial N;\BR) .	
\end {equation}
But by the long exact sequence of a pair,  for every $k$ we have %$$\cdots \longrightarrow H^k(N,\partial N) \longrightarrow H^k(N) \longrightarrow H^k(\partial N) \longrightarrow H^{k+1}(N,\partial N) \longrightarrow \cdots, $$
\begin{align*}
	\dim H^k(\partial N;\BR) & \leq  \dim H^k(N;\BR) + \dim  \mathrm{Image} \big (  H^k(\partial N;\BR) \longrightarrow H^{k+1}(N,\partial N;\BR)  \big ) \\
& \leq C+ \dim  \mathrm{Image} \big ( H^k(\partial N;\BR) \longrightarrow
 H^{k+1}(N,\partial N;\BR)  \big ) .
\end{align*}
 However,  since $\partial N $  is compact, the map whose image we are interested in factors as
$$H^k(\partial N;\BR) \cong H^k_c(\partial N;\BR) \longrightarrow H^{k+1}_c(N,\partial N;\BR) \longrightarrow H^{k+1}(N,\partial N;\BR).$$
So, the  dimension of the image is at most $C$, by \eqref{pl}. Hence, $$b_k(\partial N) = 	\dim H^k(\partial N;\BR) \leq 2C.\qedhere $$
\end{proof}

 So for instance, to prove Proposition \ref{thinbettis}  it would suffice to just estimate the Betti numbers $b_k(M_-)$, and the estimates for $ b_k(\partial M_-)$ would follow.  It turns out, however, that this is not logically how the proof will go, since one needs an a priori  estimate for $ b_k(\partial M_-)$ in order to calculate $b_k(M_-)$. We will still apply Lemma \ref{atlem} to  estimate $ b_k(\partial M_-)$, though, but the manifold $N$ in the lemma  will not be $M_-$.

\medskip

 We are now ready for the proof.

\begin{proof}[Proof of  Proposition \ref{thinbettis}]
We first estimate the Betti numbers of $\partial M_-$. By \eqref{weirdthickcontained}, 
$$M_{<\gep'}\subset M_-\subset M_{<\epsilon},$$
where $\epsilon'<\epsilon$ depends only on $\epsilon,d$.
Let $\{B_\alpha\}$ be a collection of open $\epsilon'/2$-balls in $M$ with centers on $\partial M_-$, such that the  same centers determine a maximal collection of pairwise disjoint $\epsilon'/4$-balls centered on $\partial M_-$.   Since $\inj(p) \geq \epsilon'$ on $\partial M_-$, each $B_\alpha$ is embedded and convex, and contained in $M_{<2\epsilon}$.  In particular, we have 
\begin{equation}
	\# \{B_\alpha\} \leq  C \cdot \vol(M_{<2\epsilon}), \text{ where } C=C(\epsilon,d).
\label{numberballs}
\end{equation}

Choose  arbitrary lifts $\tilde B_\alpha \subset X$ of each ball. If $X_- \subset X$ is the preimage of $M_-$, then as discussed in the proof of Lemma~\ref{itsamanifold}, we have
$$int(X_-) = \cup_{\gamma \in \Delta} int(U_{\gamma,\epsilon}) , \text{ where } int(U_{\gamma,\epsilon}):= \{x \in X \ | \ d_\gamma(x) < \epsilon \}.$$
Since $int(U_{\gamma,\epsilon})$ is convex, each intersection $\tilde B_\alpha \cap int(U_{\gamma,\epsilon})$  projects to a convex open subset $B_{\alpha,\gamma} \subset int(M_-)$. Note that for a given $\alpha$, only some $N=N(\epsilon,d)$ of the sets $B_{\alpha,\gamma}$ are nonempty, for instance by  Corollary 3.4 of \cite{bader2016homology}, so \eqref{numberballs}  implies that the number of nonempty $B_{\alpha,\gamma}$ is bounded by $C \cdot \vol(M_{<2\epsilon})$,  after adjusting $C=C(\epsilon,d)$. It then follows from the Nerve Lemma (see the proof of Claim \ref{diagram} below) that the Betti numbers of the union 
$$\mathcal U = \cup_{\alpha,\gamma} B_{\alpha,\gamma}$$
are bounded above by $C \cdot \vol(M_{<2\epsilon})$ as well. But $\mathcal U \cup \partial M_-$ is a manifold with compact boundary by Lemma \ref{itsamanifold}, so  it follows from Lemma \ref{atlem} that \begin{equation}\label{firstthing}
b_k(\partial M_-) \leq C \cdot \vol(M_{<2\epsilon})
 \end{equation}
for some $C=C(\epsilon,d)$ as well.

\medskip

The estimate for $b_k(M_-)$ closely follows the proof of \cite[Theorem 13.1]{Ballmannmanifolds}. Let $$g : (0,\infty) \longrightarrow [0,\infty) $$ be a $C^{\infty}$ function with
\begin{itemize}
	\item $g(t)>0, g'(t) < 0$ for $t\in (0,\epsilon)$,
\item $g(t)=0$ for $t\geq \epsilon$, 
\item $g(t)\to \infty$ as $t\to 0$,
\item $g(\delta)=1$.
\end{itemize}  
and with $\Delta$ as in the beginning of \S \ref{stablesec}, consider the smooth function
$$F : X \longrightarrow [0,\infty), \ \ F(x) = \sum_{\gamma\in \Delta} g \circ d_\gamma(x).$$
Since $\Delta$ is conjugation invariant in $\Gamma$, this $F$  descends to a smooth function $f : M \longrightarrow \BR$. On \cite[pg 145]{Ballmannmanifolds}, it is shown that $f$ and $F$ have finitely many critical values $$0=r_0 < r_1 < \ldots < r_s.$$
 Note that the $0$-critical set $f^{-1}(0)$ is exactly $M_+$, and that $f^{-1}(0,\infty) = int(M_-)$.

In \cite{Ballmannmanifolds}, Ballmann--Gromov--Schroeder use this function $f$ to give a linear upper bound \begin{equation}\label {BGS theorem} b_k(M) \leq C \cdot \vol(M),  \ \ C=C(d).\end{equation} We will describe this argument, then indicate how to modify it to prove that $b_k(M_-) \leq C \vol(M_{2\epsilon})$. First, at the top of pg 148 in  \cite{Ballmannmanifolds}, the authors prove:

\begin{equation}b_k(M)  \leq \sum_{i,j} b_k\big (\{f_{x_{ij}} < r_i + \rho\} \big ).	\label {Morrissey}
\end{equation}
In the summation, each index $i$ corresponds to a critical value $r_i$ of $f$, and the indices $j$ correspond to different pieces of the critical set $f^{-1}(r_i)$. More precisely, there is a collection of complete immersed submanifolds $V_{x_{ij}} \looparrowright M$ as follows\footnote{In \cite{Ballmannmanifolds}, they set $V_x := Y_x / \Gamma_x$, but mostly use the latter notation in proofs.}. For each $j$, let $f_{x_{ij}} := f |_{V_{x_{ij}}}$ be the restriction. Then the minimum value of $f_{x_{ij}}$ is $r_i$, this minimum is achieved on the set $\{f_{x_{ij}} =r_i\}$, which has nonempty interior in $V_{x_{ij}}$, and $f^{-1}(r_i)$ decomposes as:
$$f^{-1}(r_i) = \cup_{ij} \{f_{x_{ij}} =r_i\}.$$
So, in \eqref{Morrissey} the set $\{f_{x_{ij}} < r_i + \rho\}$ is just a small neighborhood of $\{f_{x_{ij}} =r_i\}$ in $V_{x_{ij}}$, since $\rho>0$ is small. 
The proof of \eqref{Morrissey} is essentially via Morse theory,  applied to the function $f$: one considers the homology of the sublevel set $\{f < r\}$, starting with $r<0$ where $\{f < r\} = \emptyset$, and one shows that passing through the critical point $r=r_i$  contributes at most the corresponding index-$i$ terms of the summation in \eqref{Morrissey}  to the Betti numbers. 

%Now on \cite[pg 145]{Ballmannmanifolds},  the authors show not only that there are finitely many critical values, but that the number of terms in the summation \eqref{Morrissey} is bounded above by some constant depending only on $d$.  

To derive \eqref{BGS theorem}, the authors show in \cite[pg 148, (16)]{Ballmannmanifolds} that each term in \eqref{Morrissey} is bounded above by a constant times the \emph{essential volume}\footnote{We are suppressing some constants in our notation. Really, essential volume depends on a choice of $\epsilon$ and $a>0$, and is written $\epsilon\text{-}ess^a\text{-}vol$ in \cite{Ballmannmanifolds}.} of the immersed submanifold $V_{x_{ij}}$:
\begin {equation}\label {bettivol} 
 b_k(\{f_{x_{ij}} < r_i + \rho\}) \leq C \cdot  ess\text{-}vol( V_{x_{ij}}), \ \ C = C(d).
\end{equation} Here, the $ess\text{-}vol(V)$ is an integer that estimates volume up to some fixed multiplicative constant, but in a way that ignores small volume Euclidean factors, see \cite[\S 12.8]{Ballmannmanifolds}.  Finally, in \cite[Theorem 12.11]{Ballmannmanifolds}  they show\footnote{The conclusion of \cite[Theorem 12.11]{Ballmannmanifolds} is about essential volume, but note that for $M$ itself, essential volume agrees with volume up to a dimensional constant, since $M$ has no Euclidean factors.} that
\begin {equation} \label {essential volume estimate}\sum_{ij} ess\text{-}vol( V_{x_{ij}}) \leq C \cdot \vol(M), \ \ C = C(d),\end {equation} so it follows from \eqref{Morrissey}, \eqref{bettivol} and \eqref{essential volume estimate} that $b_k(M) \leq C \cdot \vol(M)$, where $C=C(d)$.

\medskip

We now adapt this argument to $M_-$. It suffices to estimate the Betti numbers of 
$$int(M_-) = f^{-1}(0,\infty),$$
since $M_-$ is a manifold, by Lemma \ref{itsamanifold}, and so is homotopy equivalent to $int(M_-)$. The idea is to run the Morse theory argument proving \eqref{Morrissey}, but  only on the interval $(0,\infty)$. %The same argument as in \cite{Ballmannmanifolds} will say that when we increase $r$, the Betti numbers of  the sublevel sets $\{0<f<r\}$ increase when $r $ passes through $r_i$ by the index $i$ terms in \eqref{Morrissey}. 

There are two main differences in the argument. {First,} we are no longer starting the Morse theory argument with an empty sublevel set, so we need to estimate independently the Betti numbers of $f^{-1}(0,r)$ when $r>0$ is small.  As long as $r<r_1$, Morse theory implies that $f^{-1}(0,r)$ is homeomorphic to a product $Z \times (0,r)$. The union $f^{-1}(0,r) \cup \partial M_-$ is a manifold with boundary by Lemma \ref{itsamanifold}, so there is a collar neighborhood  $\partial M_- \times [0,1) \hookrightarrow f^{-1}(0,r)$. Since this collar gives an end neighborhood of $f^{-1}(0,r) \cong Z \times (0,r)$, there is some $t \approx 0$  such that $f^{-1}(t) \subset \partial M_- \times [0,1)$. But as the composition
$$Z \cong f^{-1}(t) \hookrightarrow \partial M_- \times [0,1) \hookrightarrow f^{-1}(0,r) \cong Z \times (0,r)$$  is a homotopy equivalence, it follows that the homology of $Z$ injects into  the homology of the collar $\partial M_- \times [0,1)$, and therefore  by \eqref{firstthing} we have 
\begin{equation}\label{firstthing2}
	b_k(f^{-1}(0,r)) = b_k(Z) \leq  b_k(\partial M_-) \leq C \cdot \vol(M_{<2\epsilon}).
\end{equation}
Above, we are avoiding saying that $Z$ is homeomorphic to $\partial M_-$, which is what you would expect in the current situation. This is probably true, but it is less obvious than the  estimate in \eqref{firstthing2},  which is all we need.

{Second,} we claim that for some $C=C(d,\epsilon)$, we have
\begin{equation}
	\sum_{i,j, \ i\neq 0} ess\text{-}vol( V_{x_{ij}}) \leq C \cdot \vol(M_{<\epsilon}). \label{2ndthing}
\end{equation}
 This follows from the arguments in \cite{Ballmannmanifolds}.  Namely, their proof of \eqref{essential volume estimate} in \cite[Theorem 12.11]{Ballmannmanifolds} is stronger than the statement: if $N := \sum_{ij} ess\text{-}vol( V_{x_{ij}})$, the authors construct a collection of $N$ injectively embedded $r$-balls centered at points of $M_{<\epsilon/2}$\footnote{See \cite[(3) pp. 132--133]{Ballmannmanifolds}. In their construction of the  centers $z$ of these balls there exists an element $\beta_z \in \Gamma$ such that 
$d_{\beta_z} (z) = \epsilon/2$, see top of page 132.} that overlap with uniformly bounded multiplicity, where here $r>0$ depends only on $d$. Hence, this shows $N$  is at most a dimensional constant times $ \vol(M_{<\epsilon})$ as desired.

The Proposition now follows from \eqref{firstthing2} and \eqref{2ndthing}. Namely,
\begin{align*}
	b_k(M_-) &\leq b_k(f^{-1}(0,r)) + \sum_{i,j,\ i\neq 0} b_k\big (\{f_{x_{ij}} < r_i + \rho\} \big ) \\
&\leq b_k(f^{-1}(0,r)) + \sum_{i,j,\ i\neq 0} ess\text{-}vol( V_{x_{ij}}) \big ) \\
& \leq C \cdot  \vol(M_{<\epsilon}),
\end{align*}
 where the first inequality is the Morse theory argument from \cite{Ballmannmanifolds}, the second  inequality is \eqref{bettivol}, and the third is \eqref{firstthing2} and \eqref{2ndthing}.\end{proof}

\subsection{Proof of Theorem \ref{npc}}

Let $\epsilon_0>0$ and let $(M_n)$  be a sequence of real analytic, finite volume Riemannian $d$-manifolds with sectional curvatures in the interval $[-1,0]$, and assume the universal covers of the $M_n$ do not have Euclidean de Rham-factors.  Assume $(M_n)$ BS-converges to a  measure $\mu$ on $\mathcal M$  that is supported on $\epsilon_0$-thick manifolds. Here,  recall that BS-convergence means that if $\mu_n$ are the associated measures on $\mathcal M^d$, then
$$\mu_n/\vol(M_n) \to \mu$$ weakly. We want to show that the following limit exists for all $k$: $$\lim_{n\to \infty} b_k(M_n)/\vol(M_n).$$

\smallskip

First, here  is the reason we assume that $\mu$  is supported on $\epsilon_0$-thick manifolds.  %The following claim will allow us to chop off (part of) the thin part of each $M_n$ without affecting the weak limit $\mu$, and without affecting the Betti numbers too much.

\begin{claim}\label {thinpartsmall}
	For all $R>0$ and $0<\epsilon<\epsilon_0$, we have $$\vol( \{ x \in M_n \ | \ d(x,M_{\leq \epsilon})\leq R \} )/\vol(M_n)\to 0.$$
\end{claim}
\begin{proof}
By the continuity of injectivity radius with respect to smooth convergence \cite{ehrlich1974continuity}, $$D:=\{ (M,p) \in \mathcal M^d \ | \ d(p,M_{\leq\epsilon}) \leq R \} \subset\mathcal M^d$$ is closed, so by the Portmanteau theorem, $$\limsup_n \frac{\vol( \{ x \in M_n \ | \ d(x,M_{\leq \epsilon})\leq R \} )}{\vol(M_n)} \leq \limsup_n \mu_n( D) \leq \mu(D) = 0,$$
so the limit of the sequence on the left is zero. 
\end{proof}

Pick some $\epsilon >0$ that is less than $\epsilon_0$ and also less than the Margulis constant, and fix some $\xi < \epsilon/20$. With the notation $( \ )_+$ of the last section, with input $\epsilon$, let $$N_n := \overline{ ((M_n)_{+} )_\xi} = \{ x \in M_n \ | \ d(x,(M_n)_{+}) \leq \xi\},$$
and using the notation and terminology of \S \ref{mmsec},  consider the extended mm-spaces
$$\mathfrak M_n := (M_n, M_n), \ \mathfrak N_n := (N_n, M_n)$$
and their associated measures $\mu_{\mathfrak M_n}, \mu_{\mathfrak N_n}$ on $\mathbb{M}^{ext}$. (Note that the space $N_n$ may be disconnected, but since it has finitely many components, it is special and hence our work in \S \ref{normbmm} still applies to $\mathfrak N_n$.) Here, if $$\iota : \mathcal M^d \longrightarrow \mathbb{M}^{ext}, \ 
(M,p) \longmapsto (M,p,M),$$ is the natural continuous map (see Corollary \ref{cnts}), then $\mu_{\mathfrak M_n}= \iota_*(\mu_n)$, so  $$\mu_{\mathfrak M_n}/\vol(M_n) = \iota_*\big (\mu_n / \vol(M_n)\big ) \to \iota_*(\mu).$$

\begin{claim}\label{too}
	We have $\mu_{\mathfrak N_n}/\vol(M_n) \to \iota_*(\mu)$ as well.
\end{claim}

Here, note  that by Claim \ref{thinpartsmall},  we have that $\vol(N_n) / \vol(M_n) \to 1$, so one could replace the normalizing factor by $\vol(N_n)$ if desired.

\begin{proof}
Let $f : \mathbb M^{ext} \longrightarrow [0,m]$ be a continuous function, and  fix $\alpha>0$. Given $\delta,R$, let 
$$C_{\delta,R} = \{ \mathfrak M \in \mathbb M^{ext} \ \ | \ \  \mathfrak M \text{ is } (\delta,R)\text{-related to } \mathfrak N \in \mathbb M^{ext} \implies |f(  \mathfrak M) - f(\mathfrak N) | <\alpha\}.$$
Since the sets $C_{\delta,R}$ are open,  are nested when $\delta$  is decreased and $R$ is increased, and union to all of $\mathbb M^{ext}$,  we can choose $\delta,R$  such that $$\iota_*(\mu)(C_{\delta,R}) > 1-\alpha.$$ By the Portmanteau theorem, $\liminf_n \mu_{\mathfrak M_n}(C_{\delta,R}) >  1-\alpha$, so there is some $N$ such that 
$$
\mu_{\mathfrak M_n}(C_{\delta,R})>1-\alpha, \ \ \forall n\geq N.
$$
Furthermore, in light of Claim \ref{thinpartsmall} and \eqref{weirdthickcontained}, we can also assume that
$$\frac{\vol( \{ x \in M_n \ | \ d(x,(M_n)_-)\leq R \} )}{\vol(M_n)} < \alpha, \ \ \forall n\geq N.$$
 Combining the above two estimates, we see that the $\vol / \vol(M_n)$-measure of the set of points $p\in M_n$  such that \emph{both} $d(x,(M_n)_-) > R$ and $(M,p,M) \in C_{\delta,R}$ is at least $(1-2\alpha)$. Now at any such point $p$, we have $p\in N_n$ as well, and the pointed extended mm-spaces $(N_n,p,M_n)$ and $ (M_n,p,M_n)$
are obviously $(\delta,R)$-related. Hence, at any such $p$, we have
\begin{equation} 
\label{close!}	|f(N_n,p,M_n) - f(M_n,p,M_n)| < \alpha.
\end{equation}
Breaking the domains of the following integrals in two, and using the upper bound $m \geq f$ on the piece where \eqref{close!} is not helpful, we see that
$$ \left | \int f \, d\mu_{\mathfrak N_n} - \int f \, d\mu_{\mathfrak M_n} \right | \leq (1-2\alpha) \cdot \alpha + \alpha \cdot 2m, \ \ \forall n\geq N.$$
 So, since $\alpha>0$ was arbitrary and $\int f \, d\mu_{\mathfrak M_n} \to \int f \, d\iota_*(\mu)$, we have that $\int f \, d\mu_{\mathfrak N_n} \to \int f \, d\iota_*(\mu)$ as well, and the claim follows.
\end{proof}

 We now want to apply Corollary \ref{Bowen3} to the sequence $\mu_{\mathfrak N_n}/\vol(M_n)$, in order to say something about normalized Betti numbers. We'll apply it with $r_0=4\xi, r_1=5 \xi, r_2=10\xi$ and $r_3=11\xi$, with $B_n = b_k(M_n)$ and $V_n = \vol(M_n)$. So, let's verify its hypotheses.

For condition (1) of Corollary \ref{Bowen3}, just note that any point $p\in N_n$ is within $\xi$ of  a point $q$ in $(M_n)_{+}$, so $B_{N_n}(p,2\xi)$  contains an embedded $\xi$-ball around $q$, which by nonpositive curvature has volume at least that of a $\xi$-ball in $\BR^d$, see e.g.\  \cite[Theorem~3.101]{gallot1990riemannian}.   Similarly, for condition (2)  the lower  curvature bound implies that any $r$-ball in $N_n$ has volume at most that of an $r$-ball in $\BH^d$, again see \cite[Theorem 3.101]{gallot1990riemannian}. 

 For condition (3),  we need to prove the following.

\begin{lemma}\label {lem3}
If $(S_n ,\rho_n )$  is a sequence of $[10\xi,11\xi]$-weighted $(4\xi,10\xi)$-nets in $N_n$, then
$$\frac{\big | b_k(N_{M_n}(S_n ,\rho_n )) - b_k(M_n) \big | }{\vol(M_n)}\to 0.$$
\end{lemma}

Assuming the lemma,  the hypotheses of  Corollary \ref{Bowen3} are satisfied, so $$B_n/V_n = b_k(M_n) / \vol(M_n)$$  converges,  proving Theorem \ref{npc}.  So, it remains to prove the lemma.

\begin {proof}[Proof of Lemma \ref{lem3}]

Since $\inj : M_n \longrightarrow \BR $ is $2$-lipschitz, we have 
$$\forall x\in S_n \subset N_n :=\overline{((M_n)_{+})_\xi}, \ \ \inj(x) \geq \epsilon-2\xi  > \frac{9 \epsilon}{10 } > \frac{11\epsilon}{20} > 11\xi \geq \rho_n(x). $$
Nonpositive curvature then implies that the balls $B_{\rho(x)}(x)$ are convex, so the Nerve Lemma (c.f.\ \cite[Corollary 4G.3]{Hatcheralgebraic}) says that $N_{S_n} := N_{M_n}(S_n ,\rho_n )$ is homotopy equivalent to $$U_n := \cup_{x\in S_n } B_{\rho_n (x)}(x).$$ So to prove the claim, it suffices to show the following:
\begin{enumerate}
	\item[(a)] If $D_{k,n}$ is the dimension of the image of the map $H_k(U_n ,\BR) \longrightarrow H_k(M_n,\BR)$ induced by inclusion, then $b_k(M_n) = D_{k,n} + o(\vol(M_n)).$
\item[(b)] If $K_{k,n}$ is the dimension of the kernel of the map $H_k(U_n ,\BR) \longrightarrow H_k(M_n,\BR)$ induced by inclusion, then $K_{k,n} = o(\vol(M_n)).$
\end{enumerate}

For (a), apply Mayer--Vietoris to $M_n = \overline{(M_n)_{-}} \cup (M_n)_{+}$, giving the long exact sequence
$$\cdots  \longrightarrow H_k(\overline{(M_n)_{-}};\BR) \oplus H_k((M_n)_{+};\BR) \longrightarrow H_k(M_n;\BR)\longrightarrow H_{k-1}(\partial (M_n)_{-},\BR) \longrightarrow \cdots. $$ By Proposition \ref{thinbettis} and Claim \ref{thinpartsmall}, $b_k( \overline{(M_n)_{-}})$ and $b_{k-1}(\partial (M_n)_{-})$ are  $o(\vol(M_n)),$ so 
$$b_k(M_n) = \dim \mathrm{Im}\Big ( H_k((M_n)_{+};\BR) \longrightarrow H_k(M_n;\BR) \Big ) + o(\vol(M_n)).$$
But the inclusion map $(M_n)_{+} \longrightarrow M_n$ factors through $U \longrightarrow M_n$, so we have $$b_k(M_n) \geq D_{k,n} \geq b_k(M_n)-o(\vol(M_n))$$ as well, proving (a).

For (b), let $T_n \subset M_n \setminus N_n $  be a maximal collection of points such that $$d (s,t)\geq \frac 1 {3} \min\{\inj(s),\inj(t)\}, \ \ \forall s,t \in T_n . $$
Since $\inj$  is continuous, $T_n$ is locally finite.  Moreover,  suppose $x \in M_n \setminus N_n $ and $x \not \in T_n$. By maximality, there must be some $t\in T_n$ with $$d(x,t) \leq \frac 1 {3} \min\{\inj(x),\inj(t)\} \leq \frac 1 {3} \inj(t), $$
 so the open balls of  radius $\rho_n (t) := \frac 1{2} \inj(t)$ around all $t\in T_n$ cover $M_n \setminus N_n$.  Let $N_{S_n \cup T_n }$ be the nerve  complex associated to the cover of $M_n$ by the collection of all  such balls $B_{\rho_n (t)}(t), \ t\in T_n$, \emph {together with  the balls} $B_{\rho_n (x)}(x), \ x \in S_n$. As all these balls are convex, $N_{S_n \cup T_n }$ is homotopy equivalent to $M_n$. In fact, more is true:
\begin{claim}\label {diagram}
 There is a diagram of maps
$$\begin{tikzcd} 
	U_n  \arrow[hookrightarrow]{d} \ar{rr}{\Phi} & & N_{S_n} \arrow[hookrightarrow]{d} \\ M_n & & \arrow{ll}{F} N_{S_n \cup T_n }
\end{tikzcd}
$$
that is commutative up to homotopy, where the vertical maps are the natural inclusions and the horizontal maps $\Phi,F$ are homotopy equivalences.
\end{claim}

The claim  does not assert that  the pairs $(M_n,U_n )$ and $(N_{S_n \cup T_n },N_{S_n})$ are homotopy equivalent, although it is certainly a result along those lines.  We should note that there is at least one `Relative Nerve Lemma' for pairs in the literature, see e.g.\ \cite[Lemma 2.9]{bader2016homology}, but this does not apply in our situation since $U_n \hookrightarrow M_n$  is not a cofibration. One can get around this, but the fix is not particularly pretty, and it is much more direct just to prove the claim above without referencing any citations.

Before proving the claim, let us quickly indicate how to finish the proof of (b). Any point $x\in S_n$ such that $B_{\rho_n (x)}(x)$ intersects a ball $B_{\rho_n (y)}(y)$, where $y\in T_n$, must lie  close to the $\epsilon$-thin part of $M_n$. By Claim \ref{thinpartsmall} the volume of any fixed $R$-neighborhood of $(M_n)_{-}$ is $o(\vol(M_n))$, so this means that there are only $o(\vol(M_n))$-many vertices of $N_{S_n}$ that are adjacent to vertices of $N_{S_n \cup T_n } \setminus N_{S_n}$. So, Mayer--Vietoris  implies that  the kernel of the map $$H_k(N_{S_n} ;\BR) \longrightarrow H_k( N_{S_n \cup T_n }; \BR)$$
 induced by inclusion has rank $o(\vol M_n )$. Therefore, Claim \ref{diagram}  implies that the same is true for the kernel of the map on homology induced by $U_n  \hookrightarrow M_n.$

\begin {proof}[Proof of Claim \ref{diagram}]
Let's review the proof of the Nerve Lemma. For a much more general proof that essentially specializes to the one below, see Hatcher \cite[4G]{Hatcheralgebraic}.

 We start with a Riemannian manifold $X$ and an open cover $\mathcal O$ by small convex balls. If $N$ is the nerve complex of the cover $\mathcal O$, we can define homotopy inverses $$\alpha: X \longrightarrow N, \ \ \beta : N \longrightarrow X$$ as follows. Pick a partition of unity $\{\phi_O \ | \ O \in \mathcal O\}$  subordinate to $\mathcal O$, and define
$$\alpha : X \longrightarrow N, \ \ \ \alpha(p) = \hspace{-2mm} \sum_{\substack{O\in \mathcal O, \  p \in O} }\phi_O(p) \cdot O  \ \in  N.$$ 
Here, the values $\phi_O(p)$ are the barycentric coordinates of $\alpha(p)$, within the simplex of $N$ spanned by those $O$ containing $p$. The map $\beta$ is defined inductively on the $i$-skeleta $BN^i$ of the first barycentric subdivision $BN$ of $N$. Before starting the construction, note that every  vertex $v$ of $BN$ is the barycenter of a simplex of $N$, 
which corresponds to some finite $$\mathcal F_v \subset \mathcal O, \ \ \cap_{O \in \mathcal{F}_v} O \neq \emptyset,$$ 
and if $\Delta$ is a simplex of $BN$, there is one vertex $v(\Delta)$ of $\Delta$ such that $\mathcal F_{v(\Delta)} $ is contained in $\mathcal F_w$ for every other vertex $w$ of $\Delta$. (This $v(\Delta)$ is just the vertex that is the barycenter of the simplex of $N$ with minimal dimension.)
 For $i=0,1,2,\ldots$, we now construct the map $\beta$ on $BN^i$ in such a way that for any $i$-simplex $\Delta$, 
\begin{equation} \label{ittt} \beta(\Delta) \subset \cap_{O \in \mathcal F_{v(\Delta)}} O.	
\end{equation}
If $v$ is a vertex of $BN$, just pick $\beta(v) \in \cap_{O \in F_v} O$ arbitrarily. In general, assuming $\beta$ has been defined on $\partial \Delta$, it follows from the definition of $v(\Delta)$ and \eqref{ittt} that $$\beta(\partial \Delta) \subset \cap_{O \in \mathcal F_{v(\Delta)}} O.$$  This intersection is contractible, so there is some extension of $\beta$ to $\Delta$ satisfying  \eqref{ittt}. The homotopy $\alpha \circ \beta \simeq 1$ is  constructed inductively on the skeleta of $BN$, using the homotopy extension principle at each step.  To see that $\beta \circ \alpha   \simeq 1$, one just notes that if $p \in X$, then $\alpha(p)$  is in some simplex $\Delta$ of $BN$ that has as a vertex some $O \ni p$, so by \eqref{ittt}, $\beta\circ \alpha(p) \in O$.  In other words, $p$ and $\beta\circ \alpha(p)$ are both contained in one of the small convex balls in our cover, so we can just take a straight line homotopy from $\beta \circ \alpha $ to $1$.

 With  the above presentation  of the proof of the Nerve Lemma (which we could not  find a reference for) the claim becomes trivial. Namely, let $\Phi : U_n  \longrightarrow N_{S_n}$  be the map called $\alpha$ above,  where the manifold is $U_n$ and the cover is by the $\rho_n (x)$-balls around $x\in S_n$. Let $F : N_{S_n \cup T_n} \longrightarrow M_n$   be the map called $\beta$ above, where   the manifold is $M_n$ and the cover is by the $\rho_n (x)$-balls around $x\in S_n \cup T_n$.    These are both homotopy equivalences, and just as above the straight-line homotopy connects $F \circ \Phi$ to the inclusion $U_n \longrightarrow M_n$.\end{proof}

 Now that we have proved the claim, the lemma follows. \end {proof}

And so does the theorem. \qed

\bibliographystyle{amsplain}
\bibliography{total}

\end{document}